\documentclass[11pt,reqno]{amsart}
\usepackage{vmargin}
\setmargins{2.5cm}{1.5cm}{16cm}{24cm}{10pt}{1cm}{opt}{2cm}

\usepackage[T1]{fontenc}
\usepackage{fourier}

\usepackage{amssymb}
\usepackage{amsmath}
\usepackage{amsthm}
\usepackage{amsfonts}
\usepackage{enumitem}
\usepackage{color}
\usepackage{hyperref}

\newcommand{\R}{\mathbb{R}}
\newcommand{\A}{\mathcal{A}}
\newcommand{\N}{\mathbb{N}}
\newcommand{\Z}{\mathbb{Z}}
\newcommand{\T}{\mathbb{T}}
\newcommand{\MM}{\mathcal{M}}
\newcommand{\PP}{\mathcal{P}}

\theoremstyle{plain}
\newtheorem{defi}{Definition}[section]
\newtheorem{prop}[defi]{Proposition}
\newtheorem{thm}[defi]{Theorem}
\newtheorem{cor}[defi]{Corollary}
\newtheorem{lema}[defi]{Lemma}
\newtheorem{remark}[defi]{Remark}

\numberwithin{equation}{section}

\begin{document}

\title[Regularity and rates in nonlocal homogenization]{\texorpdfstring{$C^{\sigma+\alpha}$}
{C^(sigma+alpha)} regularity for fractional Hamilton--Jacobi equations,
\texorpdfstring{\\}{ } and rates of convergence in periodic homogenization}

\author[]{Adina Ciomaga}
\address{Adina Ciomaga:
Universit\'e Paris Cit\'e, CNRS, Sorbonne Universit\'e, Laboratoire Jacques-Louis Lions (LJLL), F-75006 Paris, France;
O. Mayer Mathematics Institute, Romanian Academy - Ia\c si Branch, 700506 Ia\c si, Romania.}
\email{\tt adina@math.univ-paris-diderot.fr}

\author[]{Andrei Rodriguez-Paredes}
\address{Andrei Rodriguez-Paredes:
Departamento de Matem\'atica, Facultad de Ciencias F\'isicas y Matem\'aticas, Universidad de Concepci\'on, Avda. Esteban Iturra s/n, Barrio Universitario, Concepci\'on, Chile.}
\email{\tt andreirodriguez@udec.cl}

\author[]{Erwin Topp}
\address{Erwin Topp:
Instituto de Matem\'atica, Universidade Federal do Rio de Janeiro,
Rio de Janeiro -- RJ, 21941-909, Brazil.}
\email{\tt etopp@im.ufrj.br}

\begin{abstract}
We study two related questions for nonlocal equations of Bellman type: the regularity of solutions when the set of controls depends on the state variable, and the periodic homogenization of stationary nonlocal Hamilton--Jacobi equations driven by uniformly elliptic integro-differential operators of order $\sigma\in(0,2)$, whose kernels depend simultaneously on the slow variable, the fast periodic variable and the control. The effective equation is again a nonlocal Bellman equation, but its supremum runs over a set of viscosity Mather measures that varies with the space variable, in a merely H\"older continuous way.
We first develop a $C^{\sigma+\alpha}$ regularity theory for nonlocal Bellman equations with H\"older continuous, state-dependent index sets, extending the Schauder estimates of Caffarelli--Silvestre and Serra, in a general framework independent of homogenization. We then use these estimates as a source of compactness to construct a classical solution of the effective problem, unique among bounded classical solutions. Finally, combining this regularity with discounted approximate correctors, we make the perturbed test function method quantitative, and prove that the whole family of solutions converges to the effective solution, with the explicit rate $\epsilon^{\sigma\alpha/(2+\alpha)}$, where $\alpha$ is the H\"older exponent of the data.
\end{abstract}



\maketitle

\section{Introduction}\label{secIntro}

The main motivation of this paper is the study of homogenization phenomena for stationary nonlocal Hamilton--Jacobi (HJ) equations of the form
\begin{equation}\label{eq:HJ-eps}
	u^\epsilon(x) + \sup_{\theta \in \Theta}
	\big\{ -L^\theta_{x,\frac{x}{\epsilon}} u^\epsilon(x) -\ell^\theta\big(x,\tfrac{x}{\epsilon}\big) \big\}
	= 0 \quad \text{in } \R^N,
\end{equation}
where $\epsilon \in (0,1)$ is the homogenization scale, $\Theta$ is a compact metric space, $\T^N$ is the $N$-dimensional flat torus, $L^\theta_{x,y}$ is a nonlocal diffusion operator (described in detail below),  and $\ell : \R^N \times \T^N \times \Theta \to \R$ is a bounded continuous \emph{cost} function. Here we have written $\ell^\theta(x,y) = \ell(x,y,\theta)$, and will adopt the same notation for every function depending on $\theta$.

The family $(L^\theta_{x,y})_{(x,y,\theta)\in\R^N\times\T^N\times\Theta}$ consists of nonlocal operators in L\'evy form, with oscillating kernels, defined for smooth bounded functions $u : \R^N \to \R$ by
	\begin{equation}\label{eq:Lxy}
		L^\theta_{x,y} u(x) := \mathrm{P.V.} \int_{\R^N} \big(u(x+z) - u(x)\big)\, K^\theta(x,y,z)\, dz,
	\end{equation}
where the kernel $K^\theta : \R^N \times \T^N \times (\R^N \setminus \{0\}) \to \R$ is Borel measurable, symmetric, uniformly elliptic of order $\sigma \in (0,2)$, and homogeneous of degree $-(N+\sigma)$ in $z$, see Section~\ref{secApplications}.

Since the pioneering works~\cite{LPV86, E89, E92}, homogenization of degenerate elliptic HJ equations has been extensively studied. It is well known that, under suitable assumptions on the data promoting ``ergodicity'', the family of solutions $(u^\epsilon)_\epsilon$ of~\eqref{eq:HJ-eps} converges (locally) uniformly in $\R^N$ to a function $u$, and such a limit solves an \emph{effective equation}, properly defined through the averaging process in the passage to the limit. {This picture extends to nonlocal equations, see for instance~\cite{A09, S2010}, and a rate of convergence for $\| u^\epsilon - u \|_\infty$ was obtained in~\cite{RPT23} when the kernels do not depend on the slow variable $x$. In this paper, we provide such a rate when the kernels depend on all the variables, under assumptions that will be precisely presented in Section~\ref{secApplications} below.}

A key feature of~\eqref{eq:HJ-eps} is that the effective equation admits an explicit representation. As noticed in~\cite{RPT23}, the limit $u$ is expected to solve
	\begin{equation}\label{eq:HJ-bar}\tag{$\overline{HJ}$}
		u(x) + \sup_{\mu \in \mathcal A(x)} \big\{ -\overline L^\mu_x u(x) - \overline \ell^\mu(x) \big\} = 0 \quad \text{in } \R^N,
	\end{equation}
where, for each $x \in \R^N$, $\mathcal A(x)$ is a nonempty subset of the space of probability measures on the product space $\T^N \times \Theta$, and $\overline L^\mu_x, \overline \ell^\mu$ are averaged (with respect to $\mu$) fractional elliptic operators and running costs, respectively, namely
\[
	\overline L^\mu_x u(x) := \int_{\T^N\times\Theta} L^\theta_{x,y} u(x)\, d\mu(y,\theta), \qquad
	\overline \ell^\mu(x) := \int_{\T^N\times\Theta} \ell^\theta(x,y)\, d\mu(y,\theta).
\]
See~\eqref{defA} for the definition of $\mathcal A(x)$ and~\eqref{eq:mean-f-mu} for a general definition of the averaging. The measures in $\mathcal A(x)$ are referred to as \emph{viscosity Mather measures} in Ishii, Mitake and Tran~\cite{IMT17, IMT17b}, where the authors provide representation formulas for the ergodic constant associated with second-order degenerate elliptic HJ equations. For first-order PDEs linked to optimal control of deterministic dynamical systems, representation formulas for homogenization and/or singular perturbation problems are given in terms of \emph{occupational measures}, see~\cite{AB01, Terrone2011}. Representation formulas for second-order equations can also be found in~\cite{BK23}.

To obtain a rate of convergence, we follow the approach of Camilli and Marchi~\cite{CM09}, adapted to the nonlocal setting in~\cite{RPT23}, which makes the perturbed test function method quantitative by using the effective solution $u$ itself as a test function; the regularity of $u$ is therefore the key ingredient of the argument. Implementing this strategy in the present setting raises two difficulties.
First, the map $x \mapsto \mathcal A(x)$ is defined implicitly through the associated cell problem, and this state dependence prevents a direct use of the available nonlocal higher-order Schauder estimates, such as those of~\cite{CS11b, Serra15}, where the index set in the Bellman formulation does not depend on $x$. Second, we do not know whether~\eqref{eq:HJ-bar} satisfies a comparison principle for viscosity solutions, so that neither its well-posedness nor the uniqueness of the limit of $(u^\epsilon)_\epsilon$, and hence the convergence of the whole family, can be obtained by standard arguments.
These difficulties arise from the novelties of~\eqref{eq:HJ-eps} compared with previous works, as the kernels defining the nonlocal operators $L_{x,y}^\theta$ are allowed to depend on all the variables: the \emph{slow variable} $x$, the \emph{fast variable} $y = x/\epsilon$ and the \emph{control} $\theta$.

We overcome both difficulties by means of higher-order regularity estimates. We develop a $C^{\sigma+\alpha}$ regularity theory for nonlocal, uniformly elliptic equations in Bellman form, with an index set depending on the state variable $x$, of which~\eqref{eq:HJ-bar} is a particular case. This is our first main result, Theorem~\ref{thm1}, which extends the estimates of Serra~\cite{Serra15}, where the index set does not depend on $x$. The key assumption is the H\"older continuity of the set-valued map $x \mapsto \mathcal A(x)$. The proof adapts the blow-up and compactness argument of~\cite{Serra15}, based on a Liouville-type theorem: the variation of the index set creates an additional error term at every scale, and the H\"older continuity of $x \mapsto \mathcal A(x)$ ensures that it vanishes in the blow-up limit. Since we believe these results to be of independent interest, we state them in a general framework, independent of the homogenization problem.

These estimates allow us to construct a classical solution of the effective problem, unique among bounded classical solutions, without relying on a comparison principle. Combined with the approach of~\cite{CM09}, this leads to our second main result, Theorem~\ref{thm:rates-homog}, the rate of convergence
\[
	\|u^\epsilon-u\|_{L^\infty(\R^N)} \le C\epsilon^{\alpha_*}, \qquad
	\alpha_* =\frac{\sigma\alpha}{2+\alpha},
\]
for all $\epsilon \in (0,1)$. In particular, the whole family $(u^\epsilon)_\epsilon$ converges to the unique classical solution of~\eqref{eq:HJ-bar}.
To the best of our knowledge, this is the first rate of convergence in periodic homogenization of nonlinear nonlocal equations whose kernels depend jointly on the slow variable, the fast variable and the control.
When the kernels do not depend on $x$, as in~\cite{RPT23}, our rate also improves the exponent obtained there. We do not know whether $\alpha_*$ is optimal. This result relies on the H\"older continuity of the data in the slow and fast variables $(x,y)$, uniformly in the control, and on their continuity in $\theta$. We do not pursue the analysis under weaker assumptions, even for qualitative homogenization.

\subsection{Related literature.} Periodic homogenization of HJ and fully nonlinear elliptic equations goes back to Lions, Papanicolaou and Varadhan~\cite{LPV86} for first-order equations, and to Evans~\cite{E89, E92}, who introduced the perturbed test function method for second-order problems; see also Alvarez and Bardi~\cite{AB01} for connections with singular perturbations.

Rates of convergence were obtained by Capuzzo-Dolcetta and Ishii~\cite{CDI01} for first-order equations, by Camilli and Marchi~\cite{CM09} for fully nonlinear, uniformly elliptic ones, and have been sharpened in the convex first-order setting by Mitake, Tran and Yu~\cite{MTY19}.
In random media, qualitative homogenization was established by Souganidis~\cite{Souganidis99} and Rezakhanlou and Tarver~\cite{RT00} for first-order equations, by Lions and Souganidis~\cite{LS05} and Kosygina, Rezakhanlou and Varadhan~\cite{KRV06} for viscous ones, and by Caffarelli, Souganidis and Wang~\cite{CSW05} for fully nonlinear, uniformly elliptic equations. Quantitative results followed, see Caffarelli and Souganidis~\cite{CS10}, Armstrong, Cardaliaguet and Souganidis~\cite{ACS14}, Armstrong and Smart~\cite{AS14}, Armstrong and Lin~\cite{AL17}, and the monograph~\cite{AKM19} for divergence-form equations.

For nonlocal equations, periodic homogenization of nonlinear integro-differential equations was initiated by Arisawa~\cite{A09} and Schwab~\cite{S2010}, and extended to nonlocal HJ equations with gradient terms in~\cite{BCT20, CGT22}, including the critical case $\sigma = 1$; rates of convergence were obtained in~\cite{RPT23}. Stochastic homogenization of integro-differential equations was established by Schwab~\cite{Schwab13}, building on the approach of~\cite{CSW05}. Linear nonlocal operators have been homogenized by variational, probabilistic and operator-theoretic methods, and the nature of the limit depends on the tails of the kernel. For symmetric L\'evy-type operators of stable type, the limit is a nonlocal operator, see Kassmann, Piatnitski and Zhizhina~\cite{KPZ19} for periodically oscillating coefficients and Chen, Chen, Kumagai and Wang~\cite{CCKW21} for stationary ergodic media. Quantitative results for such linear operators, including operator estimates with correctors and higher-order convergence rates, have been obtained very recently, see~\cite{PSSZ26, CCKW26} and the references therein. In all the nonlocal works above, the kernel may depend on the fast variable and the control, but not on the slow variable: the dependence on $x$ enters only through lower-order terms, such as a drift or the running cost.

The regularity theory for fully nonlinear, uniformly elliptic nonlocal equations was developed by Caffarelli and Silvestre, who established the Harnack inequality and $C^{1,\alpha}$ estimates~\cite{CS09}, then perturbative regularity results for equations with variable coefficients~\cite{CS11}, and a nonlocal Evans--Krylov theorem~\cite{CS11b}.
For viscosity solutions of nonlinear integro-differential equations satisfying a weaker ellipticity condition, H\"older estimates were obtained by Barles, Chasseigne and Imbert~\cite{BCI11}, and Lipschitz estimates, for mixed local and nonlocal equations, by Barles, Chasseigne, Ciomaga and Imbert~\cite{BCCI12}.
Interior $C^{1,\alpha}$ estimates for kernels without regularity assumptions were obtained by Kriventsov~\cite{Kriventsov13}, and $C^{\sigma+\alpha}$ estimates for concave equations with rough kernels, H\"older continuous in $x$, by Serra~\cite{Serra15}. Schauder and Dini-type estimates for nonlocal fully nonlinear equations were established by Jin and Xiong~\cite{JX16} and Mou~\cite{Mou16}. We refer to the monograph of Fern\'andez-Real and Ros-Oton~\cite{FRRO24} for an overview of the theory. In all these results, whenever the operator is written in Bellman or Isaacs form, the index set does not depend on $x$.

\smallskip
\noindent\textbf{Organization of the paper.}
Section~\ref{secAssumptions} collects the notation and the standing assumptions, and states the main results: the regularity estimate of Theorem~\ref{thm1} and its version for proper equations, Corollary~\ref{cor:proper}, in an abstract setting; and, in Section~\ref{secApplications}, the definition of the sets $\mathcal A(x)$ and the rate of convergence of Theorem~\ref{thm:rates-homog}.
 Section~\ref{sec-reg} proves Theorem~\ref{thm1}.
 Section~\ref{sec-existence} constructs the classical solution of~\eqref{eq:HJ-bar}.
 Section~\ref{sec-convergence} proves Theorem~\ref{thm:rates-homog}: we first establish the H\"older regularity of $x \mapsto \mathcal A(x)$ and the properties of the approximate correctors, and then complete the quantitative perturbed test function argument. The Appendix collects elementary estimates on second differences of H\"older functions.

\section{Assumptions and main results}
\label{secAssumptions}

\subsection{Notations}
We introduce here the main notations used throughout the paper.
\smallskip

\noindent \textbf{Euclidean space.}
$\R^N$ denotes Euclidean space, $N \ge 1$. $\T^N := \R^N/\Z^N$ denotes the $N$-dimensional flat torus, identified throughout with $[0,1)^N$. $\Theta$ is a fixed nonempty set of \emph{controls}, assumed to be a compact metric space.  For $x \in \R^N$ and $r>0$, $B_r(x)$ denotes the open Euclidean ball of radius $r$ centered at $x$. We write $B_r := B_r(0)$ for the ball centered at the origin, and $ B$ for the open unit ball; $\overline{B_r(x)}$, $\overline{B_r}$ denote their closures.

\smallskip
\noindent \textbf{Characteristic functions.} For a set $A$ contained in a given ambient space, $\chi_A$ denotes its characteristic (indicator) function, $\chi_A(x) = 1$ if $x \in A$ and $\chi_A(x) = 0$ otherwise.

\smallskip
\noindent \textbf{Probability measures on topological spaces.}
Let $X$ be a topological space; in the homogenization problem, $X := \T^N \times \Theta$, with the product topology. We denote by $C_b(X)$, resp.\ $C_b(B \times X)$, the bounded continuous functions on $X$, resp.\ on $B \times X$ with the product topology, and by $\PP(X)$ the set of Radon probability measures on $X$.

For $\mu \in \PP(X)$ and a $\mu$-integrable function $\phi : X \to \R$, we write $\overline\phi^\mu := \int_X \phi\, d\mu$ and, more generally, for $f : X \times Z \to \R$ with $f(\cdot,z)$ $\mu$-integrable for every $z \in Z$,
\begin{equation}\label{eq:mean-f-mu}
	\overline f^\mu(z) := \int_X f(\xi, z)\, d\mu(\xi), \qquad z \in Z.
\end{equation}
Given a family of nonempty sets $\mathcal A(x) \subseteq \PP(X)$, we denote the set of averages of $\phi$ by
\begin{equation}\label{meanphi}
	\langle \phi \rangle_x := \{ \overline \phi^\mu : \mu \in \mathcal A (x) \} \subseteq \R.
\end{equation}

\smallskip
\noindent \textbf{The ellipticity class $\mathcal L_0$.} For $\sigma \in (0,2)$ and $0 < \Lambda_0 \le \Lambda < \infty$, we denote by $\mathcal L_0 = \mathcal L_0(\sigma,\Lambda_0,\Lambda)$ the class of \emph{symmetric, uniformly elliptic kernels of order $\sigma$}, that is, Borel measurable functions $K : \R^N\setminus\{0\} \to \R$ such that, for a.e.\ $z \neq 0$,
\begin{equation}\label{eq:L0-class}
	K(-z) = K(z), \qquad \frac{\Lambda_0}{|z|^{N+\sigma}} \le K(z) \le \frac{\Lambda}{|z|^{N+\sigma}}.
\end{equation}
For $K \in \mathcal L_0$, we write
\[
	L_K u(x) := \mathrm{P.V.} \int_{\R^N} \big(u(x+z)-u(x)\big)K(z)\,dz = \frac12 \int_{\R^N} \delta^2u(x,z)K(z)\,dz,
\]
with $\delta^2 u(x,z) := u(x+z) + u(x-z) - 2u(x).$
This is well defined if $u \in C^{\sigma+\epsilon}(B_r(x))$ for some $r, \epsilon > 0$, which controls the singularity at $z = 0$, and $u \in L^1_\sigma(\R^N)$, that is, $\int_{\R^N} |u(y)|(1+|y|)^{-N-\sigma}\, dy < \infty$, which controls the tail; see~\cite{CS09}. We refer to~\cite{BI08} for the notion of viscosity solution of nonlocal equations.

\smallskip
\noindent \textbf{Pucci extremal operators.} For the ellipticity class $\mathcal L_0 = \mathcal L_0(\sigma,\Lambda_0,\Lambda)$, we denote by $\MM^\pm_{\mathcal L_0}$ the associated Pucci operators, defined for $u:\R^N\to\R$ and $x\in\R^N$ by
\[
	\MM^+_{\mathcal L_0} u(x) := \sup_{K\in\mathcal L_0} L_K u(x), \qquad
	\MM^-_{\mathcal L_0} u(x) := \inf_{K\in\mathcal L_0} L_K u(x),
\]
so that
\[
\MM^-_{\mathcal L_0} u(x) \ \le \ L_{K} u(x) \ \le \ \MM^+_{\mathcal L_0} u(x) \qquad \text{for every } K \in \mathcal L_0.
\]

\smallskip
\noindent \textbf{H\"older spaces.} For $\alpha>0$ with $\alpha\notin\N$, write $\alpha= k+\gamma$ with
$k=\lfloor\alpha\rfloor\in\N_0 := \N \cup \{0\}$ and $\gamma\in(0,1)$. For $\Omega\subseteq\R^N$ and $u\in C^k(\Omega)$, we denote
\[
	[u]_{C^\alpha(\Omega)} := \max_{|l|=k}\ \sup_{x\ne y\in\Omega} \frac{|D^lu(x)-D^lu(y)|}{|x-y|^\gamma}, \qquad
	\|u\|_{C^\alpha(\Omega)} := \sum_{|l|\le k}\|D^lu\|_{L^\infty(\Omega)} + [u]_{C^\alpha(\Omega)}.
\]
In particular, for $\alpha\in(0,1)$ and $\Omega=\R^N$, this gives the usual global H\"older norm
$$
\|u\|_{C^\alpha(\R^N)} = \|u\|_{L^\infty(\R^N)} + \sup_{x\ne y \in \R^N} \frac{|u(x)-u(y)|}{|x-y|^\alpha}.
$$
We write $u\in C^\alpha(\Omega)$ (resp. $u \in C^\alpha_b(\Omega)$) whenever $u\in C^k(\Omega)$ and $[u]_{C^\alpha(\Omega)}<\infty$ (resp. $\|u\|_{C^\alpha(\Omega)}<\infty$). For $k\in\N_0$, $C^k(\Omega)$ and $C^k_b(\Omega)$ denote the classical spaces of $k$ times continuously differentiable functions (resp. with bounded derivatives up to order $k$). For $k\in\N_0$, we also write $[u]_{C^k(\Omega)} := \max_{|l|=k}\|D^lu\|_{L^\infty(\Omega)}$. Finally, if $\Omega$ is either $\R^N$ or a bounded convex domain, then, by interpolation, for all $u \in C^\alpha_b(\Omega)$,
\[
\|u\|_{C^\alpha(\Omega)} \le C\big([u]_{C^\alpha(\Omega)} + \|u\|_{L^\infty(\Omega)}\big).
\]

\subsection{Regularity estimates for effective problems.}
We develop, independently of any homogenization procedure, a self-contained interior regularity theory for a class of fully nonlinear, nonlocal Bellman equations in which the supremum is taken over point-dependent sets of probability measures $\A(x)\subseteq\PP(X)$:
\begin{equation}\label{eq:HJ-eff-sup}
	\sup_{\mu \in \mathcal A(x)} \big\{ -\overline L^\mu_{x} u(x) - \overline \ell^\mu(x) \big\} = 0 \quad \text{in } B,
\end{equation}
where $\overline L^\mu_x$ and $\overline\ell^\mu$ are as in~\eqref{eq:HJ-bar} and~\eqref{eq:mean-f-mu}.

The main result of this subsection, Theorem~\ref{thm1} below, gives interior $C^{\sigma+\alpha}$ estimates for viscosity solutions of~\eqref{eq:HJ-eff-sup}, under structural assumptions general enough to cover operators obtained by averaging the diffusion and the running cost over an arbitrary index set $X$.

The assumptions come in three groups, each attached to one of the objects appearing in~\eqref{eq:HJ-eff-sup}: the kernel $K$ defining $\overline L^\mu_x$, the family of measure sets $\{\mathcal A(x)\}_{x \in B}$, and the running cost $\ell$.

\subsection*{Assumptions}
Let $X$ be a topological space, and let $E \subseteq C(X)$ be a linear subspace endowed with a norm $\| \cdot \|_*$ such that $(E, \| \cdot \|_*)$ is a Banach space and $\chi_X \in E$. We consider a Borel measurable function
\[
K: B \times (\R^N \setminus \{ 0 \}) \times X \to \R,
\]
and a running cost $\ell \in C_b(B \times X)$.

We first collect the assumptions on the kernel $K$.

\begin{enumerate}[start=0,label=\textup{(K\arabic*)},ref=\textup{(K\arabic*)}]

\item \label{K0} \emph{Ellipticity.} For all $x \in B$ and $\xi \in X$,
	$K(x, \cdot, \xi) \in \mathcal L_0(\sigma, \Lambda_0, \Lambda)$.

\item \label{K1} \emph{H\"older regularity of the averaged kernel.}
	There exist a constant $K_0 > 0$ and $\alpha\in (0,1)$ such that, for all $\mu \in \PP(X)$, $x, x' \in B$ and $r > 0$,
	\[
	\int_{B_{2r} \setminus B_r} |\overline K^\mu (x, z) - \overline K^\mu (x', z)|\, dz
	\leq K_0 \, \frac{|x - x'|^\alpha}{r^\sigma}.
	\]

\item \label{K2} \emph{Target space $E$ and compatibility.} There exists a constant $K^*>0$ such that, for all $x \in B$ and $z\in\R^N\setminus \{0\}$, we have $K(x, z, \cdot) \in E$ and
	\[
	\|K(x,z, \cdot)\|_* \leq K^* |z|^{-(N+ \sigma)}.
	\]
	Moreover, for all $x\in B$ and all $u \in L^1_\sigma(\R^N)\cap C^{\sigma +\epsilon}(B_r(x))$,
	for some $r>0$ and $\epsilon > 0$, we have
		\begin{equation*}
		    \zeta := \frac12 \int_{\R^N} \delta^2 u(x,z) \, K(x, z, \cdot) \, dz \in E
		\end{equation*}
	and
		\[
		\| \zeta \|_* \leq \frac12 \int_{\R^N} |\delta^2 u(x,z)| \, \| K(x, z, \cdot)\|_* \, dz.
		\]
\end{enumerate}

Assumption~\ref{K2} holds, for instance, whenever the $E$-valued map $z \mapsto \min\{1, |z|^{\sigma+\epsilon}\}\, K(x,z,\cdot)$ is Bochner integrable on $\R^N$ for every $x \in B$; see Yosida~\cite[Section V.5]{Yosida95}. This is the only place where we use that $E$ is a Banach space, rather than merely a normed vector space. The role of $E$ in the applications will become clear in Subsection~\ref{secApplications}, see~\eqref{defE}.

\smallskip
Next, we turn to the family of measure sets $\{\mathcal A(x)\}_{x \in B}$. Rather than imposing conditions on the individual measures, we require the map $x \mapsto \mathcal A(x)$ itself to be H\"older continuous. We measure its variation through the Hausdorff distance between the sets $\langle \phi \rangle_x$ obtained by averaging a test function $\phi \in E$ against the measures in $\mathcal A(x)$, uniformly in $\phi$.

\begin{enumerate}[start=1,label=\textup{(A\arabic*)},ref=\textup{(A\arabic*)}]
\item \label{A1} \emph{H\"older regularity of $x \mapsto \A(x)$.}
	There exist a constant $A_0 \geq 0$ and $\alpha \in (0,1)$ such that,
	for all $x, x' \in B$ and $\phi \in E$,
	\[
	d_H(\langle \phi \rangle_x, \langle \phi \rangle_{x'}) \leq A_0 \| \phi \|_* |x - x'|^\alpha,
	\]
	where $\langle \phi \rangle_x$ is defined in~\eqref{meanphi},
	and $d_H$ denotes the Hausdorff distance between nonempty bounded subsets of $\R$:
	\( d_H(S,T) = \max \big\{ \sup_{s \in S} \mathrm{dist}(s, T), \ \sup_{t \in T} \mathrm{dist}(t, S) \big\}.\)
\end{enumerate}

\smallskip
Finally, we state the assumption on the running cost $\ell$.
\begin{enumerate}[start=1,label=\textup{(L\arabic*)},ref=\textup{(L\arabic*)}]
\item \label{L1} \emph{Regularity of the cost.}
	For all $x \in B$, $\ell(x, \cdot) \in E$, and there exists a constant $C_0 > 0$ such that
	\[
	\sup_{x\in B} \| \ell(x, \cdot) \|_* + \sup_{\mu \in \mathcal A} \big[ \overline \ell^\mu \big]_{C^\alpha(B)} \leq C_0,
	\qquad \text{where } \mathcal A := \bigcup_{x \in B} \mathcal A(x).
	\]
\end{enumerate}

With assumptions~\ref{K0}--\ref{K2}, \ref{A1} and \ref{L1} in place, we can now state our first main result. Since $\sup_{\mu}\{-a_\mu\} = -\inf_{\mu} a_\mu$, equation~\eqref{eq:HJ-eff-sup} can be equivalently written in $\inf$-form, which is the one we use from now on, following~\cite{Serra15}.
\begin{thm}[$C^{\sigma+\alpha}$ estimate]\label{thm1}
Let $\sigma\in(0,2)$. There exists $\bar\alpha\in(0,1)$ such that the following holds.
For all $0 < \alpha < \bar \alpha$ such that $\sigma + \alpha\not\in\N$,
there exists a constant $C >0$ depending on $N, \sigma, \Lambda_0, \Lambda, K_0, K^*, A_0$ and $\alpha$ such that,
if $u \in C^{\sigma+\alpha}(B)\cap C^\alpha_b(\R^N)$ is a viscosity solution of
\begin{equation}\label{eq:effHJ-x}
	 \inf_{\mu \in \A(x)} \Big\{ \overline L_x^\mu u(x) + \overline \ell^\mu(x) \Big\} = 0
	\quad \text{in } B,
\end{equation}
where the family of kernels $\big(K(x,\cdot,\xi)\big)_{x \in B, \xi \in X}$ satisfies \ref{K0}--\ref{K2},
the family of measure sets $(\A(x))_{x \in B}$ satisfies \ref{A1}, and the running cost $\ell$ satisfies \ref{L1},
with this exponent $\alpha$, then
\[
	\| u\|_{C^{\sigma + \alpha}(B_{1/2})} \leq C
	\Big( \| u \|_{C^\alpha(\R^N)} + C_0 \Big),
\]
where $C_0>0$ is the constant in~\ref{L1}.
\end{thm}

The proof of this theorem follows the lines of~\cite{Serra15}, combining a blow-up analysis with a Liouville-type theorem. When $\mathcal A(x)$ does not depend on $x$, assumption~\ref{A1} holds with $A_0 = 0$, and the result reduces to~\cite[Theorem 1.1]{Serra15}, with $E = C_b(X)$ and $\|\cdot\|_* = \|\cdot\|_\infty$.

The same estimate holds for equations with a zeroth-order term, such as~\eqref{eq:HJ-bar}.
\begin{cor}[$C^{\sigma+\alpha}$ estimate for proper equations]\label{cor:proper}
Under the assumptions of Theorem~\ref{thm1}, the same estimate holds, with a constant depending in addition on $\|\chi_X\|_*$, if $u$ is instead a viscosity solution of
\begin{equation}\label{eq:properHJ-x}
	u(x) - \inf_{\mu \in \A(x)} \Big\{ \overline L_x^\mu u(x) + \overline \ell^\mu(x) \Big\} = 0
	\quad \text{in } B.
\end{equation}
\end{cor}

We conclude this subsection with a remark on first-order terms. We expect the methods developed here to extend to equations with drift terms when $\sigma > 1$. The critical case $\sigma = 1$ appears substantially more delicate, since diffusion and drift then act at the same scale, and the argument would require a Liouville-type theorem for the resulting operator, which, to our knowledge, is not currently available.

\subsection{Application to rates of convergence in periodic homogenization.}\label{secApplications}

We now apply the abstract theory of the previous subsection to its motivating example: the periodic homogenization of the nonlocal {HJ} equation~\eqref{eq:HJ-eps}. More precisely, we quantify the rate at which its solutions $u^\epsilon$ converge, as $\epsilon \to 0$, to the solution $u$ of the effective problem~\eqref{eq:HJ-bar}.

\smallskip

 In Theorem~\ref{thm:rates-homog} and more specifically in what respects~\eqref{eq:HJ-eps}, we will always consider
 $$
 X := \T^N \times \Theta,
 $$
 and the kernel $K$ depends on the fast variable $y = x/\epsilon \in \T^N$, in addition to the slow variable $x$ and the control $\theta \in \Theta$. The main result of this subsection, Theorem~\ref{thm:rates-homog} below, asserts that, under H\"older regularity of $K$ and $\ell$ in $(x,y)$ and mere continuity in $\theta$, the family $\{u^\epsilon\}_\epsilon$ converges uniformly to $u$ at an algebraic rate in $\epsilon$.

\subsection*{Assumptions}
We consider a Borel measurable kernel and a running cost
\begin{equation}\label{defK}
\begin{aligned}
	& { K: \R^N \times \T^N \times (\R^N \setminus \{ 0 \}) \times \Theta \to \R, \quad (x,y,z,\theta) \mapsto K^\theta(x,y,z)}, \\
	& \ell: \R^N \times \T^N \times \Theta \to \R, \quad (x,y,\theta) \mapsto \ell^\theta(x,y).
\end{aligned}
\end{equation}
In the notation of the previous subsection, $K(x,z,\xi) := K^\theta(x,y,z)$ for $\xi = (y,\theta) \in X$. We assume that $K$ satisfies the following:

\begin{enumerate}[label=\textup{(H0)},ref=\textup{(H0)}]
\item \label{H0} \emph{Ellipticity.} For all $x \in \R^N$, $y\in \T^N$ and $\theta \in \Theta$,
	$K^\theta(x, y, \cdot) \in \mathcal L_0(\sigma, \Lambda_0, \Lambda)$.
\end{enumerate}

\begin{enumerate}[label=\textup{(H1)},ref=\textup{(H1)}]
\item \label{H1} \emph{Continuity in $\theta$.}
	For all $x \in \R^N$, $z \in \R^N \setminus \{0\}$ and $y \in \T^N$, the map $\theta \mapsto { K^\theta(x,y,z)}$ is continuous.
\end{enumerate}

\begin{enumerate}[label=\textup{(H2)},ref=\textup{(H2)}]
\item \label{H2} \emph{H\"older regularity in $(x,y)$.}
	There exist a constant $K_0 > 0$ and an exponent $\alpha \in (0,1)$ such that,
	for all $z \in \R^N\setminus\{0\}$ and $\theta \in \Theta$,
	\[
	\big[ { K^\theta(\cdot, \cdot, z)} \big]_{C^\alpha(\R^N \times \T^N)} \leq K_0 \, |z|^{-(N+\sigma)}.
	\]
\end{enumerate}

\begin{enumerate}[label=\textup{(H3)},ref=\textup{(H3)}]
\item \label{H3} \emph{Homogeneity.}
	For all $x \in \R^N$, $z \in \R^N \setminus \{0\}$, $y \in \T^N$, $\theta \in \Theta$ and $t>0$,
	\[
	{ K^\theta(x, y, t z) \ = \ t^{-(N+\sigma)} K^\theta(x, y, z)}.
	\]
\end{enumerate}

\noindent We assume the running cost satisfies the following assumptions, with the same exponent $\alpha$ as in~\ref{H2}.

\begin{enumerate}[label=\textup{(H4)},ref=\textup{(H4)}]
\item \label{H4} \emph{Continuity in $\theta$.}
	For all $x \in \R^N$ and $y \in \T^N$, the map $\theta \mapsto \ell^\theta(x,y)$ is continuous.
\end{enumerate}

\begin{enumerate}[label=\textup{(H5)},ref=\textup{(H5)}]
\item \label{H5} \emph{H\"older regularity in $(x,y)$.}
	There exists a constant $C_0 > 0$ such that, for all $\theta \in \Theta$,
	\[
	\| \ell^\theta \|_{C^\alpha(\R^N \times \T^N)} \leq C_0.
	\]
\end{enumerate}

To fit the abstract framework of the previous subsection, we take $E$ to be the space of continuous functions on $X$ that are $\alpha$-H\"older continuous in $y$, uniformly in $\theta$:
\begin{equation}\label{defE}
	E := \Big\{ \phi \in C(X) : \ \| \phi \|_* := \sup_{\theta \in \Theta} \| \phi(\cdot,\theta) \|_{C^\alpha(\T^N)} <\infty \Big\}.
\end{equation}
Then $(E,\|\cdot\|_*)$ is a Banach space, with $\chi_X \in E$ and $\|\chi_X\|_* = 1$. Moreover,~\ref{H0}-\ref{H2} imply~\ref{K0}-\ref{K2}, and~\ref{H4} and~\ref{H5} imply~\ref{L1}: it suffices to average the pointwise bounds against $\mu \in \PP(X)$, the continuity on $X$ of $K(x,z,\cdot)$ and $\ell(x,\cdot)$ following from continuity in $\theta$ and H\"older continuity in $y$, uniformly in $\theta$.

\smallskip
\noindent\textbf{The effective problem and the sets $\A(x)$.}
For $x \in \R^N$, let $\mathcal G(x)$ be the set of data for which the cell problem frozen at $x$ admits a subsolution,
\begin{equation}\label{defG}
	\mathcal G(x) := \Big\{ \phi \in C(X) : \ \sup_{\theta \in \Theta} \big\{ -L^\theta_{x,y} w(y) - \phi^\theta(y) \big\} \le 0 \ \text{ in } \T^N \ \text{for some } w \in C(\T^N) \Big\},
\end{equation}
where $L^\theta_{x,y}$ acts on the periodic variable $y$, with $x$ frozen, and the inequality holds in the viscosity sense. Since the left-hand side is convex and positively homogeneous in $(w,\phi)$, $\mathcal G(x)$ is a convex cone, and we define $\A(x)$ as the set of probability measures in its polar cone,
\begin{equation}\label{defA}
	\A(x) := \Big\{ \mu \in \PP(X) : \ \int_X \phi \, d\mu \ge 0 \ \text{ for all } \phi \in \mathcal G(x) \Big\}.
\end{equation}

Then, the effective Hamiltonian associated to problem~\eqref{eq:HJ-eps} takes the form~\eqref{eq:HJ-bar}, where, for each $\mu \in \mathcal A(x)$ as in~\eqref{defA}, $\overline K^\mu$, $\overline \ell^\mu$ the averages of $K$ and $\ell$ against $\mu$ as in~\eqref{eq:mean-f-mu}, with $\xi = (y,\theta) \in \T^N \times \Theta$.

We will show in Lemma~\ref{lem:A-reg} that~\ref{A1} follows from~\ref{H0}-\ref{H2}.

We are now ready to state our rate of convergence result, which in particular yields the homogenization of~\eqref{eq:HJ-eps}.
\begin{thm}[Rate of convergence]\label{thm:rates-homog}
Let $\sigma \in (0,2)$, let $\bar\alpha \in (0,1)$ be as in Theorem~\ref{thm1}, and let $\alpha \in (0,\bar\alpha)$ with $\sigma+\alpha\notin\N$. Assume that $K$ satisfies ~\ref{H0}-\ref{H3}, and that $\ell$ satisfies~\ref{H4} and~\ref{H5}. Then~\eqref{eq:HJ-bar} admits a unique classical solution $u \in C^{\sigma+\alpha}_b(\R^N)$, and there exists $C > 0$, depending only on $N$, $\sigma$, $\Lambda_0$, $\Lambda$, $K_0$, $C_0$ and $\alpha$, such that, for all $\epsilon \in (0,1)$,
\[
	\| u^\epsilon - u \|_{L^\infty(\R^N)} \le C \, \epsilon^{\alpha_*},
	\qquad \alpha_* := \frac{\sigma\alpha}{2+\alpha} \in (0,\alpha),
\]
where $u^\epsilon$ is the unique solution of~\eqref{eq:HJ-eps}.
\end{thm}

The proof follows the strategy of Camilli and Marchi~\cite{CM09}, adapted to the nonlocal setting as in~\cite{RPT23}. Its key ingredient is the $C^{\sigma+\alpha}$ regularity of the solution of~\eqref{eq:HJ-bar}, given by Proposition~\ref{prop:existence-Csigma}, which applies since $\A(x)$ satisfies~\ref{A1} by Lemma~\ref{lem:A-reg}. Note that we do not use any comparison principle for viscosity solutions of~\eqref{eq:HJ-bar}, which, to our knowledge, is not available when $x \mapsto \A(x)$ is merely continuous: the equation holds pointwise for the classical solution $u$.

\bigskip
\section{A priori \texorpdfstring{$C^{\sigma + \alpha}$}{C^{sigma + alpha}}  estimates for effective-type problems}
\label{sec-reg}

This section is devoted to the proof of Theorem~\ref{thm1}, that is, interior $C^{\sigma+\alpha}$ estimates for viscosity solutions of~\eqref{eq:effHJ-x}. Our arguments build on those of Serra~\cite{Serra15}, and proceed in three steps.

We first treat kernels that do not depend on $x$, for which we adapt the blow-up and compactness argument of~\cite{Serra15}, based on a Liouville-type theorem. The only, but essential, new difficulty at this stage is the dependence of the measure sets $\mathcal A(x)$ on $x$, which must be tracked along the blow-up sequence and is controlled by~\ref{A1}. We then allow the kernels to depend on $x$, freezing the coefficients at a point and controlling the resulting error by~\ref{K1}. Finally, we combine these two estimates, through a scaling argument and a weighted H\"older seminorm, to conclude the proof of Theorem~\ref{thm1}.

\subsection{Kernels independent of $x$}
We first treat the case of kernels with frozen coefficients,
$$
K: (\R^N \setminus \{ 0 \}) \times X \to \R,
$$
for which assumptions~\ref{K0}--\ref{K2} specialize in the obvious way, with $x$ simply dropped throughout. In particular, since $\overline K^\mu$ no longer depends on $x$, assumption~\ref{K1} holds trivially.

Averaging preserves the symmetry and the two-sided bounds in~\eqref{eq:L0-class}, so that $\overline K^\mu \in \mathcal L_0(\sigma,\Lambda_0,\Lambda)$ for every $\mu\in\PP(X)$. Moreover, if $u \in L^1_\sigma(\R^N) \cap C^{\sigma+\epsilon}(B_r(x))$ for some $r,\epsilon>0$, Fubini--Tonelli's theorem gives
\begin{equation}\label{eq:L-mu-sym}
\overline L^\mu u(x) := \int_X L_{K(\cdot,\xi)} u(x)\, d\mu(\xi) = \frac{1}{2}\int_{\R^N} \delta^2 u(x,z)\, \overline K^\mu (z)\, dz = L_{\overline K^\mu} u(x).
\end{equation}

The key tool is the following Liouville-type theorem for operators in the class $\mathcal L_0(\sigma,\Lambda_0,\Lambda)$, due to Serra~\cite[Theorem~{ 2.1}]{Serra15}.

\begin{thm}[{\cite[Theorem~{ 2.1}]{Serra15}}]\label{thm:Liouville}
Let $\sigma_0 \in (0,2)$ and $\sigma\in[\sigma_0,2)$. There exists $\bar \alpha = \bar \alpha(N,\sigma_0,\Lambda_0,\Lambda) > 0$ such that the following holds. If $\ 0 < \alpha' < \alpha < \bar \alpha$ and $u \in C^{\sigma + \alpha'}_{loc}(\R^N)$ satisfies
\begin{itemize}
\item[(i)] there exists $C > 0$ such that, for all $0 \leq \gamma \leq \sigma + \alpha'$ and all $R \ge 1$,
	\[
	[u]_{C^\gamma (B_R)} \leq C\, R^{\sigma + \alpha - \gamma},
	\]
\item[(ii)] for all $h \in \R^N$,
	\[
	\MM_{\mathcal L_0}^- \big(u(\cdot + h) - u\big) \leq 0 \leq
	\MM_{\mathcal L_0}^+ \big(u(\cdot + h) - u\big) \quad \text{in } \R^N,
	\]
\item[(iii)] for every nonnegative $f \in L^1(\R^N)$ with compact support and $\|f\|_{L^1(\R^N)} = 1$,
	\[
	\MM_{\mathcal L_0}^+ \Big( \int_{\R^N} u(\cdot + h) f(h)\,dh - u \Big) \geq 0 \quad \text{in } \R^N,
	\]
\end{itemize}
then $u$ is a polynomial of degree at most $\nu := \lfloor \sigma + \alpha \rfloor$.
\end{thm}

The following proposition is the core of Theorem~\ref{thm1}. It is the counterpart of~\cite[Proposition~3.1]{Serra15}: the estimate is uniform over all kernels satisfying~\ref{K0} and~\ref{K2} and all running costs satisfying~\ref{L1}, and, in addition, it allows for a general norm $\|\cdot\|_*$ on $E$ and for point-dependent measure sets $\A(x)$.

\begin{prop}[A priori $C^{\sigma+\alpha}$ estimate]\label{prop:apriori-estimate}
Let $\sigma\in(0,2)$, and let $\bar\alpha\in(0,1)$ be given by Theorem~\ref{thm:Liouville}.
For all $0 < \alpha' < \alpha < \bar \alpha$ with $\nu := \lfloor \sigma+\alpha \rfloor < \sigma + \alpha'$, there exists {a constant $C > 0$, depending only on $N$, $\sigma$, $\alpha$, $\alpha'$ and the constants $\Lambda_0$, $\Lambda$, $K^*$ in~\ref{K0} and~\ref{K2},} such that,
if $u \in C^{\sigma+\alpha'}_b(\R^N)$ is a viscosity solution of
\begin{equation}\label{eq:effHJ-no-x}
	 \inf_{\mu \in \A(x)} \Big\{ \overline L^\mu u(x) + \overline \ell^\mu(x) \Big\} = 0
	\quad \text{in } B,
\end{equation}
where the family of kernels $\big(K(\cdot,\xi)\big)_{\xi \in X}$ satisfies~\ref{K0} and~\ref{K2}, the family of measure sets $(\A(x))_{x \in B}$ satisfies~\ref{A1}, and the running cost $\ell$ satisfies~\ref{L1}, with this exponent $\alpha$, then $u \in C^{\sigma+\alpha}(\overline{B_{1/2}})$ and
\[
	[ u ]_{C^{\sigma+\alpha}(B_{1/2})} \ \le \ C \big(1 + A_0 \big) \big( \|u\|_{C^{\sigma+\alpha'}(\R^N)} +  C_0 \big),
\]
where $A_0 \geq 0$ and $C_0>0$ are the constants in~\ref{A1} and~\ref{L1}, respectively.
\end{prop}

\begin{proof}[Proof of Proposition~\ref{prop:apriori-estimate}]
We follow closely the proof of~\cite[Proposition~3.1]{Serra15}, adapted to track the dependence of $\A(x)$ on $x$. We detail the steps in which this dependence plays a role, and only sketch those that are identical to~\cite{Serra15}.
Fix $\sigma$, $\alpha'$, $\alpha$ and $\nu$ as in the statement.

\smallskip
\emph{Step 1 (Contradiction and normalization).}
Suppose the estimate fails. Then, for each integer $k\ge1$, there exist a family of kernels $\big(K_k(\cdot,\xi)\big)_{\xi\in X}$ satisfying~\ref{K0} and~\ref{K2}, a family of measure sets $(\A_k(x))_{x \in B}$ satisfying~\ref{A1} with constant $A_{0,k} \geq 0$, a running cost $\ell_k$ satisfying~\ref{L1} with constant $C_{0,k} > 0$, and a viscosity solution $u^k \in C^{\sigma + \alpha'}_b(\R^N)$ of
\begin{equation}\label{eq:uk}
\inf_{\mu \in \A_k(x)} \Big\{ \overline L_k^\mu u^k(x) + \overline \ell_k^\mu(x) \Big\} = 0 \quad \text{in } B,
\end{equation}
where $\overline L_k^\mu := L_{\overline K_k^\mu}$, such that
\begin{equation}\label{eq:contr-est}
	[u^k]_{C^{\sigma + \alpha}(B_{1/2})} >
	k \big(1 + A_{0,k}\big) \big( \|u^k\|_{C^{\sigma+\alpha'}(\R^N)} + C_{0,k} \big).
\end{equation}

Equation~\eqref{eq:uk} is invariant under the simultaneous rescaling $u^k \mapsto u^k/\lambda$, $\ell_k \mapsto \ell_k/\lambda$, for any $\lambda>0$, which leaves the kernels $K_k$ unchanged. Choosing $\lambda = \lambda_k := (1+A_{0,k})\big(\|u^k\|_{C^{\sigma+\alpha'}(\R^N)} + C_{0,k}\big) > 0$, we may therefore assume that
\begin{equation}\label{eq:u^k-norm}
	\|u^k\|_{C^{\sigma+\alpha'}(\R^N)} + C_{0,k} \ = \ \frac{1}{1+A_{0,k}} \ \le \ 1, \qquad
	[u^k]_{C^{\sigma+\alpha}(B_{1/2})} \ > \ k.
\end{equation}
Throughout, the kernels $K_k$ satisfy~\ref{K0} and~\ref{K2} with the same constants $\Lambda_0$, $\Lambda$ and $K^*$.

We will reach a contradiction by a blow-up and compactness argument, using Theorem~\ref{thm:Liouville} to identify the limit of the rescaled sequence.

\smallskip
\emph{Step 2 (Blow-up sequence $(v_m)_m$).}
This step reproduces the construction of~\cite{Serra15}, and the dependence of $\A(x)$ on $x$ does not enter yet. For this reason, we only record the outcome, referring to~\cite{Serra15} for the details. Set
\[
\rho(r) := \sup_{k \in \N}\ \sup_{z \in B_{1/2}}\ \sup_{r' \geq r}\ (r')^{\alpha' - \alpha} [u^k]_{C^{\sigma + \alpha'}(B_{r'}(z))},
\]
a nonincreasing function, finite for every $r>0$, with $\rho(r)\to+\infty$ as $r\to0$, since $[u^k]_{C^{\sigma+\alpha}(B_{1/2})} > k$ by~\eqref{eq:u^k-norm}. For each $m\ge1$, we choose $k_m\in\N$, $z_m \in B_{1/2}$ and $r_m \geq 1/m$ nearly realizing the supremum defining $\rho(1/m)$, so that $r_m \to 0$.

Let $p_m$ be the best $L^2(B_{r_m}(z_m))$-approximation of $u^{k_m}$ by polynomials of degree at most $\nu$, and set
\[
	v_m(x) \ := \ \frac{u^{k_m}(z_m+r_m x) - p_m(z_m+r_m x)}{r_m^{\sigma+\alpha}\,\rho(r_m)}, \qquad x \in \R^N.
\]
By the choice of $p_m$, each $v_m$ is orthogonal to $P_\nu$  (the space of polynomials of degree at most $\nu$) in $L^2(B)$, that is,
\begin{equation}\label{eq:v_m-orth}
	\int_{B} v_m(x)\, q(x)\, dx \ = \ 0 \qquad \text{for all } q \in P_\nu.
\end{equation}
As in~\cite{Serra15}, the choice of $(k_m,z_m,r_m)$ gives
\begin{equation}\label{eq:v_m-nondeg}
	[v_m]_{C^{\sigma+\alpha'}(B)} \ \ge \ \frac12,
\end{equation}
and, after a further rescaling as in~\cite[(3.12)]{Serra15}, which preserves~\eqref{eq:v_m-orth}, \eqref{eq:v_m-nondeg} and~\eqref{eq:v_m-Growth} below, we may also assume that
\begin{equation}\label{eq:v_m-osc}
	\sup_{|l|=\nu}\ \mathrm{osc}_{B} D^l v_m \ \ge \ \frac14,
\end{equation}
while the monotonicity of $\rho$, the $L^2$-optimality of $p_m$ and interpolation give
\begin{equation}\label{eq:v_m-Growth}
	[v_m]_{C^\gamma(B_R)} \ \le \ CR^{\sigma+\alpha-\gamma} \qquad \text{for all } \gamma \in [0,\sigma+\alpha'] \text{ and } R\ge1.
\end{equation}
By the Arzel\`a--Ascoli theorem and a diagonal argument, there exists $v \in C^{\sigma+\alpha'}_{loc}(\R^N)$ such that, up to a subsequence, $v_m \to v$ in $C^{\sigma+\alpha'/2}_{loc}(\R^N)$ as $m\to\infty$.

\smallskip
\emph{Step 3 (Hypotheses of Theorem~\ref{thm:Liouville} for the limit $v$).}
Hypothesis~(i) is immediate, while hypotheses~(ii) and~(iii) are where the dependence of $\A(x)$ on $x$ enters.
In~\cite{Serra15}, the infimum is taken over the same index set at every point, so that comparing the equations satisfied at $\bar x$ and $\bar x+\bar h$ only involves the H\"older regularity of the running costs, and directly yields Pucci inequalities for $u(\cdot+\bar h)-u$. Here, the two equations involve different index sets, $\A_{k_m}(\bar x)$ and $\A_{k_m}(\bar x+\bar h)$, and this mismatch produces an error term $E_\A(\bar h)$, defined below. We control this term using~\ref{A1}, show that it vanishes in the blow-up limit, and then select a near-optimal measure $\mu\in\A_{k_m}(\bar x)$, whose rescaled kernel $\widetilde K^\mu \in \mathcal L_0$ allows us to pass to the extremal operators $\MM^\pm_{\mathcal L_0}$.

\smallskip
\emph{(i).} Since $v_m \to v$ in $C^{\sigma+\alpha'/2}_{loc}(\R^N)$, the lower semicontinuity of H\"older seminorms allows us to pass to the limit in~\eqref{eq:v_m-Growth}, which gives hypothesis~(i) of Theorem~\ref{thm:Liouville} for $v$:
\[
	[v]_{C^\gamma(B_R)} \ \le \ CR^{\sigma+\alpha-\gamma} \qquad \text{for all } \gamma \in [0,\sigma+\alpha'] \text{ and } R\ge1.
\]

\smallskip
\emph{(ii).} Let $x, h \in \R^N$ with $|x| + |h| < 1/(2r_m)$, and set $\bar x := z_m+r_m x$ and $\bar h := r_m h$, so that $\bar x, \bar x+\bar h \in B$. Since $u^{k_m} \in C^{\sigma+\alpha'}_b(\R^N)$, it satisfies~\eqref{eq:uk} pointwise in $B$. Writing~\eqref{eq:uk} at $\bar x+\bar h$ and at $\bar x$, and subtracting, we get
\[
	\inf_{\mu\in\A_{k_m}(\bar x+\bar h)}\Big\{\overline L_{k_m}^\mu u^{k_m}(\bar x+\bar h)+
		\overline\ell_{k_m}^\mu(\bar x+\bar h)\Big\}
	- \inf_{\mu\in\A_{k_m}(\bar x)}\Big\{\overline L_{k_m}^\mu u^{k_m}(\bar x)+
		\overline\ell_{k_m}^\mu(\bar x)\Big\} = 0.
\]
Inserting and subtracting the intermediate quantity obtained by freezing the index set at $\bar x$, we obtain
\begin{equation}\label{eq:v-decomp}
	E_{\A}(\bar h)  +  E_{L}(\bar h)  = \ 0,
\end{equation}
where
\begin{eqnarray*}
	E_{\A}(\bar h)  & :=  &
	\inf_{\mu\in\A_{k_m}(\bar x+\bar h)}\Big\{\overline L_{k_m}^\mu u^{k_m}(\bar x+\bar h)+
		\overline\ell_{k_m}^\mu(\bar x+\bar h)\Big\}
	- \inf_{\mu\in\A_{k_m}(\bar x)}\Big\{\overline L_{k_m}^\mu u^{k_m}(\bar x+\bar h)+
		\overline\ell_{k_m}^\mu(\bar x+\bar h)\Big\}, \\
	E_{L}(\bar h)  & :=  &
	\inf_{\mu\in\A_{k_m}(\bar x)}\Big\{\overline L_{k_m}^\mu u^{k_m}(\bar x+\bar h)+
		\overline\ell_{k_m}^\mu(\bar x+\bar h)\Big\} -
	\inf_{\mu\in\A_{k_m}(\bar x)}\Big\{\overline L_{k_m}^\mu u^{k_m}(\bar x)+
		\overline\ell_{k_m}^\mu(\bar x)\Big\}.
\end{eqnarray*}

\smallskip
\emph{The error term $E_{\A}(\bar h)$.} Since $\overline L_{k_m}^\mu$ and $\overline\ell_{k_m}^\mu$ are averages, we may write, for every $\mu\in\PP(X)$,
\[
	\overline L_{k_m}^\mu u^{k_m}(\bar x+\bar h)+\overline\ell_{k_m}^\mu(\bar x+\bar h) \ = \ \int_X \phi_{k_m}\,d\mu,
\]
where
\[
	\phi_{k_m}(\xi) := L_{K_{k_m}(\cdot,\xi)}u^{k_m}(\bar x+\bar h) + \ell_{k_m}(\bar x+\bar h,\xi).
\]
By~\ref{K2} and~\ref{L1}, $\phi_{k_m} \in E$. Since the Hausdorff distance between two nonempty bounded subsets of $\R$ dominates the difference of their infima, we have
\begin{eqnarray*}
	\left| E_{\A}(\bar h)\right|
	& = &\Big|  \inf_{\mu\in \A_{k_m}(\bar x + \bar h)} \int_X  \phi_{k_m}\, d\mu -
		  \inf_{\mu\in \A_{k_m}(\bar x)} \int_X  \phi_{k_m}\, d\mu \Big|
	 =  \big|\inf \langle \phi_{k_m}\rangle_{\bar x + \bar h} - \inf \langle \phi_{k_m}\rangle_{\bar x}\big| \\
	& \le & d_H \big(\langle  \phi_{k_m}\rangle_{\bar x},\langle  \phi_{k_m}\rangle_{\bar x+\bar h}\big)
	 \le  A_{0,k_m}\, \|  \phi_{k_m}\|_* \, |\bar h|^\alpha,
\end{eqnarray*}
where the last inequality follows from~\ref{A1}. By~\ref{K2} and~\ref{L1},
\begin{eqnarray*}
	\|  \phi_{k_m}\|_*
	& \le & \frac12\int_{\R^N} |\delta^2 u^{k_m}(\bar x+\bar h, z)|\, \|K_{k_m}(z,\cdot)\|_*\, dz
		+ \|\ell_{k_m}(\bar x+\bar h,\cdot)\|_* \\
	& \le & \frac{K^*}{2}\int_{\R^N} |\delta^2 u^{k_m}(\bar x+\bar h, z)|\, \frac{dz}{|z|^{N+\sigma}}
		+  C_{0,k_m} \\
	& \le  & C\Big(\|u^{k_m}\|_{C^{\sigma+\alpha'}(\R^N)} + C_{0,k_m}\Big)
	\ = \ \frac{C}{1+A_{0,k_m}},
\end{eqnarray*}
with $C$ depending only on $N,\sigma,K^*,\alpha'$, where the last equality is the normalization~\eqref{eq:u^k-norm}. Therefore,
\begin{equation}\label{eq:est-E_A}
	|E_{\A}(\bar h)| \ \le \ C\,\frac{A_{0,k_m}}{1+A_{0,k_m}}\, |\bar h|^\alpha \ \le \ C\,r_m^\alpha |h|^\alpha,
\end{equation}
with $C$ independent of $m$. The key point is that the factor $A_{0,k_m}$ coming from~\ref{A1} is exactly compensated by the factor $1/(1+A_{0,k_m})$ built into the normalization of Step~1.

\smallskip
\emph{The error term $E_L(\bar h)$.}
Fix $\epsilon>0$, and let $\mu_\epsilon^- \in \A_{k_m}(\bar x)$ be an $\epsilon$-minimizer of the second infimum in the definition of $E_L(\bar h)$, namely
\[
	\overline L_{k_m}^{\mu_\epsilon^-} u^{k_m}(\bar x) + \overline\ell_{k_m}^{\mu_\epsilon^-}(\bar x)
	\ \le \ \inf_{\mu\in\A_{k_m}(\bar x)}\Big\{\overline L_{k_m}^\mu u^{k_m}(\bar x)+\overline\ell_{k_m}^\mu(\bar x)\Big\} + \epsilon.
\]
Bounding the first infimum from above by its value at $\mu_\epsilon^-$, and using that the kernels do not depend on $x$, we get
\begin{eqnarray*}
	E_L(\bar h)
	& \le &
	\overline L_{k_m}^{\mu_\epsilon^-} u^{k_m}(\bar x+\bar h) + \overline\ell_{k_m}^{\mu_\epsilon^-}(\bar x+\bar h)
	- \Big( \overline L_{k_m}^{\mu_\epsilon^-} u^{k_m}(\bar x) + \overline\ell_{k_m}^{\mu_\epsilon^-}(\bar x) \Big) + \epsilon \\
	& = & \overline L_{k_m}^{\mu_\epsilon^-} \big(u^{k_m}(\cdot+\bar h) - u^{k_m}\big)(\bar x)
	+ \Big( \overline\ell_{k_m}^{\mu_\epsilon^-}(\bar x+\bar h) - \overline\ell_{k_m}^{\mu_\epsilon^-}(\bar x) \Big) + \epsilon \\
	& \le & \overline L_{k_m}^{\mu_\epsilon^-} \big(u^{k_m}(\cdot+\bar h) - u^{k_m}\big)(\bar x)
	+ C_{0,k_m}|\bar h|^\alpha + \epsilon,
\end{eqnarray*}
where we used~\ref{L1} in the last line. By~\eqref{eq:v-decomp} and~\eqref{eq:est-E_A}, $|E_L(\bar h)| = |E_\A(\bar h)| \le C|\bar h|^\alpha$, and $C_{0,k_m} \le 1$ by~\eqref{eq:u^k-norm}. Hence, recalling that $\bar h = r_m h$,
\begin{equation}\label{eq:mu-minus}
	\overline L_{k_m}^{\mu_\epsilon^-}\big(u^{k_m}(\cdot+\bar h)-u^{k_m}\big)(\bar x) \ \ge \ -C\,r_m^\alpha|h|^\alpha - \epsilon.
\end{equation}
Symmetrically, taking $\mu_\epsilon^+ \in \A_{k_m}(\bar x)$ an $\epsilon$-minimizer of the first infimum in the definition of $E_L(\bar h)$, we obtain
\begin{equation}\label{eq:mu-plus}
	\overline L_{k_m}^{\mu_\epsilon^+}\big(u^{k_m}(\cdot+\bar h)-u^{k_m}\big)(\bar x) \ \le \ C\,r_m^\alpha|h|^\alpha + \epsilon.
\end{equation}

\smallskip
\emph{The equation for $v_m$.} We now rewrite~\eqref{eq:mu-minus} and~\eqref{eq:mu-plus} in terms of $v_m$. By~\eqref{eq:L-mu-sym}, for every $\mu\in\PP(X)$,
\[
	\overline L_{k_m}^\mu \big(u^{k_m}(\cdot+\bar h)-  u^{k_m}\big)(\bar x)
	\ = \ \frac12\int_{\R^N} \Big(\delta^2 u^{k_m}(\bar x+\bar h,z) - \delta^2 u^{k_m}(\bar x,z)\Big)\, \overline K_{k_m}^\mu(z)\, dz.
\]
By definition of $v_m$, for all $y, z \in \R^N$,
\[
	\rho(r_m)\,r_m^{\sigma+\alpha}\,\delta^2v_m(y,z) \ = \ \delta^2u^{k_m}(z_m+r_m y,r_m z) - \delta^2 p_m(z_m+r_m y,r_m z).
\]
Since $\nu \le 2$, $\delta^2 p_m(\cdot,z)$ is constant, so that the polynomial $p_m$ cancels when taking the difference between $y = x+h$ and $y = x$. Changing variables $z \mapsto r_m z$ in the integral above, we therefore obtain
\begin{eqnarray*}
	\overline L_{k_m}^\mu \big(u^{k_m}(\cdot+\bar h)  -  u^{k_m}\big)(\bar x)
	& = & \rho(r_m)\,r_m^{\sigma+\alpha}\, \frac12\int_{\R^N} \Big(\delta^2 v_m(x+h, z) - \delta^2 v_m(x,z)\Big)\,
	r_m^N\,\overline K_{k_m}^\mu(r_m z)\, dz \\
	& = & \rho(r_m)\,r_m^{\alpha}\, L_{\widetilde K^\mu}\big(v_m(\cdot+h)-v_m\big)(x),
\end{eqnarray*}
where
\[
	\widetilde K^\mu(z):=r_m^{N+\sigma}\,\overline K_{k_m}^\mu(r_m z) \in \mathcal L_0(\sigma,\Lambda_0,\Lambda).
\]

Applying this identity to $\mu_\epsilon^\pm$ in~\eqref{eq:mu-plus} and~\eqref{eq:mu-minus}, and dividing by $\rho(r_m)\,r_m^\alpha$, we obtain
\begin{equation}\label{eq:v_m}
	L_{\widetilde K^{\mu_\epsilon^+}}\big(v_m(\cdot+h)-v_m\big)(x)
	\le \frac{C|h|^\alpha + \epsilon\, r_m^{-\alpha}}{\rho(r_m)}, \qquad
	L_{\widetilde K^{\mu_\epsilon^-}}\big(v_m(\cdot+h)-v_m\big)(x)
	\ge  -\frac{C|h|^\alpha + \epsilon\, r_m^{-\alpha}}{\rho(r_m)}.
\end{equation}
Since $\widetilde K^{\mu_\epsilon^\pm} \in \mathcal L_0(\sigma,\Lambda_0,\Lambda)$, the left-hand sides are bounded below and above by $\MM^-_{\mathcal L_0}$ and $\MM^+_{\mathcal L_0}$, respectively, which do not depend on $\epsilon$. Letting $\epsilon \to 0$, we get
\[
	\MM^-_{\mathcal L_0}\big(v_m(\cdot+h)-v_m\big)(x)
	\ \le \ \frac{C|h|^\alpha}{\rho(r_m)}, \qquad
	 \MM^+_{\mathcal L_0}\big(v_m(\cdot+h)-v_m\big)(x)
	 \ \ge \ -\frac{C|h|^\alpha}{\rho(r_m)}.
\]
Letting now $m\to\infty$, and using $\rho(r_m)\to+\infty$, $v_m\to v$ in $C^{\sigma+\alpha'/2}_{loc}(\R^N)$ and the uniform growth bound~\eqref{eq:v_m-Growth}, which controls the tails of the integrals, we obtain, for every $h\in\R^N$,
\[
	\MM_{\mathcal L_0}^-\big(v(\cdot+h)-v\big) \ \leq \ 0 \ \leq \ \MM_{\mathcal L_0}^+\big(v(\cdot+h)-v\big) \qquad \text{in } \R^N,
\]
which is hypothesis~(ii) of Theorem~\ref{thm:Liouville}.

\smallskip
\emph{(iii).} Let $f\in L^1(\R^N)$ be nonnegative, supported in $B_{R_0}$ for some $R_0>0$, with $\|f\|_{L^1(\R^N)}=1$. For hypothesis~(iii), only the lower bound~\eqref{eq:mu-minus} is needed. The key observation is that $\mu_\epsilon^-$ is an $\epsilon$-minimizer of the infimum at $\bar x$ only, so that $\mu_\epsilon^-$, and hence $\widetilde K^{\mu_\epsilon^-}$, does \emph{not} depend on $\bar h$. This allows us to keep the kernel fixed while integrating in $h$.

Let $x \in \R^N$ with $|x| + R_0 < 1/(2r_m)$. Multiplying the second inequality in~\eqref{eq:v_m} by $f(h) \geq 0$, integrating over $h \in B_{R_0}$, and using Fubini's theorem together with $\|f\|_{L^1(\R^N)} = 1$, as in~\cite{Serra15}, we obtain
\[
	L_{\widetilde K^{\mu_\epsilon^-}}\Big(\int_{\R^N} v_m(\cdot+h) f(h)\,dh - v_m\Big)(x) \ \geq \
	-\frac{C R_0^\alpha + \epsilon\, r_m^{-\alpha}}{\rho(r_m)}.
\]
Since $\widetilde K^{\mu_\epsilon^-} \in \mathcal L_0(\sigma,\Lambda_0,\Lambda)$, the left-hand side is bounded above by $\MM^+_{\mathcal L_0}$ applied to the same function, which does not depend on $\epsilon$. Letting $\epsilon \to 0$, we get
\[
	\MM^+_{\mathcal L_0}\Big(\int_{\R^N} v_m(\cdot+h) f(h)\,dh - v_m\Big)(x) \ \geq \ -\frac{C R_0^\alpha}{\rho(r_m)}.
\]
Letting now $m\to\infty$, and using $\rho(r_m)\to+\infty$, $v_m \to v$ in $C^{\sigma+\alpha'/2}_{loc}(\R^N)$ and the growth bound~\eqref{eq:v_m-Growth} to control the tails, we conclude that
\[
	\MM^+_{\mathcal L_0}\Big(\int_{\R^N} v(\cdot+h) f(h)\,dh - v\Big) \ \geq \ 0 \qquad \text{in } \R^N,
\]
which is hypothesis~(iii) of Theorem~\ref{thm:Liouville}.

\smallskip
\emph{Step 4 (Conclusion).}
By Step~3, $v$ satisfies all the hypotheses of Theorem~\ref{thm:Liouville}, and is therefore a polynomial of degree at most $\nu$. On the other hand, passing~\eqref{eq:v_m-orth} to the limit, we obtain that $v$ is orthogonal to every polynomial of degree at most $\nu$ in $L^2(B)$, and hence $v \equiv 0$ in $B$. But passing~\eqref{eq:v_m-osc} to the limit, as in~\cite{Serra15}, which is possible since $D^\nu v_m \to D^\nu v$ uniformly in $\overline B$, shows that $v$ cannot vanish identically in $B$, a contradiction. This proves the proposition.
\end{proof}

\subsection{Kernels depending on \texorpdfstring{$x$}{x}}
We now allow the kernel to depend on $x$,
$$
K: B \times (\R^N \setminus \{ 0 \}) \times X \to \R.
$$
Freezing the kernel at a point, we apply Proposition~\ref{prop:apriori-estimate} to the frozen problem, and control the resulting error by the H\"older continuity of $x \mapsto \overline K^\mu(x, \cdot)$ given by~\ref{K1}.

\begin{prop}[$C^{\sigma+\alpha}$ estimate for $x$-dependent kernels]\label{prop:x-dep-estimate}
Let $\sigma\in(0,2)$, and let $\bar\alpha\in(0,1)$ be given by Theorem~\ref{thm:Liouville}.
For all $0 < \alpha' < \alpha < \bar \alpha$ with $\nu := \lfloor \sigma+\alpha \rfloor < \sigma + \alpha'$, there exists {a constant $C > 0$, depending only on $N$, $\sigma$, $\alpha$, $\alpha'$ and the constants $\Lambda_0$, $\Lambda$, $K_0$, $K^*$ in~\ref{K0}--\ref{K2},} such that,
if $u \in C^{\sigma+\alpha}(\overline B)\cap C^\alpha_b(\R^N)$ is a viscosity solution of
\[
	 \inf_{\mu \in \A(x)} \Big\{ \overline L_x^\mu u(x) + \overline \ell^\mu(x) \Big\} = 0
	\quad \text{in } B,
\]
where the family of kernels $\big(K(x,\cdot,\xi)\big)_{x \in B, \xi \in X}$ satisfies~\ref{K0}--\ref{K2}, the family of measure sets $(\A(x))_{x \in B}$ satisfies~\ref{A1}, and the running cost $\ell$ satisfies~\ref{L1}, with this exponent $\alpha$, then
\[
	[u]_{C^{\sigma + \alpha}(B_{1/4})} \leq C \big(1 + A_0\big)
	\Big( \| u\|_{C^{\sigma + \alpha'}(B)} +
	K_0\, [u]_{C^{\sigma + \alpha}(B)} +
	\| u \|_{C^\alpha(\R^N)} + C_0 \Big),
\]
where $A_0 \geq 0$ and $C_0>0$ are the constants in~\ref{A1} and~\ref{L1}, respectively.
\end{prop}
{The dependence of $C$ on $K_0$ is nondecreasing; the factor $K_0$ in front of $[u]_{C^{\sigma+\alpha}(B)}$ is kept explicit, as it is made small by rescaling in the proof of Theorem~\ref{thm1}.}

\begin{proof}
We first prove the estimate under the additional assumption $u \in C^{\sigma + \alpha}_b(\R^N)$, which allows us to apply Proposition~\ref{prop:apriori-estimate} directly and to isolate the difficulties that are new to the $x$-dependent problem. We then remove this assumption by a localization argument, as in~\cite{Serra15}.

\smallskip
\emph{Step 1 (Freezing the kernel).}
Assume that $u \in C^{\sigma + \alpha}_b(\R^N)$. Freezing the kernel at the origin, we rewrite the equation as
\begin{equation}\label{eq:HJ-frozen}
\inf_{\mu \in \mathcal A(x)} \big\{ \overline L_0^\mu u(x) + \overline{\tilde\ell\,}^\mu(x) \big\} = 0 \quad \text{in } B,
\end{equation}
where
\[
	\overline L^\mu_0 u(x) := \frac12 \int_{\R^N} \delta^2 u(x,z)\, \overline K^\mu (0, z)\, dz
\]
is the operator with kernel frozen at the origin, and $\tilde \ell = \ell + T$, where
\begin{align*}
	T(x,\xi) & := \frac12 \int_{\R^N} \delta^2u(x,z) \big(K(x,z,\xi) - K(0,z,\xi)\big)\, dz.
\end{align*}
Indeed, $\overline L_x^\mu u = \overline L_0^\mu u + \overline T^\mu$, so that the difference between the original and the frozen operator is collected into the new running cost $\tilde\ell$.

\smallskip
\emph{Step 2 (The new running cost $\tilde\ell$).}
To apply Proposition~\ref{prop:apriori-estimate} to~\eqref{eq:HJ-frozen}, we need $\tilde\ell$ to satisfy~\ref{L1}, that is, to control both $\|\tilde\ell(x,\cdot)\|_*$ and the H\"older seminorm of $\overline{\tilde\ell\,}^\mu$.

\smallskip
We start with the $\|\cdot\|_*$ norm. By~\ref{K2}, applied at $x$ and at $0$, we have $T(x,\cdot) \in E$ and
\[
	\|T(x,\cdot)\|_*
	\ \leq \ \frac12 \int_{\R^N} |\delta^2u(x,z)| \, \big(\|K(x,z,\cdot)\|_* + \|K(0,z,\cdot)\|_*\big)\, dz
	\ \leq \ K^* \int_{\R^N} |\delta^2u(x,z)|\, |z|^{-N-\sigma}\, dz.
\]
By Lemma~\ref{lem:near-far}, with $\alpha'$ in place of $\alpha$ and $r = 1/2$, we conclude that
\begin{equation}\label{eq:T-star}
	\sup_{x \in B_{1/2}} \|T(x,\cdot)\|_*
	\ \leq \ C \big(\|u\|_{C^{\sigma+\alpha'}(B)} + \|u\|_{L^\infty(\R^N)}\big),
\end{equation}
and therefore, by~\ref{L1},
\[
\sup_{x \in B_{1/2}} \|\tilde\ell(x,\cdot)\|_* \le C\big(\|u\|_{C^{\sigma+\alpha'}(B)} + \|u\|_{L^\infty(\R^N)}\big) + C_0.
\]

\smallskip
We continue with the H\"older seminorm, assuming for simplicity that $\sigma+\alpha<2$; the case $\sigma+\alpha>2$ is identical, using the corresponding bounds of Lemmas~\ref{lem:near-far} and~\ref{lem:delta2-increments}.
Since $\overline{\tilde\ell\,}^\mu = \overline T^\mu + \overline\ell^\mu$ and $\overline\ell^\mu$ is controlled by~\ref{L1}, it suffices to estimate $\overline T^\mu$. For $x, x' \in B_{1/2}$, we split
\begin{align*}
\overline T^\mu(x) - \overline T^\mu(x') & = T_1^\mu(x, x') + T_2^\mu(x, x'), \quad \text{where} \\
T_1^\mu(x, x') & := \frac12 \int_{\R^N} \delta^2u(x', z) \left( \overline K^\mu(x, z) - \overline K^\mu(x',z)\right) dz, \\
T_2^\mu (x, x') & := \frac12 \int_{\R^N}  \left( \delta^2u(x, z) - \delta^2u(x', z)\right) \left(  \overline K^\mu(x, z) - \overline K^\mu(0, z) \right) dz.
\end{align*}
The H\"older continuity of $T_1^\mu$ comes from that of the kernel, and that of $T_2^\mu$ from that of $u$. We estimate the two terms separately.

For the first term we use~\ref{K1}, and Lemma~\ref{lem:near-far} with $\alpha'$ in place of $\alpha$ and $r = 1/2$:
\begin{align*}
	|T_1^\mu(x, x')|
	\leq & \ \frac12 \int_{\R^N} |\delta^2u(x', z)| \left| \overline K^\mu(x, z) - \overline K^\mu(x',z)\right| dz \\
	\leq & \ \|u\|_{C^{\sigma+\alpha'}(B)} \int_{B_{1/2}}  |z|^{\sigma+\alpha'}
		\left|\overline K^\mu(x, z) - \overline K^\mu(x', z)\right| dz  \\
	& \quad
	 + 2\, \| u \|_{L^\infty(\R^N)}  \int_{B_{1/2}^c} \left|\overline K^\mu(x, z) - \overline K^\mu(x', z)\right| dz \\
	=: & \ \|u\|_{C^{\sigma+\alpha'}(B)}\, I_1(x, x') + 2\,\| u \|_{L^\infty(\R^N)}\, I_2(x, x').
\end{align*}
For $I_1(x,x')$, we decompose $B_{1/2}$ into dyadic annuli $B_{2^{-k}} \setminus B_{2^{-k-1}}$, $k \geq 1$, and apply~\ref{K1} with $r = 2^{-k-1}$ on each:
\begin{align*}
	I_1(x,x') = & \sum_{k \geq 1} \int_{B_{2^{-k}} \setminus B_{2^{-k-1}}}
		|z|^{\sigma + \alpha'} \left|\overline K^\mu(x, z) - \overline K^\mu(x', z)\right| dz \\
	\leq & \sum_{k \geq 1} (2^{-k})^{\sigma + \alpha'} \int_{B_{2^{-k}} \setminus B_{2^{-k-1}}}
		 \left|\overline K^\mu(x, z) - \overline K^\mu(x', z)\right| dz \\
	\leq & \sum_{k \geq 1} (2^{-k})^{\sigma + \alpha'} K_0 \frac{|x - x'|^\alpha}{(2^{-k-1})^\sigma}
	=  \ 2^\sigma K_0 |x - x'|^\alpha \sum_{k \geq 1} 2^{-k\alpha'} \ = \ C K_0 |x - x'|^\alpha,
\end{align*}
the series converging since $\alpha' > 0$.
Similarly, for $I_2(x,x')$ we decompose $B_{1/2}^c$ into dyadic annuli $B_{2^k} \setminus B_{2^{k-1}}$, $k \geq 0$, and apply~\ref{K1} with $r = 2^{k-1}$:
\begin{align*}
I_2(x,x') = & \sum_{k \geq 0} \int_{B_{2^k} \setminus B_{2^{k-1}}} \left|\overline K^\mu(x, z) - \overline K^\mu(x', z)\right| dz \\
\leq & \sum_{k \geq 0} K_0 \frac{|x - x'|^\alpha}{(2^{k-1})^\sigma} \ = \ 2^\sigma K_0 |x - x'|^\alpha \sum_{k \geq 0} 2^{-k\sigma} \ = \ C K_0 |x - x'|^\alpha,
\end{align*}
the series converging since $\sigma > 0$.
Combining the bounds on $I_1(x,x')$ and $I_2(x,x')$, we obtain
\begin{equation}\label{eq:est-T1}
	|T_1^\mu(x, x')| \leq C K_0 \Big(\|u\|_{C^{\sigma + \alpha'}(B)} + \| u \|_{L^\infty(\R^N)}\Big)|x - x'|^\alpha.
\end{equation}

\smallskip
We now turn to $T_2^\mu$, where the roles are reversed: the kernel difference is merely bounded, and the H\"older continuity in $x$ comes from $u$. We write the computation for $\nu = 1$; the case $\nu = 0$ is analogous, using the corresponding bounds of Lemma~\ref{lem:delta2-increments}.

Let $x, x' \in B_{1/2}$ with $d := |x-x'| \leq 1/2$; for $d > 1/2$, the bound below follows directly from the same estimates applied to $\overline T^\mu(x)$ and $\overline T^\mu(x')$ separately. By Lemma~\ref{lem:delta2-increments},
\begin{align*}
	|T_2^\mu(x,x')|
	\leq & \ C [u]_{C^{\sigma+\alpha}(B)} \int_{B_d} |z|^{\sigma+\alpha} \left|\overline K^\mu(x,z) - \overline K^\mu(0,z)\right| dz \\
	& + C [u]_{C^{\sigma+\alpha}(B)}\, d \int_{B_{1/2} \setminus B_d} |z|^{\sigma+\alpha-1} \left|\overline K^\mu(x,z) - \overline K^\mu(0,z)\right| dz \\
	& + 2\, \| u \|_{C^\alpha(\R^N)}\, d^\alpha \int_{B_{1/2}^c} \left|\overline K^\mu(x,z) - \overline K^\mu(0,z)\right| dz.
\end{align*}
By~\ref{K1} at the pair $(x,0)$,
\[
	\int_{B_{2r}\setminus B_r} |\overline K^\mu(x,z) - \overline K^\mu(0,z)|\, dz \leq K_0\, r^{-\sigma}\quad  \text{for all } r>0.
\]
Summing over dyadic annuli, as for $I_1$ and $I_2$, the three integrals above are bounded by $C K_0\, d^\alpha$, $C K_0\, d^{\alpha-1}$ and $C K_0$, respectively, using $\alpha>0$, $\alpha<1$ and $\sigma>0$. Therefore,
\begin{equation}\label{eq:est-T2}
	|T_2^\mu(x,x')|
	 \leq  C K_0 \Big( [u]_{C^{\sigma+\alpha}(B)} + \|u\|_{C^\alpha(\R^N)}\Big) |x - x'|^\alpha.
\end{equation}

\smallskip
Combining~\eqref{eq:est-T1} and~\eqref{eq:est-T2}, and using $\|u\|_{L^\infty(\R^N)} \le \|u\|_{C^\alpha(\R^N)}$, we obtain
\begin{equation}\label{eq:est-Tmu}
	\sup_{\mu \in \PP(X)} \left[ \overline T^\mu \right]_{C^{\alpha}(B_{1/2})}
	\leq C \big( K_0\, [u]_{C^{\sigma+\alpha}(B)} + \|u\|_{C^{\sigma+\alpha'}(B)} + \|u\|_{C^\alpha(\R^N)} \big).
\end{equation}
Together with~\eqref{eq:T-star} and~\ref{L1}, this shows that $\tilde\ell = T + \ell$ satisfies~\ref{L1} in $B_{1/2}$, with constant
\begin{equation}\label{eq:ell-tilde}
	\tilde C_0 \ := \ C \big( K_0\, [u]_{C^{\sigma+\alpha}(B)} + \|u\|_{C^{\sigma+\alpha'}(B)} + \|u\|_{C^\alpha(\R^N)} \big) + C_0.
\end{equation}

\smallskip
\emph{Step 3 (Conclusion in the global case).}
Equation~\eqref{eq:HJ-frozen} satisfies, in $B_{1/2}$, all the hypotheses of Proposition~\ref{prop:apriori-estimate}, with the running cost $\ell$ replaced by $\tilde\ell$ and the constant $C_0$ replaced by $\tilde C_0$. Proposition~\ref{prop:apriori-estimate} therefore yields
\begin{equation}\label{eq:apriori-x-dep-raw}
	[u]_{C^{\sigma+\alpha}(B_{1/4})} \leq C(1+A_0)
	\big(\|u\|_{C^{\sigma+\alpha'}(\R^N)} + K_0\, [u]_{C^{\sigma+\alpha}(B)} + \|u\|_{C^\alpha(\R^N)} + C_0 \big).
\end{equation}

\smallskip
\emph{Step 4 (Localization).}
We now remove the assumption $u \in C^{\sigma+\alpha}_b(\R^N)$, and only assume, as in the statement, that $u \in C^{\sigma+\alpha}(\overline B) \cap C^\alpha_b(\R^N)$. We argue as in~\cite{Serra15}. Fix a cut-off $\eta \in C_c^\infty(B)$ with $0\leq\eta\leq1$ and $\eta \equiv 1$ on $B_{3/4}$, and set $u_\eta := \eta u$. Then $u_\eta \in C^{\sigma+\alpha}_b(\R^N)$, and
\begin{equation}\label{eq:cutoff-reg}
	\|u_\eta\|_{C^{\sigma+\alpha'}(\R^N)} \leq  C \|u \|_{C^{\sigma+\alpha'}(B)}, \qquad
	[u_\eta]_{C^{\sigma+\alpha}(B)} \leq [u]_{C^{\sigma+\alpha}(B)} + C\|u \|_{C^{\sigma+\alpha'}(B)}.
\end{equation}

We write $u = u_\eta + w$, with $w := (1-\eta)u$, and note that $u_\eta$ solves the original equation in $B_{1/2}$, with running cost $\ell_\eta := \ell + S$, where
\[
	S(x,\xi) := \frac12 \int_{\R^N} \delta^2 w(x,z)\, K(x,z,\xi)\, dz, \qquad x \in B_{1/2}.
\]

To estimate its $\|\cdot\|_*$ norm, we note that, for $x \in B_{1/2}$, $\delta^2 w(x,z)$ vanishes for $|z| < 1/4$ and is bounded by $2\|u\|_{L^\infty(\R^N)}$ otherwise. Thus, arguing as for~\eqref{eq:T-star}, we get
\[
\sup_{x\in B_{1/2}}\| S (x,\cdot)\|_* \leq C\|u\|_{L^\infty(\R^N)}.
\]
For the H\"older seminorm of $\overline S^\mu$, let $x, x' \in B_{1/2}$ and observe that $w(x) = w(x') = 0$, so that, for $|z|\geq 1/4$,
\[
	\big|\delta^2w(x,z)-\delta^2w(x',z)\big| \le 2\,[w]_{C^\alpha(\R^N)}\,|x-x'|^\alpha \le C\,\|u\|_{C^\alpha(\R^N)}\,|x-x'|^\alpha,
	\qquad
	\big|\delta^2w(x',z)\big| \le 2\,\|u\|_{L^\infty(\R^N)},
\]
while both expressions vanish for $|z| < 1/4$. Therefore, using the upper bound in~\ref{K0} for the first term, and~\ref{K1} on dyadic annuli, as for $I_2$, for the second one, we get
\begin{eqnarray*}
	\left|\overline S^\mu(x) - \overline S^\mu(x')\right|
	& \le &  \frac12 \int_{\R^N} \left|\delta^2w(x,z)-\delta^2w(x',z)\right| \overline K^\mu(x,z)\, dz  \\
	& & + \frac12 \int_{\R^N} \left|\delta^2w(x',z)\right| \left|\overline K^\mu(x,z)-\overline K^\mu(x',z)\right| dz \\
	& \le & C\,\|u\|_{C^\alpha(\R^N)}\,|x-x'|^\alpha \int_{|z|\ge 1/4} \Lambda |z|^{-N-\sigma}\, dz  \\
	 && + \|u\|_{L^\infty(\R^N)} \int_{|z|\ge 1/4} \left|\overline K^\mu(x,z)-\overline K^\mu(x',z)\right| dz \\
	& \le & C\,\|u\|_{C^\alpha(\R^N)}\,|x-x'|^\alpha.
\end{eqnarray*}

Combining this with the bound on $\|S(x,\cdot)\|_*$ and with~\ref{L1} for $\ell$, we conclude that $\ell_\eta$ satisfies~\ref{L1} in $B_{1/2}$, with constant $C_0 + C\|u\|_{C^\alpha(\R^N)}$. Steps~1--3 then apply to $u_\eta$, and~\eqref{eq:apriori-x-dep-raw} gives
\[
	[u_\eta]_{C^{\sigma+\alpha}(B_{1/4})} \leq C(1+A_0)
	\Big(\|u_\eta\|_{C^{\sigma+\alpha'}(\R^N)} +
	K_0\, [u_\eta]_{C^{\sigma+\alpha}(B)} +
	\|u_\eta\|_{C^\alpha(\R^N)} + C_0 + \|u\|_{C^\alpha(\R^N)} \Big).
\]
Finally, since $u_\eta = u$ in $B_{1/4}$ and $\|u_\eta\|_{C^\alpha(\R^N)} \le C\|u\|_{C^\alpha(\R^N)}$, we obtain, by~\eqref{eq:cutoff-reg},
\[
	[u]_{C^{\sigma + \alpha}(B_{1/4})} \leq C (1 + A_0)
	\Big( \| u\|_{C^{\sigma + \alpha'}(B)} + K_0\, [u]_{C^{\sigma + \alpha}(B)} + \| u \|_{C^\alpha(\R^N)} + C_0 \Big),
\]
which is the estimate of Proposition~\ref{prop:x-dep-estimate}.

\end{proof}

\subsection{The final estimate}
We now prove Theorem~\ref{thm1}, following~\cite[Proof of Theorem~1.1]{Serra15}. We apply Proposition~\ref{prop:x-dep-estimate} on balls $B_\rho(x_0) \subset B$ of varying radius, and collect the resulting bounds into an estimate for an adimensional H\"older seminorm.

For $\Omega\subset\R^N$ and $\alpha>0$ with $\alpha\notin\N$, we define the adimensional H\"older seminorm
\[
	[u]^*_{\alpha;\Omega} := \sup_{x \neq y \in \Omega,\ |l| = k} d_{x,y}^{\alpha}\,\frac{|D^l u(x) - D^l u(y)|}{|x - y|^{\alpha-k}},
	\qquad k := \lfloor\alpha\rfloor, \quad d_{x,y} := \min \{ \mathrm{dist}(x, \partial \Omega), \mathrm{dist}(y, \partial \Omega)\}.
\]
We use the following classical properties, see~\cite[Sections~4.1 and~6.8]{GT77}.

\begin{remark}\label{rmk:adim}
Let $\nu := \lfloor\sigma+\alpha\rfloor$, and let $0<\alpha'<\alpha$ be such that $\nu<\sigma+\alpha'$, so that $\lfloor\sigma+\alpha'\rfloor = \nu$.
\begin{enumerate}[label=\textup{(\roman*)},ref=\textup{(\roman*)}]
\item \label{rmk:adim-seminorm} If $\Omega'\subset\Omega$ and $\mathrm{dist}(\Omega',\partial\Omega)\geq\delta>0$, then
	\[
	[u]_{C^{\sigma+\alpha}(\Omega')} \leq \delta^{-(\sigma+\alpha)}\, [u]^*_{\sigma+\alpha;\Omega},
	\]
	and the same holds with $\alpha'$ in place of $\alpha$.
\item \label{rmk:adim-seminorm2} We have
	\[
	[u]^*_{\sigma+\alpha;\Omega} \leq
	2^{\sigma+\alpha} \sup_{B_{2\rho}(x_0)\subset\Omega} \rho^{\sigma+\alpha}[u]_{C^{\sigma+\alpha}(B_{\rho/4}(x_0))}
	\ + \ 8^{\alpha-\alpha'}\,[u]^*_{\sigma+\alpha';\Omega}.
	\]
\item \label{rmk:adim-interp} For every $\varepsilon>0$, there exists $C_\varepsilon>0$,
	depending only on $\varepsilon$, $N$, $\sigma$, $\alpha$ and $\alpha'$, such that
	\[
	[u]^*_{\sigma+\alpha';\Omega} \leq \varepsilon\, [u]^*_{\sigma+\alpha;\Omega} + C_\varepsilon\, \|u\|_{L^\infty(\Omega)}.
	\]
\end{enumerate}
\end{remark}

\begin{proof}[Proof of Theorem~\ref{thm1}]
We follow closely the lines of~\cite[Theorem~1.1]{Serra15}.
Let $x_0 \in \R^N$ and $\rho \in (0,1)$ be such that $B_{2\rho}(x_0) \subset B$. The rescaled function $\tilde u(\tilde x) := u(x_0+\rho\tilde x)$ solves
\begin{equation}\label{eq:rescaled-eq}
	\inf_{\mu \in \A(x_0+\rho\tilde x)} \Big\{ \overline L^\mu_\rho \tilde u(\tilde x) + \overline\ell^\mu_\rho(\tilde x) \Big\} = 0 \qquad \text{in } B,
\end{equation}
where $\overline L^\mu_\rho := L_{\overline K^\mu_\rho}$, and $\overline K^\mu_\rho$ and $\overline\ell^\mu_\rho$ are the averages of the rescaled kernel and running cost
\[
	K_\rho(\tilde x,\tilde z,\xi) := \rho^{N+\sigma}\, K(x_0+\rho\tilde x, \rho \tilde z, \xi), \qquad
	\ell_\rho(\tilde x,\xi) := \rho^{\sigma}\, \ell(x_0+\rho\tilde x, \xi){ ,}
\]
for $\tilde x \in B$, $\tilde z\in\R^N\setminus\{0\}$ and $\xi\in X$.
We check that~\eqref{eq:rescaled-eq} satisfies the hypotheses of Proposition~\ref{prop:x-dep-estimate}, with constants independent of $x_0$ and $\rho$.

We start with the assumptions on the kernel. It is immediate to see, upon rescaling, that $\overline K^\mu_\rho$ belongs to the same ellipticity class $\mathcal L_0(\sigma,\Lambda_0,\Lambda)$ as $\overline K^\mu$, so that~\ref{K0} is satisfied.

For~\ref{K1}, fix $\tilde x,\tilde x' \in B$ and $r>0$. The points $x_0+\rho\tilde x$ and $x_0+\rho\tilde x'$ lie in $B_\rho(x_0) \subset B$, so that, substituting $z = \rho\tilde z$ and applying~\ref{K1} for $\overline K^\mu$ with radius $\rho r$,
\begin{align*}
	\int_{B_{2r}\setminus B_r} \left|\overline K^\mu_\rho(\tilde x,\tilde z) - \overline K^\mu_\rho(\tilde x',\tilde z)\right| d\tilde z
	= & \ \rho^\sigma \int_{B_{2\rho r}\setminus B_{\rho r}} \left|\overline K^\mu(x_0+\rho\tilde x,z) - \overline K^\mu(x_0+\rho\tilde x',z)\right| dz \\
	\leq & \ \rho^\sigma\, K_0 \frac{(\rho|\tilde x-\tilde x'|)^\alpha}{(\rho r)^\sigma} = K_0 \rho^\alpha \, \frac{|\tilde x-\tilde x'|^\alpha}{r^\sigma}.
\end{align*}
Hence $K_\rho$ satisfies~\ref{K1} with constant $K_0 \rho^\alpha$ in place of $K_0$.

For~\ref{K2}, since $\|\cdot\|_*$ is positively homogeneous,
\[
	\|K_\rho(\tilde x,\tilde z,\cdot)\|_* = \rho^{N+\sigma}\, \|K(x_0+\rho\tilde x, \rho \tilde z, \cdot)\|_* \leq \rho^{N+\sigma}\, K^*\,(\rho|\tilde z|)^{-N-\sigma} = K^* |\tilde z|^{-N-\sigma}.
\]
Moreover, the same change of variables gives $L_{K_\rho(\tilde x,\cdot,\xi)}\tilde u(\tilde x) = \rho^\sigma L_{K(x,\cdot,\xi)}u(x)$, with $x = x_0+\rho\tilde x$, so that the second part of~\ref{K2} for $K_\rho$ follows from the one for $K$.

The family $\{\A(x_0+\rho \tilde x)\}_{\tilde x\in B}$ satisfies~\ref{A1} with constant $A_0 \rho^\alpha$ in place of $A_0$. Indeed, for $\tilde x,\tilde x' \in B$ and $\phi \in E$, applying~\ref{A1} at the points $x_0+\rho\tilde x, x_0+\rho\tilde x' \in B$, we have
\[
	d_H\big(\langle\phi\rangle_{x_0+\rho\tilde x}, \langle\phi\rangle_{x_0+\rho\tilde x'}\big)
	\leq A_0 \|\phi\|_*(\rho|\tilde x-\tilde x'|)^\alpha
	= A_0 \rho^\alpha \, \|\phi\|_* \, |\tilde x-\tilde x'|^\alpha.
\]

Finally, the running cost $\ell_\rho$ satisfies~\ref{L1} with constant $C_0\rho^{\sigma}$ in place of $C_0$.
Indeed, for all $\tilde x \in B$, $\ell_\rho(\tilde x,\cdot) \in E$ with $\|\ell_\rho(\tilde x,\cdot)\|_* = \rho^\sigma\|\ell(x_0+\rho\tilde x,\cdot)\|_*$, and, for all $\tilde x,\tilde x' \in B$ and $\mu \in \A(x_0+\rho\tilde x)$,
\[
	\left|\overline\ell^\mu_\rho(\tilde x) - \overline\ell^\mu_\rho(\tilde x')\right|
	= \rho^\sigma \left|\overline\ell^\mu(x_0+\rho\tilde x)-\overline\ell^\mu(x_0+\rho\tilde x')\right|
	\leq \rho^{\sigma+\alpha}\, [\overline\ell^\mu]_{C^\alpha(B)}\, |\tilde x-\tilde x'|^\alpha,
\]
and $\rho^{\sigma+\alpha} \le \rho^\sigma$.

We may therefore apply Proposition~\ref{prop:x-dep-estimate} to $\tilde u$. Let $\nu := \lfloor\sigma+\alpha\rfloor$, and fix $\alpha' \in (0,\alpha)$ with $\nu<\sigma+\alpha'$. Since $K_0\rho^\alpha \le K_0$, we obtain
\begin{eqnarray*}
	[\tilde u]_{C^{\sigma+\alpha}(B_{1/4})}
	& \leq & C\big(1+A_0\rho^\alpha\big) \Big( \|\tilde u\|_{C^{\sigma+\alpha'}(B)} +
	K_0\rho^\alpha\, [\tilde u]_{C^{\sigma+\alpha}(B)} +
	\|\tilde u\|_{C^\alpha(\R^N)} + C_0\rho^{\sigma} \Big) \\
	& \leq & C\big(1+A_0\rho^\alpha\big) \Big( [\tilde u]_{C^{\sigma+\alpha'}(B)} +
	K_0\rho^\alpha\, [\tilde u]_{C^{\sigma+\alpha}(B)} +
	\|\tilde u\|_{C^\alpha(\R^N)} + C_0\rho^{\sigma} \Big),
\end{eqnarray*}
with $C$ independent of $x_0$ and $\rho$, where we used that $\|\tilde u\|_{C^{\sigma+\alpha'}(B)} \le C\big([\tilde u]_{C^{\sigma+\alpha'}(B)} + \|\tilde u\|_{L^\infty(B)}\big)$, by interpolation on $B$, and absorbed the $L^\infty$ term into $\|\tilde u\|_{C^\alpha(\R^N)}$.

Rescaling back to $u$, with $[\tilde u]_{C^{s}(B_r)} = \rho^{s}[u]_{C^{s}(B_{\rho r}(x_0))}$, we obtain
\begin{eqnarray*}
	\rho^{\sigma+\alpha}[u]_{C^{\sigma+\alpha}(B_{\rho/4}(x_0))}
	& \leq & C(1+A_0\rho^\alpha)\Big(\rho^{\sigma+\alpha'} [u]_{C^{\sigma+\alpha'}(B_\rho(x_0))} +
	K_0\rho^{\sigma+2\alpha}[u]_{C^{\sigma+\alpha}(B_\rho(x_0))} \\
	& & \qquad + \|u\|_{C^\alpha(\R^N)} + C_0\rho^{\sigma} \Big).
\end{eqnarray*}

\smallskip
Since $B_{2\rho}(x_0)\subset B$, every point of $B_\rho(x_0)$ is at distance at least $\rho$ from $\partial B$. Hence, by property~\ref{rmk:adim-seminorm} of Remark~\ref{rmk:adim}, applied with $\Omega = B$ and $\Omega' = B_\rho(x_0)$ to the two seminorms on the right-hand side of the previous estimate, and using $\rho<1$,
\[
	\rho^{\sigma+\alpha}[u]_{C^{\sigma+\alpha}(B_{\rho/4}(x_0))} \leq
	C(1+A_0)\Big([u]^*_{\sigma+\alpha';B} + K_0\, [u]^*_{\sigma+\alpha;B} + \|u\|_{C^\alpha(\R^N)} + C_0 \Big).
\]
This estimate holds for every $(x_0,\rho)$ with $B_{2\rho}(x_0)\subset B$. Taking the supremum and using property~\ref{rmk:adim-seminorm2} of Remark~\ref{rmk:adim}, we obtain
\begin{equation}\label{eq:adim-est}
	[u]^*_{\sigma+\alpha;B} \leq
	C(1+A_0)\Big([u]^*_{\sigma+\alpha';B} + K_0\, [u]^*_{\sigma+\alpha;B} + \|u\|_{C^\alpha(\R^N)} + C_0 \Big),
\end{equation}
 where $C$ depends only on $N$, $\sigma$, $\Lambda_0$, $\Lambda$, $K^*$, $\alpha$, and $\alpha'$ and, nondecreasingly, on $K_0$.

\smallskip
\emph{The case when $K_0$ and $A_0$ are small.}
Assume that $A_0 \le 1$ and $K_0 \le \min\{1, (8C)^{-1}\}$,  where $C$ is the constant in~\eqref{eq:adim-est} computed with $K_0 = 1$, so that it depends only on $N$, $\sigma$, $\Lambda_0$, $\Lambda$, $K^*$, $\alpha$ and $\alpha'$. Applying property~\ref{rmk:adim-interp} of Remark~\ref{rmk:adim} with $\varepsilon := 1/(8C)$, and using $1+A_0 \le 2$, the coefficient of $[u]^*_{\sigma+\alpha;B}$ on the right-hand side of~\eqref{eq:adim-est} is at most $2C(\varepsilon + K_0) \le 1/2$. Since $[u]^*_{\sigma+\alpha;B} \le [u]_{C^{\sigma+\alpha}(B)} < \infty$, this term can be absorbed into the left-hand side, and, using $\|u\|_{L^\infty(B)} \le \|u\|_{C^\alpha(\R^N)}$, we obtain
\[
	[u]^*_{\sigma+\alpha;B} \leq C\big(\|u\|_{C^\alpha(\R^N)} + C_0\big).
\]
Finally, by property~\ref{rmk:adim-seminorm} of Remark~\ref{rmk:adim}, applied with $\Omega = B$ and $\Omega' = B_{1/2}$, and by interpolation on $B_{1/2}$,
\begin{equation}\label{eq:thm1-small}
	\|u\|_{C^{\sigma+\alpha}(B_{1/2})} \leq C\big(2^{\sigma+\alpha}[u]^*_{\sigma+\alpha;B} + \|u\|_{L^\infty(B)}\big) \leq C\big(\|u\|_{C^\alpha(\R^N)} + C_0\big),
\end{equation}
with $C$ depending only on $N$, $\sigma$, $\Lambda_0$, $\Lambda$, $K^*$, $\alpha$ and $\alpha'$.

\smallskip
\emph{Conclusion in the general case.} Let
\[
	\rho_0 := \min\Big\{ {\big(\min\{1,(8C)^{-1}\}/(1+K_0)\big)^{1/\alpha}},\ (1+A_0)^{-1/\alpha},\ 1/2 \Big\},
\]
{with $C$ as above}. For $z \in B_{1/2}$, let $\bar u(\bar x) := u(z + \rho_0\bar x)$. Since $B_{\rho_0}(z) \subset B$, the verification carried out at the beginning of this proof, with $x_0 = z$ and $\rho = \rho_0$, shows that $\bar u$ solves in $B$ an equation of the same form, with the same constants $\Lambda_0$, $\Lambda$, $K^*$, and with $K_0\rho_0^\alpha$, $A_0\rho_0^\alpha$ and $C_0\rho_0^\sigma$ in place of $K_0$, $A_0$ and $C_0$. By the choice of $\rho_0$, we have {$K_0\rho_0^\alpha \le \min\{1,(8C)^{-1}\}$} and $A_0\rho_0^\alpha \le 1$, so that~\eqref{eq:thm1-small} applies to $\bar u$ and gives, since $\|\bar u\|_{C^\alpha(\R^N)} \le \|u\|_{C^\alpha(\R^N)}$ and $C_0\rho_0^\sigma \le C_0$,
\[
	\|\bar u\|_{C^{\sigma+\alpha}(B_{1/2})} \leq C\big(\|u\|_{C^\alpha(\R^N)} + C_0\big).
\]
Scaling back, and covering $B_{1/2}$ by the balls $B_{\rho_0/2}(z)$, $z\in B_{1/2}$, we obtain
\[
	\|u\|_{C^{\sigma+\alpha}(B_{1/2})} \leq C\rho_0^{-(\sigma+\alpha)}\big(\|u\|_{C^\alpha(\R^N)} + C_0\big).
\]
Since $\rho_0$ depends only on $K_0$, $A_0$ and the constants above, this proves Theorem~\ref{thm1}.
\end{proof}

\subsection{Equations with a zeroth-order term}\label{subsec:proper}
We now prove Corollary~\ref{cor:proper}, by reducing it to Theorem~\ref{thm1}.

\begin{proof}[Proof of Corollary~\ref{cor:proper}]
Moving the zeroth-order term into the running cost, we rewrite~\eqref{eq:properHJ-x} as
\[
	\inf_{\mu\in\mathcal A(x)}\Big\{\overline L_x^\mu u(x) + \overline{\tilde\ell\,}^\mu(x)\Big\} = 0 \quad \text{in } B,
\]
where $\overline{\tilde\ell\,}^\mu$ is the average of
\[
	\tilde\ell(x,\xi) := \ell(x,\xi) - u(x)\,\chi_X(\xi), \qquad x\in B,\ \xi\in X.
\]
Indeed, since every $\mu \in \PP(X)$ is a probability measure, $\overline{\tilde\ell\,}^\mu = \overline\ell^\mu - u$. This equation is of the form~\eqref{eq:effHJ-x}, with $\tilde\ell$ in place of $\ell$, and~\ref{K0}--\ref{K2} and~\ref{A1} are unaffected, since they do not involve the running cost.

It remains to check that $\tilde\ell$ satisfies~\ref{L1}. Since $u$ is bounded and continuous, $\tilde\ell \in C_b(B\times X)$, and, since $\chi_X \in E$, $\tilde\ell(x,\cdot) \in E$ for every $x\in B$, with
\[
	\sup_{x\in B}\|\tilde\ell(x,\cdot)\|_* \ \le \ \sup_{x\in B}\|\ell(x,\cdot)\|_* + \|\chi_X\|_*\,\|u\|_{L^\infty(\R^N)}.
\]
Moreover, for all $x,x'\in B$ and $\mu \in \mathcal A(x)$,
\[
	\left|\overline{\tilde\ell\,}^\mu(x) - \overline{\tilde\ell\,}^\mu(x')\right|
	\ \le \ \left|\overline\ell^\mu(x)-\overline\ell^\mu(x')\right| + |u(x)-u(x')|
	\ \le \ \big([\overline\ell^\mu]_{C^\alpha(B)} + [u]_{C^\alpha(\R^N)}\big)|x-x'|^\alpha.
\]
Hence $\tilde\ell$ satisfies~\ref{L1} with constant
\[
	\tilde C_0 := C_0 + \big(1+\|\chi_X\|_*\big)\|u\|_{C^\alpha(\R^N)}
\]
in place of $C_0$. Applying Theorem~\ref{thm1}, we obtain
\[
	\|u\|_{C^{\sigma+\alpha}(B_{1/2})} \ \le \ C\big(\|u\|_{C^\alpha(\R^N)} + \tilde C_0\big)
	\ \le \ C\big(2+\|\chi_X\|_*\big)\big(\|u\|_{C^\alpha(\R^N)} + C_0\big),
\]
which is the estimate of Corollary~\ref{cor:proper}, with $C(2+\|\chi_X\|_*)$ in place of $C$.
\end{proof}

\bigskip
\section{Existence of a classical solution for the effective problem}
\label{sec-existence}

In this section, we first prove the existence of a viscosity solution to~\eqref{eq:HJ-bar} under the following qualitative assumptions, which replace the H\"older regularity in~\ref{K1}, \ref{A1} and~\ref{L1} by mere continuity (Proposition~\ref{prop:existence}). We then prove, under the assumptions of Section~\ref{secAssumptions}, the existence and uniqueness of a classical solution $u \in C^{\sigma+\alpha}_b(\R^N)$ to~\eqref{eq:HJ-bar} (Proposition~\ref{prop:existence-Csigma}).

\begin{enumerate}[start=1,label=\textup{(K\arabic*$'$)},ref=\textup{(K\arabic*$'$)}]
\item \label{K1'} \emph{Continuity of the averaged kernel.}
	There exists a modulus of continuity $m_K:\R_+ \to \R_+$ such that, for all $\mu \in \PP(X)$, $x, x' \in \R^N$ and $r > 0$,
	\[
	\int_{B_{r}^c} \left|\overline K^\mu (x, z) - \overline K^\mu (x', z)\right| dz
	\leq \frac{m_K(|x - x'|)}{r^\sigma}.
	\]
\end{enumerate}

\begin{enumerate}[start=1,label=\textup{(A\arabic*$'$)},ref=\textup{(A\arabic*$'$)}]
\item \label{A1'} \emph{Continuity of $x \mapsto \A(x)$.}
	There exists a modulus of continuity $m_A:\R_+ \to \R_+$ such that, for all $x, x' \in \R^N$ and $\phi \in E$,
	\[
	d_H(\langle \phi \rangle_x, \langle \phi \rangle_{x'}) \leq m_A(|x - x'|)\, \| \phi \|_*.
	\]
\end{enumerate}

\begin{enumerate}[start=1,label=\textup{(L\arabic*$'$)},ref=\textup{(L\arabic*$'$)}]
\item \label{L1'} \emph{Continuity of the running cost.}
	For all $x \in \R^N$, $\ell(x, \cdot) \in E$, with $C^* := \sup_{x \in \R^N} \| \ell(x, \cdot) \|_* < \infty$, and there exists a modulus of continuity $m_\ell:\R_+ \to \R_+$ such that, for all $x,x'\in \R^N$ and $\mu \in \A(x)$,
	\[
	\left|\overline \ell^\mu(x) - \overline \ell^\mu(x')\right| \leq m_\ell(|x - x'|).
	\]
\end{enumerate}
Here, a modulus of continuity is a nondecreasing function $m:[0,\infty)\to[0,\infty)$ with $m(t)\to0$ as $t\to0$.

We start with an existence result for~\eqref{eq:HJ-bar} under these weaker assumptions.
\begin{prop}\label{prop:existence}
Assume that the family of kernels $\big(K(x,\cdot,\xi)\big)_{x \in \R^N, \xi \in X}$ satisfies~\ref{K0}, \ref{K1'} and~\ref{K2}, with $B$ replaced by $\R^N$ in~\ref{K0} and~\ref{K2}, that the family of measure sets $(\A(x))_{x \in \R^N}$ satisfies~\ref{A1'}, and that the running cost $\ell$ satisfies~\ref{L1'}. Then~\eqref{eq:HJ-bar} admits a viscosity solution $u \in C_b^{\gamma}(\R^N)$, for some $\gamma\in(0,1)$ depending only on $N$, $\sigma$, $\Lambda_0$ and $\Lambda$.
\end{prop}

\begin{proof}
\smallskip
\emph{Step 1 (Kernel truncation).}
Let $\eta \in C_c^\infty(B)$ be a radial cut-off function such that $0 \leq \eta \leq 1$, $\eta = 1$ in $B_{1/2}$ and $\eta = 0$ in $B^c$. For $\delta \in (0,1)$, set $\eta_\delta(z) := \eta(z/\delta)$, for all $z \in \R^N$.

We define $K_\delta: \R^N \times (\R^N \setminus \{ 0 \}) \times X \to \R_+$ by
\begin{equation}\label{eq:Keps-convex}
	K_\delta(x,z,\xi) := \eta_\delta(z)\, C_\Lambda|z|^{-(N+\sigma)}\chi_X(\xi) \;+\; (1-\eta_\delta(z))\, K(x,z,\xi),
	\qquad C_\Lambda := \frac{\Lambda_0 + \Lambda}{2}.
\end{equation}
Then $K_\delta(x,z,\xi)$ is a convex combination of the fixed, translation-invariant kernel $C_\Lambda|z|^{-(N+\sigma)}\chi_X(\xi)$, with weight $\eta_\delta(z)$, and of $K(x,z,\xi)$.
It is easy to check that $K_\delta(x,\cdot,\xi)\in \mathcal L_0(\sigma,\Lambda_0,\Lambda)$ and that $K_\delta$ inherits \ref{K0}, \ref{K1'} and \ref{K2} from $K$, with constants that are uniform in $\delta$.

\smallskip
\emph{Step 2 (The approximating problems).}
For $\delta \in (0,1)$, we consider the approximating problem
\begin{equation}\label{eq:HJ-eff-eps}
	u + \sup_{\mu \in \mathcal A(x)} \left\{ -\overline L_{x, \delta}^\mu u(x) - \overline \ell^\mu(x) \right\} = 0
	\quad \text{in } \R^N,
\end{equation}
where $\overline L^\mu_{x,\delta} := L_{\overline K_\delta^\mu(x,\cdot)}$, and $\overline K_\delta^\mu$ is the average of $K_\delta$ against $\mu\in\PP(X)$. In particular,
\[
	\overline K_\delta^\mu(x,z) \ = \
	C_\Lambda |z|^{-(N+\sigma)} \ + \
	(1-\eta_\delta(z)) \left( \overline K^\mu(x,z) - C_\Lambda |z|^{-(N+\sigma)} \right),
\]
so that, in $B_{\delta/2}$, all the averaged kernels coincide with the fixed kernel $C_\Lambda|z|^{-(N+\sigma)}$.

Splitting $\overline K^\mu_\delta$ accordingly, problem~\eqref{eq:HJ-eff-eps} can equivalently be written as
\begin{equation}\label{eq:HJ-eff-eps2}
	u - \mathcal I u(x) +
	\sup_{\mu \in \mathcal A(x)} \Big\{ -\overline {\mathcal J}_{x, \delta}^\mu u(x) - \overline\ell^\mu(x) \Big\} = 0
	\quad \text{in } \R^N,
\end{equation}
where
\begin{eqnarray*}
&&\mathcal I u(x)
		 := \frac{1}{2} \int_{\R^N} \delta^2u(x,z) \,\frac{C_\Lambda}{ |z|^{N + \sigma}}dz, \\
&& \overline {\mathcal J}_{x, \delta}^\mu u(x)
		 := \int_{\R^N} \left( u(x + z) - u(x)\right) \overline J_\delta^\mu(x,z)\,dz, \quad
		\overline J_\delta^\mu(x,z) := (1-\eta_\delta(z))
		\left(\overline K^\mu(x, z)-\frac{C_\Lambda}{|z|^{N+\sigma}} \right),
\end{eqnarray*}
and $\overline J^\mu_\delta$ is the average of
\[
	J_\delta(x,z,\xi) := (1-\eta_\delta(z))\left(K(x,z,\xi) - C_\Lambda |z|^{-(N+\sigma)}\chi_X(\xi)\right).
\]
Since $J_\delta$ vanishes for $|z| < \delta/2$, the integral defining $\overline{\mathcal J}^\mu_{x,\delta}$ converges absolutely. Moreover, by~\ref{K0} and~\ref{K2},
\[
	|J_\delta(x,z,\xi)| \le \Lambda\,|z|^{-(N+\sigma)}, \qquad
	\|J_\delta(x,z,\cdot)\|_* \leq \big(K^* + C_\Lambda\|\chi_X\|_*\big)\,|z|^{-(N+\sigma)}.
\]
In particular, $\overline{\mathcal J}^\mu_{x,\delta}$ is a bounded operator on $L^\infty(\R^N)$, with $|\overline{\mathcal J}^\mu_{x,\delta} v| \le C\delta^{-\sigma}\|v\|_{L^\infty(\R^N)}$ for a constant $C$ depending only on $N$, $\sigma$ and $\Lambda$.

\smallskip
\emph{Step 3 (Comparison principle and existence).}
Equation~\eqref{eq:HJ-eff-eps2} is the sum of the translation-invariant integro-differential operator $\mathcal I$ and of the bounded, $x$-dependent nonlocal operators $\overline{\mathcal J}^\mu_{x,\delta}$. Combining the arguments of Barles and Imbert~\cite{BI08}, for $\mathcal I$, with those of Alvarez and Tourin~\cite{AT96}, for zeroth-order, state-dependent nonlocal operators such as $\overline{\mathcal J}^\mu_{x,\delta}$, and using the continuity assumptions~\ref{K1'}, \ref{A1'} and~\ref{L1'} in $x$, equation~\eqref{eq:HJ-eff-eps} satisfies a comparison principle between discontinuous viscosity sub- and supersolutions.

We provide a sketch of the proof for completeness.
Let $u$ and $v$ be bounded, semicontinuous viscosity sub- and supersolutions to~\eqref{eq:HJ-eff-eps}, in the sense of~\cite{BI08}, and assume by contradiction that
\[
	M := \sup_{\R^N} \{ u - v \} > 0.
\]
Let $\psi \in C^2_b(\R^N)$ be a nonnegative, radially nondecreasing function with $\psi = 0$ in $B$ and $\psi = 2 + R$ in $B_2^c$, where $R:= \| u \|_\infty + \| v \|_\infty + 1$. For $\beta \in (0,1)$, let $\psi_\beta(x) := \psi(\beta x)$ and
\[
    M_\beta := \sup_{\R^N} \{ u - v - \psi_\beta \}.
\]
Then $M_\beta \to M$ as $\beta\to 0$, and we assume from now on that $\beta$ is small enough that $M_\beta \geq M/2$. By scaling, $\mathcal I\psi_\beta(x) = \beta^\sigma\,(\mathcal I\psi)(\beta x)$, so that
\( \sup_{x \in \R^N} |\mathcal I\psi_\beta(x)|  = \ o_\beta(1).\)

Following the classical doubling-of-variables technique, for $\rho \in (0,1)$,  let $\varphi(x,y) := \rho^{-1}|x-y|^2$ and consider
\[
	M_{\rho, \beta} := \sup_{x, y \in \R^N} \left\{ u(x) - v(y) - \varphi(x,y) - \psi_{\beta}(x) \right\} \ \geq \ M_\beta.
\]
Since $u - v$ is bounded and upper semicontinuous, and $\psi_\beta \geq 2+R$ outside $B_{2/\beta}$, this supremum is attained at some $(\bar x, \bar y)$. By standard arguments, as $\rho \to 0$ with $\beta$ fixed,
\begin{equation}\label{rho-conv}
	\rho^{-1}|\bar x - \bar y|^2 \to 0 \qquad \text{and} \qquad M_{\rho, \beta} \to M_\beta.
\end{equation}
By maximality of $(\bar x,\bar y)$, the function $x \mapsto \varphi(x,\bar y) + \psi_\beta(x)$ touches $u$ from above at $\bar x$, and $y \mapsto -\varphi(\bar x,y)$ touches $v$ from below at $\bar y$, globally, up to additive constants{ .}

For $\tau \in (0,\delta/2)$, let
\begin{align*}
	u^{\tau} & := u\, \chi_{B_\tau(\bar x)^c} + \big(u(\bar x) + \varphi(\cdot,\bar y) - \varphi(\bar x,\bar y) + \psi_{\beta} - \psi_\beta(\bar x)\big)\, \chi_{B_\tau(\bar x)}, \\
	v^{\tau} & := v\, \chi_{B_\tau(\bar y)^c} + \big(v(\bar y) - \varphi(\bar x, \cdot) + \varphi(\bar x,\bar y)\big)\, \chi_{B_\tau(\bar y)},
\end{align*}
so that $u^\tau(\bar x) = u(\bar x)$ and $v^\tau(\bar y) = v(\bar y)$.
The viscosity inequalities, in the sense of~\cite{BI08}, then give
\begin{align*}
	u(\bar x) - \mathcal I u^{\tau} (\bar x) +
	\sup_{\mu \in \mathcal A(\bar x)}\big\{ -\overline{\mathcal J}_{\bar x, \delta}^\mu u^{\tau}(\bar x) - \overline \ell^\mu(\bar x) \big\} & \leq 0, \\
	v(\bar y) - \mathcal I v^{\tau} (\bar y) +
	\sup_{\mu \in \mathcal A(\bar y)}\big\{ -\overline{\mathcal J}_{\bar y, \delta}^\mu v^{\tau}(\bar y) - \overline \ell^\mu(\bar y) \big\} & \geq 0.
\end{align*}
Subtracting these inequalities, and using that $u(\bar x) - v(\bar y) \geq M_{\rho,\beta} \geq M_\beta$, we obtain
\begin{equation}\label{eq:visc-est}
	\mathcal N + \mathcal T_1 + \mathcal T_2 \ \geq \ M_\beta/2,
\end{equation}
where
\begin{align*}
	\mathcal N & := \mathcal I u^\tau(\bar x) - \mathcal I v^\tau(\bar y),\\
	\mathcal T_1 & :=
	\inf_{\mu \in \mathcal A(\bar x)}\Big\{\overline{\mathcal J}_{\bar x, \delta}^\mu u^\tau(\bar x)
		+ \overline \ell^\mu(\bar x)\Big\} -
	\inf_{\mu \in \mathcal A(\bar x)}\Big\{\overline{\mathcal J}_{\bar y, \delta}^\mu v^\tau (\bar y)
		+ \overline \ell^\mu(\bar y)\Big\}, \\
	\mathcal T_2 & :=
	\inf_{\mu \in \mathcal A(\bar x)}\Big\{\overline{\mathcal J}_{\bar y, \delta}^\mu v^\tau(\bar y)
		+ \overline \ell^\mu(\bar y)\Big\} -
	\inf_{\mu \in \mathcal A(\bar y)}\Big\{\overline{\mathcal J}_{\bar y, \delta}^\mu v^\tau (\bar y)
		+ \overline \ell^\mu(\bar y)\Big\}.
\end{align*}
Since $\overline J_\delta^\mu(x,\cdot)$ vanishes on $B_{\delta/2}$, while $u^\tau = u$ and $v^\tau = v$ outside $B_\tau(\bar x)$ and $B_\tau(\bar y)$, respectively, and $\tau<\delta/2$, we have,
$\overline{\mathcal J}_{\bar x,\delta}^\mu u^\tau(\bar x) = \overline{\mathcal J}_{\bar x,\delta}^\mu u(\bar x),$
$\overline{\mathcal J}_{\bar y,\delta}^\mu v^\tau(\bar y) = \overline{\mathcal J}_{\bar y,\delta}^\mu v(\bar y),$
for every $\mu \in \PP(X)$. In particular, $\mathcal T_1$ and $\mathcal T_2$ do not depend on $\tau$, and we drop the superscript $\tau$ in these two terms.

\smallskip
\emph{Estimate of $\mathcal N$.} We write $\mathcal N = \mathcal N_1 + \mathcal N_2 + \mathcal N_3$, where $\mathcal N_1$, $\mathcal N_2$ and $\mathcal N_3$ denote the contributions of $B_\tau$, $B_{\delta/2}\setminus B_\tau$ and $B_{\delta/2}^c$, respectively, to $\mathcal I u^\tau(\bar x) - \mathcal I v^\tau(\bar y)$. Proceeding as in~\cite{BI08}, we have $\mathcal N_1 \leq \rho^{-1}o_\tau(1)$.

On $B_\tau^c$, since $u^\tau(\bar x) = u(\bar x)$ and $v^\tau(\bar y) = v(\bar y)$, the second differences of $u^\tau$ and $v^\tau$ coincide with those of $u$ and $v$. Moreover, testing the maximality of $(\bar x, \bar y)$ at $(\bar x \pm z, \bar y \pm z)$ and adding, we obtain
\begin{equation}\label{eq:max-ineq}
	\delta^2 u(\bar x,z) - \delta^2 v(\bar y,z) \ \leq \ \delta^2\psi_\beta(\bar x,z) \qquad \text{for every } z \in \R^N,
\end{equation}
so that $\mathcal N_2  { \le} o_\beta(1)$. Hence
\begin{equation}\label{eq:est-N}
	\mathcal N \ \leq \ \rho^{-1} \,o_\tau(1) \ + \ o_\beta(1) \ + \ \mathcal N_3,
	\qquad
	\mathcal N_3 := \frac{C_\Lambda}{2}\int_{B_{\delta/2}^c} \big(\delta^2 u(\bar x,z) - \delta^2 v(\bar y,z)\big)\, |z|^{-(N+\sigma)}\,dz.
\end{equation}
The term $\mathcal N_3$ could be estimated in the same way, but we keep it as it is, since it cancels exactly with a term arising in $\mathcal T_1$, whose kernel $\overline J^\mu_\delta$ changes sign.

\smallskip
\emph{Estimate of $\mathcal T_1$.}
For $k\in\N$, let $\mu_k\in\mathcal A(\bar x)$ be such that
\[
\overline{\mathcal J\;}_{\bar y,\delta}^{\mu_k} v(\bar y) + \overline \ell^{\mu_k}(\bar y)
\ \leq \
\inf_{\mu \in \mathcal A(\bar x)}\Big\{\overline{\mathcal J\;}_{\bar y,\delta}^\mu v(\bar y) + \overline \ell^\mu(\bar y)\Big\} + k^{-1}.
\]
Since $\mu_k \in \mathcal A(\bar x)$, testing the first infimum in the definition of $\mathcal T_1$ at $\mu = \mu_k$ gives
\[
\mathcal T_1 \ \leq \
\overline{\mathcal J\;}_{\bar x,\delta}^{\mu_k} u(\bar x) - \overline{\mathcal J\;}_{\bar y,\delta}^{\mu_k} v(\bar y)
+ \overline \ell^{\mu_k}(\bar x) - \overline \ell^{\mu_k}(\bar y) + k^{-1}.
\]
By symmetry of $\overline J_\delta^{\mu_k}(\bar x,\cdot)$ and $\overline J_\delta^{\mu_k}(\bar y,\cdot)$,
\begin{equation*}
\begin{split}
\overline{\mathcal J\;}_{\bar x,\delta}^{\mu_k} u(\bar x) - \overline{\mathcal J\;}_{\bar y,\delta}^{\mu_k} v(\bar y)
\ = \
& \frac12\int_{\R^N} \big[\delta^2 u(\bar x,z) - \delta^2 v(\bar y,z)\big]\, \overline J_\delta^{\mu_k}(\bar x,z)\,dz \\
\ + \
& \frac12\int_{\R^N} \delta^2 v(\bar y,z) \Big(\overline J_\delta^{\mu_k}(\bar x,z) - \overline J_\delta^{\mu_k}(\bar y,z)\Big)dz,
\end{split}
\end{equation*}
and, since $\overline J_\delta^{\mu_k}(\bar x,\cdot)$ and $\overline J_\delta^{\mu_k}(\bar y,\cdot)$ both vanish on $B_{\delta/2}$, both integrals above are in fact supported on $B_{\delta/2}^c$.

For the first term, we cannot use~\eqref{eq:max-ineq} directly since the kernel $\overline J_\delta^{\mu_k}$ changes sign. Instead, we write $\overline J_\delta^{\mu_k}(\bar x,z) = \overline K^{\mu_k}_\delta(\bar x,z) - C_\Lambda |z|^{-(N+\sigma)}$, where $\overline K^{\mu_k}{ _\delta} \geq 0${ .} Using~\eqref{eq:max-ineq}, we get
\begin{align*}
	\frac12\int_{B_{\delta/2}^c} \big(\delta^2 u(\bar x,z) - \delta^2 v(\bar y,z)\big)\, \overline J_\delta^{\mu_k}(\bar x,z)\,dz
	\ & = \
	\frac12\int_{B_{\delta/2}^c} \big(\delta^2 u(\bar x,z) - \delta^2 v(\bar y,z)\big)\, \overline K_\delta^{\mu_k}(\bar x,z)\,dz  -  \mathcal N_3 \\
	& \leq \
	\frac{\Lambda}{2} \int_{\R^N} |\delta^2 \psi_\beta(\bar x,z)|\, |z|^{-(N+\sigma)}\,dz  -  \mathcal N_3
	\ = \ o_\beta(1)  -  \mathcal N_3.
\end{align*}

For the second term, we use that the fixed kernel cancels in the difference, namely
\[
	\overline J_\delta^{\mu_k}(\bar x,z) - \overline J_\delta^{\mu_k}(\bar y,z) \ = \ (1 - \eta_\delta(z)) \left(\overline K^{\mu_k}(\bar x,z) - \overline K^{\mu_k}(\bar y,z)\right),
\]
together with $0 \leq 1 - \eta_\delta \leq 1$, $|\delta^2 v(\bar y,z)| \leq 4\|v\|_\infty$, and~\ref{K1'} with $r = \delta/2$. In view of~\eqref{rho-conv}, this gives
\begin{align*}
	\frac12\int_{B_{\delta/2}^c} \delta^2 v(\bar y,z) \left(\overline J_\delta^{\mu_k}(\bar x,z) - \overline J_\delta^{\mu_k}(\bar y,z)\right)dz
	& \leq
	2\|v\|_\infty \int_{B_{\delta/2}^c} \left|\overline K^{\mu_k}(\bar x,z) - \overline K^{\mu_k}(\bar y,z)\right|dz \\
	& \leq
	2\|v\|_\infty\, \frac{m_K(|\bar x - \bar y|)}{(\delta/2)^\sigma} \ = \ \delta^{-\sigma} o_\rho(1).
\end{align*}
Finally, \ref{L1'} gives $\overline \ell^{\mu_k}(\bar x) - \overline \ell^{\mu_k}(\bar y) \le m_\ell(|\bar x - \bar y|) = o_\rho(1)$. Combining these estimates, we conclude that
\begin{equation}\label{eq:est-T1-prop}
	\mathcal T_1 \ \leq \ o_\beta(1) \ + \ \delta^{-\sigma}\,o_\rho(1) \ + \ k^{-1} \ - \ \mathcal N_3.
\end{equation}

\smallskip
\emph{Estimate of $\mathcal T_2$.}
Let
\[
	\phi(\xi) \ := \ \int_{\R^N} \big(v(\bar y+z)-v(\bar y)\big)\, J_\delta(\bar y,z,\xi)\,dz \ + \ \ell(\bar y,\xi), \qquad \xi \in X,
\]
so that $\overline\phi^\mu = \overline{\mathcal J}_{\bar y,\delta}^\mu v(\bar y) + \overline\ell^\mu(\bar y)$ for every $\mu \in \PP(X)$. By~\ref{K2} and~\ref{L1'}, $\phi \in E$ and
\[
	\|\phi\|_*
	\ \leq \ \int_{B_{\delta/2}^c} |v(\bar y+z)-v(\bar y)|\, \|J_\delta(\bar y,z,\cdot)\|_*\,dz \ + \ \|\ell(\bar y,\cdot)\|_*
	\ \leq \ C_1 \|v\|_\infty\, \delta^{-\sigma} \ + \ C^*,
\]
for a constant $C_1>0$ depending only on $N$, $\sigma$, $\Lambda$, $K^*$ and $\|\chi_X\|_*$. Furthermore, \ref{A1'} gives
\begin{align*}
	\mathcal T_2
	& = \ \inf \langle \phi \rangle_{\bar x} \ - \ \inf \langle \phi \rangle_{\bar y}
	\ \leq \ d_H\big(\langle \phi \rangle_{\bar x}, \langle \phi \rangle_{\bar y}\big) \ \leq \ \|\phi\|_*\, m_A(|\bar x-\bar y|),
\end{align*}
and therefore, by~\eqref{rho-conv},
\begin{equation}\label{eq:est-T2-prop}
	\mathcal T_2 \ \leq \ \big(C_1\|v\|_\infty\,\delta^{-\sigma}+C^*\big)\, m_A(|\bar x-\bar y|)
	\ = \ \delta^{-\sigma}\,o_\rho(1).
\end{equation}

\smallskip
\emph{Conclusion.} Combining~\eqref{eq:est-N}, \eqref{eq:est-T1-prop} and~\eqref{eq:est-T2-prop} with~\eqref{eq:visc-est}, the terms $\pm \mathcal N_3$ cancel out, and we obtain
\[
	0 \ < \ M_\beta/2\ \leq \ \rho^{-1}\,o_\tau(1) \ + \ o_\beta(1) \ + \ \delta^{-\sigma}\,o_\rho(1) \ + \ k^{-1}.
\]
Letting first $\tau \to 0$, then $k \to \infty$, and then $\rho \to 0$, with $\beta,\delta$ fixed, we get $M_\beta \leq 2\,o_\beta(1)$. Since $M_\beta \geq M/2$, letting $\beta \to 0$ gives $M \leq 0$, a contradiction. This proves the comparison principle between discontinuous viscosity sub- and supersolutions of~\eqref{eq:HJ-eff-eps}.

\smallskip
\emph{Step 4 (Passage to the limit as $\delta \to 0$).}
Since the constants $\pm C^*$ are, respectively, a super- and a subsolution of~\eqref{eq:HJ-eff-eps}, Perron's method, together with the comparison principle, yields a viscosity solution $u^\delta \in C_b(\R^N)$ of~\eqref{eq:HJ-eff-eps}, satisfying $-C^* \leq u^\delta \leq C^*$. Moreover, since $|\overline\ell^\mu| \leq C^*$ and $\overline K^\mu_\delta(x,\cdot) \in \mathcal L_0(\sigma,\Lambda_0,\Lambda)$, the function $u^\delta$ satisfies, in the viscosity sense,
\[
	\MM_{\mathcal L_0}^+ u^\delta \geq -2C^*, \qquad \MM_{\mathcal L_0}^- u^\delta \leq 2C^* \qquad \text{in } \R^N.
\]
By the results of~\cite{CS09}, we conclude that $u^\delta \in C^{\gamma}(\R^N)$, for some $\gamma \in (0,1)$ depending only on $N$, $\sigma$, $\Lambda_0$ and $\Lambda$, with
\[
	-C^* \leq u^\delta \leq C^*, \qquad [u^\delta]_{C^{\gamma}(\R^N)} \leq C,
\]
for a constant $C > 0$ depending only on $N$, $\sigma$, $\Lambda_0$, $\Lambda$ and $C^*$, and in particular independent of $\delta$.

By the uniform bounds just established, and the Arzel\`a--Ascoli theorem, there exist a sequence $\delta_j \to 0$ and $u \in C_b^{\gamma}(\R^N)$ such that $u^{\delta_j} \to u$ locally uniformly in $\R^N$, with
\[
	-C^* \leq u \leq C^*, \qquad [u]_{C^{\gamma}(\R^N)} \leq C,
\]
by lower semicontinuity of the H\"older seminorm under local uniform convergence.

Finally, $u$ is a viscosity solution of~\eqref{eq:HJ-bar}, in view of the stability of viscosity solutions under local uniform convergence, see Barles and Imbert~\cite{BI08}. Note, however, that neither uniqueness for~\eqref{eq:HJ-bar} nor convergence of the whole family $(u^\delta)_\delta$ is obtained at this level of generality.

\end{proof}

With Proposition~\ref{prop:existence} and Corollary~\ref{cor:proper} at hand, we can now prove the existence and uniqueness of a classical solution of~\eqref{eq:HJ-bar}, which in particular gives the existence and uniqueness part of Theorem~\ref{thm:rates-homog}.

\begin{prop}\label{prop:existence-Csigma}
Let $\sigma \in (0,2)$, let $\bar\alpha \in (0,1)$ be given by Theorem~\ref{thm1}, and let $\alpha \in (0,\bar\alpha)$ with $\sigma+\alpha \notin \N$. Assume that the family of kernels $\big(K(x,\cdot,\xi)\big)_{x \in \R^N, \xi \in X}$ satisfies~\ref{K0}--\ref{K2}, the family of measure sets $(\A(x))_{x \in \R^N}$ satisfies~\ref{A1}, and the running cost $\ell$ satisfies~\ref{L1}, with $B$ replaced by $\R^N$. Then~\eqref{eq:HJ-bar} admits a unique classical solution $u \in C^{\sigma+\alpha}_b(\R^N)$, and
\[
	\| u \|_{C^{\sigma+\alpha}(\R^N)} \ \leq \ C,
\]
for a constant $C>0$ depending only on $N$, $\sigma$, $\Lambda_0$, $\Lambda$, $K_0$, $K^*$, $A_0$, $C_0$, $\|\chi_X\|_*$ and $\alpha$.
\end{prop}

\begin{proof}
\emph{Step 1 (Regularity of the approximating solutions).}
The assumptions imply~\ref{K1'}, \ref{A1'} and~\ref{L1'}, so that Proposition~\ref{prop:existence} applies: for every $\delta \in (0,1)$, problem~\eqref{eq:HJ-eff-eps} has a unique viscosity solution $u^\delta \in C^{\gamma}_b(\R^N)$. By~\eqref{eq:HJ-eff-eps2}, $u^\delta$ solves
\[
	\mathcal I u^\delta = f_\delta  \quad \text{in } \R^N, \qquad \text{where} \quad
	f_\delta := u^\delta +
	\sup_{\mu \in \A(x)} \big\{ -\overline{\mathcal J}^\mu_{x,\delta} u^\delta - \overline\ell^\mu \big\}.
\]
Since $\overline{\mathcal J}^\mu_{x,\delta}$ is bounded, \ref{K1}, \ref{A1} and~\ref{L1} give $f_\delta \in C^{\gamma_0}_b(\R^N)$, with $\gamma_0 := \min\{\gamma, \alpha\}$, and the Schauder estimates of~\cite{CS09} give $u^\delta \in C^{\sigma+\gamma_0}_b(\R^N)$. Iterating finitely many times, we obtain $u^\delta \in C^{\sigma+\alpha}_b(\R^N)$, with a bound depending on $\delta$.

\smallskip
\emph{Step 2 (Uniform estimate and passage to the limit).}
Since $u^\delta \in C^{\sigma+\alpha}_b(\R^N)$, Corollary~\ref{cor:proper}, applied on every ball $B_1(x_0)$, gives $\|u^\delta\|_{C^{\sigma+\alpha}(\R^N)} \leq C$, with $C$ independent of $\delta$.
By the Arzel\`a--Ascoli theorem, up to a subsequence, $u^{\delta} \to u$ in $C^{\sigma+\alpha''}_{loc}(\R^N)$ for every $\alpha'' < \alpha$, with $\|u\|_{C^{\sigma+\alpha}(\R^N)} \leq C$. By stability~\cite{BI08}, $u$ is a viscosity solution of~\eqref{eq:HJ-bar} and, since $u \in C^{\sigma+\alpha}_b(\R^N)$, it is a classical solution.

\smallskip
\emph{Step 3 (Uniqueness).}
Let $u, v \in C^{\sigma+\alpha}_b(\R^N)$ be two classical solutions of~\eqref{eq:HJ-bar}, and assume by contradiction that $M := \sup_{\R^N}(u-v) > 0$. With $\psi_\beta$ as in the proof of Proposition~\ref{prop:existence}, for $\beta$ small, $u - v - \psi_\beta$ attains its maximum $M_\beta \ge M/2$ at some $\bar x \in \R^N$. Then $\delta^2(u - v - \psi_\beta)(\bar x, z) \leq 0$ for all $z$, so that, by~\ref{K0} and scaling, $\overline L^\mu_{\bar x}(u-v)(\bar x) \leq \overline L^\mu_{\bar x} \psi_\beta(\bar x) \leq C \beta^\sigma$ for all $\mu \in \A(\bar x)$. Subtracting the equations for $u$ and $v$ at $\bar x$, which involve the same index set $\A(\bar x)$, we get
\[
 M/2 \ \leq \ u(\bar x) - v(\bar x) \ \leq \ \sup_{\mu \in \A(\bar x)} \overline L^\mu_{\bar x}(u-v)(\bar x) \ \leq \ C\beta^\sigma,
\]
a contradiction for $\beta$ small. In particular, since every subsequential limit of $(u^\delta)_\delta$ is a classical solution of~\eqref{eq:HJ-bar}, by Step~2, the whole family converges to $u$.
\end{proof}

\bigskip
\section{Application to rates of convergence in periodic homogenization}
\label{sec-convergence}

This section is devoted to the proof of Theorem~\ref{thm:rates-homog}. Recall from Subsection~\ref{secApplications} that $X = \T^N \times \Theta$ and that $E$ is given by~\eqref{defE}. The proof relies on the $C^{\sigma+\alpha}$ regularity of the solution $u$ of the effective problem~\eqref{eq:HJ-bar}, given by Proposition~\ref{prop:existence-Csigma}.
To apply it, it remains to show that the sets $\A(x)$, defined below as polar cones, satisfy~\ref{A1}. Since $\A(x)$ is only defined implicitly, through the ergodic problem frozen at $x$, this is the main new step of this section, carried out in Lemma~\ref{lem:A-reg}. It is a technical step, which requires a quantitative control of the effective Hamiltonian $\overline H(x,\phi)$ with respect to $x$, uniformly in $\phi\in E$.
The rest of the proof follows the lines of Camilli and Marchi~\cite{CM09} for second-order problems, adapted to the nonlocal setting as in~\cite{RPT23}.

\subsection{The corrector equation and the effective problem}\label{subsec:derivation}

We briefly recall, following the classical two-scale expansion of Lions, Papanicolaou and Varadhan~\cite{LPV86} and its nonlocal adaptation in~\cite{RPT23}, how the effective problem~\eqref{eq:HJ-bar} formally arises from~{\eqref{eq:HJ-eps}}. We seek an expansion of the solution $u^\epsilon$ of~{\eqref{eq:HJ-eps}} of the form
\begin{equation}\label{ansatz}
	u^\epsilon(x) \ = \ u(x) \ + \ \epsilon^\sigma w\Big(\frac{x}{\epsilon}\Big) \ + \ o(\epsilon^\sigma), \qquad \epsilon \to 0,
\end{equation}
for some limiting profile $u \in C_b(\R^N)$ and some $\T^N$-periodic \emph{corrector} $w = w(\cdot\,;x) \in C(\T^N)$, to be determined. By~\ref{H3}, ${ K^\theta(x,y,\cdot)}$ is positively homogeneous of degree $-(N+\sigma)$, so that the change of variables $z = \epsilon z'$ in~{\eqref{eq:Lxy}} gives
\[
	L^\theta_{x,\frac{x}{\epsilon}} \Big( \epsilon^\sigma w\big(\tfrac{\cdot}{\epsilon}\big) \Big)(x) \ = \ L^\theta_{x,\frac{x}{\epsilon}} w\Big( \frac{x}{\epsilon} \Big):
\]
the factor $\epsilon^\sigma$ exactly compensates the order $\sigma$ of $L^\theta_{x,y}$, which is why this exponent appears in~\eqref{ansatz}. Substituting~\eqref{ansatz} into~{\eqref{eq:HJ-eps}} and matching the terms of order $O(1)$, with $x$ frozen and $y = x/\epsilon$ ranging over $\T^N$, we are led to
\[
	u(x) + \sup_{\theta \in \Theta} \Big\{ -L^\theta_{x,y} u(x) - L^\theta_{x,y} w(y) - \ell^\theta(x,y) \Big\} = 0, \qquad y \in \T^N.
\]

This equation couples the slow variable $x$ and the fast variable $y$, and is solved in two steps. First, for fixed $x\in\R^N$ and $\varphi \in C_b^{\sigma+\iota}(\R^N)$, for some $\iota > 0$, we consider the \emph{cell problem}: find the unique constant $c$, called the \emph{ergodic constant}, for which there exists a periodic corrector $w = w(\cdot\,;x,\varphi) \in C(\T^N)$ solving, in the viscosity sense,
\begin{equation}\label{eq:cell-pb}
	\sup_{\theta \in \Theta} \big\{ - L^\theta_{x,y} {w}(y) - \tilde \ell^\theta(y;x,\varphi) \big\}  = c
	\quad \text{in } \T^N,
	\qquad \text{where} \quad
	\tilde \ell^\theta(y;x,\varphi) := \ell^\theta(x,y) + L^\theta_{x,y}\varphi(x).
\end{equation}
The ergodic constant defines the \emph{effective Hamiltonian} $\overline H(x,\varphi) := c$. Second, the equation above, with $\varphi = u$, requires this constant to be $-u(x)$, that is, $u$ must solve the \emph{effective problem}
\begin{equation}\label{eq:eff-pb}
	u(x) + \overline H(x,u) \ = \ 0 \qquad \text{in } \R^N.
\end{equation}
Finally, $\overline H(x,\varphi)$ admits a variational formula in terms of the sets $\A(x)$, given in Theorem~\ref{thm-eff-H} below, which turns~\eqref{eq:eff-pb} into~\eqref{eq:HJ-bar}.

\subsection{Effective Hamiltonian and viscosity Mather measures \texorpdfstring{$\mathcal A(x)$}{A(x)}}\label{subsec:Ax-holder}
Following the viscosity Mather measures approach of Ishii, Mitake and Tran~\cite{IMT17}, as adapted to the nonlocal setting by Rodr\'iguez-Paredes and Topp~\cite{RPT23}, the effective Hamiltonian $\overline H(x,\varphi)$ admits a variational representation. Recall that $X = \T^N\times \Theta$, that $E$ is given by~\eqref{defE}, and that the sets $\mathcal G(x)$ and $\A(x)$ are defined in~\eqref{defG} and~\eqref{defA}.

\begin{thm}[{\cite[Theorem~2.1]{RPT23}}]\label{thm-eff-H}
Assume that $K$ satisfies~\ref{H0}-\ref{H2}, and that $\ell$ satisfies~\ref{H4} and~\ref{H5}. Then, for every $x \in \R^N$ and $\varphi \in C_b^{\sigma+\iota}(\R^N)$, with $\iota>0$,
\begin{equation}\label{eq:eff-H}
	\overline H(x,\varphi) = \sup_{\mu \in \A(x)}
	\left( -\overline L^\mu_x \varphi(x) - \overline \ell^\mu(x) \right),
\end{equation}
where $\overline L^\mu_x := L_{\overline K^\mu(x,\cdot)}$, and $\overline K^\mu$ and $\overline \ell^\mu$ are the averages of $K$ and $\ell$ against $\mu$, as in~\eqref{eq:mean-f-mu}.
\end{thm}

To prove~\ref{A1}, we need to control the sets $\langle \phi\rangle_x$ for every $\phi \in E$, and not only for the particular data $\tilde\ell(\cdot\,;x,\varphi)$. For this reason, we consider the cell problem~\eqref{eq:cell-pb} with a general datum $\phi \in E$, namely
\begin{equation}\label{eq:cell-pb-phi}
	\sup_{\theta \in \Theta} \big\{ - L^\theta_{x,y} w (y) - \phi(y,\theta) \big\}  = c
	\quad \text{in } \T^N,
\end{equation}
and denote by $\overline H(x,\phi)$ the corresponding ergodic constant. Since the proof of Theorem~\ref{thm-eff-H} only uses that $\tilde\ell(\cdot\,;x,\varphi) \in E$, and not its specific form, we have
\begin{equation}\label{eq:eff-H-inf}
	\overline H(x,\phi)  = \sup_{\mu\in\A(x)}\left(-\overline\phi^\mu\right) =  - \min_{\mu\in\A(x)} \overline\phi^\mu = -\min \langle\phi\rangle_x.
\end{equation}
Here the minimum is attained, since $\A(x)$ is weak-$*$ compact, as a closed subset of $\PP(X)$ with $X$ compact, and $\mu \mapsto \overline\phi^\mu$ is weak-$*$ continuous.

Equivalently, $\overline H(x,\phi)$ is the smallest constant for which~\eqref{eq:cell-pb-phi} admits a subsolution:
\begin{equation}\label{eq:eff-H-inf-char}
	\overline H(x,\phi) \ = \ \inf \Big\{ \lambda \in \R \ : \
	\sup_{\theta \in \Theta} \big\{ -L^\theta_{x,y}w(y) - \phi(y,\theta) \big\} \le \lambda \ \text{in } \T^N \text{ for some } w\in C(\T^N) \Big\},
\end{equation}
where the inequality is understood in the viscosity sense. It is this characterization, rather than~\eqref{eq:eff-H-inf}, that we use in the proof of Lemma~\ref{lem:A-reg}.

The following well-posedness result is classical, see~\cite{RPT23} and~\cite{BCCI14}.

\begin{thm}[Solvability of the cell problem]\label{thm:cell-phi}
Let $\alpha \in (0,\bar\alpha)$, with $\bar\alpha$ given by Theorem~\ref{thm1} and $\sigma+\alpha\notin\N$, and assume that $K$ satisfies~\ref{H0}-\ref{H2}. Then, for every $x \in \R^N$ and $\phi \in E$, there exists a unique constant $c = \overline H(x,\phi) \in \R$ for which~\eqref{eq:cell-pb-phi} admits a periodic viscosity solution $w = w(\cdot\,; x, \phi)$. This solution is unique up to additive constants, belongs to $C^{\sigma+\alpha}(\T^N)$, and satisfies
\[
	 \| w(\cdot\,; x, \phi) -  w(0; x, \phi) \|_{C^{\sigma+\alpha}(\T^N)} \leq C\|\phi\|_*,
\]
for a constant $C > 0$ depending only on $N$, $\sigma$, $\Lambda_0$, $\Lambda$, $K_0$ and $\alpha$.
\end{thm}

We now show that the sets $\A(x)$ satisfy~\ref{A1}.

\begin{lema}\label{lem:A-reg}
Let $\alpha \in (0,\bar\alpha)$, with $\bar\alpha$ given by Theorem~\ref{thm1} and $\sigma+\alpha\notin\N$, and assume that $K$ satisfies~\ref{H0}-\ref{H2}. Then the family $(\A(x))_{x \in \R^N}$ given by~\eqref{defA} satisfies~\ref{A1} in $\R^N$: there exists a constant $A_0 > 0$, depending only on $N$, $\sigma$, $\Lambda_0$, $\Lambda$, $K_0$ and $\alpha$, such that, for all $\phi \in E$ and $x, x' \in \R^N$,
\[
	d_H(\langle \phi \rangle_{x}, \langle \phi \rangle_{x'}) \leq A_0 \| \phi \|_* |x - x'|^\alpha.
\]
\end{lema}

\begin{proof}
We first note that, for every $w \in C_b^{\sigma+\alpha}(\R^N)$, by~\ref{H2} and Lemma~\ref{lem:near-far}, with $r = 1$,
\begin{align*}
	\big|L_{x, y}^\theta w(y) - L_{x', y}^\theta w(y) \big|
	\leq & \ \frac{1}{2} \int_{\R^N} |\delta^2 w(y, z)|\,\big|{ K^\theta(x,y,z)} - { K^\theta(x',y,z)}\big|\,dz \\
	\leq & \ C_1 K_0\, \| w \|_{C^{\sigma + \alpha}(\R^N)} \, |x - x'|^\alpha,
\end{align*}
for a constant $C_1 > 0$ depending only on $N$, $\sigma$ and $\alpha$.

\smallskip
\emph{Step 1 (H\"older continuity of $x \mapsto \overline H(x,\phi)$).}
Fix $\phi \in E$ and, for $x \in \R^N$, let $w(\cdot\,;x,\phi)$ be the solution of~\eqref{eq:cell-pb-phi} with $c = \overline H(x,\phi)$, normalized so that $w(0;x,\phi) = 0$. For $x, x' \in \R^N$, applying the estimate above to $w(\cdot\,;x',\phi)$, and then Theorem~\ref{thm:cell-phi}, we get
\begin{align*}
	\big|L_{x, y}^\theta w(y;x',\phi) - L_{x', y}^\theta w(y;x',\phi) \big|
	\ \leq \ C_1 K_0 \, \| w(\cdot\,;x',\phi) \|_{C^{\sigma+\alpha}(\T^N)} \, |x - x'|^\alpha \\
	\ \leq \ A_0 \, \| \phi \|_* \, |x - x'|^\alpha,
\end{align*}
with $A_0 := C C_1 K_0$, for all $y \in \T^N$ and $\theta \in \Theta$. Since $w(\cdot\,;x',\phi) \in C^{\sigma+\alpha}(\T^N)$ solves~\eqref{eq:cell-pb-phi} at $x'$, with $c = \overline H(x',\phi)$, pointwise, this shows that $w(\cdot\,;x',\phi)$  is a subsolution of the cell problem at $x$, up to an error of order $|x-x'|^\alpha$:
\[
	\sup_{\theta \in \Theta} \big\{ -L_{x, y}^\theta w(y;x',\phi) - \phi(y,\theta) \big\} \ \leq \ \overline H(x', \phi) + A_0\, \| \phi \|_* \, |x - x'|^\alpha \qquad \text{in } \T^N.
\]
Since, in view of \eqref{eq:eff-H-inf-char}, $\overline H(x,\phi)$ is the smallest constant for which the ergodic problem at $x$ admits a subsolution we deduce $\overline H(x,\phi) - \overline H(x',\phi) \leq A_0 \|\phi\|_* |x-x'|^\alpha$, and, exchanging the roles of $x$ and $x'$,
\begin{equation}\label{eq:lemabase-c-holder}
	|\overline H(x, \phi) - \overline H(x', \phi)| \ \leq \ A_0\, \| \phi \|_* \, |x - x'|^\alpha \qquad \text{for all } \phi \in E \text{ and } x, x' \in \R^N.
\end{equation}

\smallskip
\emph{Step 2 (From $\overline H(\cdot,\phi)$ to $\A(x)$).} Since $\A(x)$ is convex and weak-$*$ compact, and $\mu \mapsto \overline\phi^\mu$ is linear and weak-$*$ continuous, $\langle\phi\rangle_x$ is a compact interval. By~\eqref{eq:eff-H-inf}, its endpoints are $\min_{\mu\in\A(x)}\overline\phi^\mu = - \overline H(x,\phi)$ and $\max_{\mu\in\A(x)}\overline\phi^\mu = -\min_{\mu\in\A(x)}\overline{(-\phi)}^\mu = \overline H(x,-\phi)$, that is,
\[
	\langle\phi\rangle_x \ = \ \big[-\overline H(x,\phi),\ \overline H(x,-\phi)\big], \qquad x \in \R^N.
\]
Hence\footnote{For real intervals, the Hausdorff distance is given explicitly by
\[
	d_H\big([a_1,b_1],[a_2,b_2]\big) \ = \ \max\big(|a_1-a_2|,\,|b_1-b_2|\big).
\]}
\[
	d_H(\langle\phi\rangle_{x},\langle\phi\rangle_{x'}) \ = \ \max\Big(|\overline H(x,\phi)-\overline H(x',\phi)|,\ |\overline H(x,-\phi)-\overline H(x',-\phi)|\Big).
\]
Applying~\eqref{eq:lemabase-c-holder} to $\phi$ and to $-\phi$, and using $\|-\phi\|_*=\|\phi\|_*$, both terms on the right-hand side are bounded by $A_0\|\phi\|_*|x-x'|^\alpha$, which proves the lemma.
\end{proof}

\subsection{Approximate correctors and the effective Hamiltonian}\label{subsec:properties}
In the study of the cell problem~\eqref{eq:cell-pb}, it is convenient to consider, for $\lambda> 0$, the equation
\begin{equation}\label{eq:dp}
	\lambda w_\lambda + \sup_{\theta \in \Theta} \big\{ -L^\theta_{x,y} w_\lambda(y) - \tilde\ell^\theta(y;x,\varphi) \big\} = 0\quad\text{in }\T^N,
\end{equation}
with $\tilde\ell^\theta(y;x,\varphi)$ as in~\eqref{eq:cell-pb}. We refer to~\eqref{eq:dp} as the \emph{approximate cell problem}, as is customary in homogenization~\cite{RPT23,ACS14}, or as the \emph{discounted cell problem}, as in ergodic theory and optimal control~\cite{IMT17}. We rely on the following comparison principle, which is standard in the theory of viscosity solutions.

\begin{lema}[Comparison with a regular competitor]\label{lem:comparison-regular}
Assume that $K$ satisfies~\ref{H0} and~\ref{H1}, and let $x \in \R^N$, $\varphi \in C_b^{\sigma+\iota}(\R^N)$, with $\iota>0$, and $\lambda>0$. Let $u \in C(\T^N)$ be a viscosity subsolution of~\eqref{eq:dp}, and let $v \in C^{\sigma+\gamma}(\T^N)$, for some $\gamma>0$, be a classical supersolution of~\eqref{eq:dp}. Then $u \le v$ in $\T^N$. The same holds with the roles of sub- and supersolution exchanged.
\end{lema}

We start with two technical lemmas on the regularity of $L^\theta_{x,y}\varphi(x)$, in the fast variable $y$ and in the slow variable $x$.

\begin{lema}[H\"older continuity of $L^\theta_{x,\cdot}\varphi(x)$ in $y$]\label{lem:Lxy-holder-y}
Assume that $K$ satisfies \ref{H0} and~\ref{H2} with $\alpha\in (0,1)$, and let $\varphi \in C_b^{\sigma+\iota}(\R^N)$, with $\iota\in(0,1)$ and $\sigma+\iota\notin\{1,2\}$.
Then, for all $x \in \R^N$ and $\theta \in \Theta$,
\[
	\big[L^\theta_{x,\cdot}\varphi(x)\big]_{C^\alpha(\T^N)} \ \le \ C K_0\,\|\varphi\|_{C^{\sigma+\iota}(\R^N)},
\]
for a constant $C > 0$ depending only on $N$, $\sigma$ and $\iota$.
\end{lema}

\begin{proof}
Let $x \in \R^N$, $\theta \in \Theta$ and $y,y' \in \T^N$. By~\ref{H2}, and Lemma~\ref{lem:near-far}, with $\iota$ in place of $\alpha$ and $r = 1$, we have
\begin{eqnarray*}
	\big|L^\theta_{x,y}\varphi(x) - L^\theta_{x,y'}\varphi(x)\big|
	 & \le & \frac12\int_{\R^N} |\delta^2\varphi(x,z)|\,\big| K^\theta(x,y,z) - K^\theta(x,y',z)\big|\,dz \\
	& \le & \frac{K_0}{2}\,|y-y'|^\alpha \int_{\R^N} |\delta^2\varphi(x,z)|\,|z|^{-(N+\sigma)}\,dz \\
	& \le & \frac{K_0}{2}\,|y-y'|^\alpha \Big( 2\|\varphi\|_{C^{\sigma+\iota}(\R^N)} \int_{B} |z|^{\iota - N}\,dz \\
	& & \qquad\qquad\qquad + \ 4\|\varphi\|_{L^\infty(\R^N)} \int_{B^c} |z|^{-(N+\sigma)}\,dz \Big) \\
	& \le & C\,K_0\,\|\varphi\|_{C^{\sigma+\iota}(\R^N)}\,|y-y'|^\alpha,
\end{eqnarray*}
with $C$ depending only on $N$, $\sigma$ and $\iota$, since $\iota>0$ and $\sigma>0$. This proves the claim.
\end{proof}

\begin{lema}[H\"older continuity of $L^\theta_{\cdot,y}\varphi(\cdot)$ in $x$]\label{lem:Lxy-holder-x}
Assume $K$ satisfies ~\ref{H0} and \ref{H2} with $\alpha\in(0,1)$, and let $\varphi \in C^{\sigma+\iota}_b(\R^N)$,
with $\iota\in(0,1)$ and $\sigma+\iota\notin\{1,2\}$. Then, for all $x\in\R^N,\ y\in\T^N,\ \theta\in\Theta$,
\[
	\big| L^\theta_{x,y}\varphi(x) \big| \ \le \ C\,\|\varphi\|_{C^{\sigma+\iota}(\R^N)} \qquad\]
and, for all $x,x' \in \R^N$, $y \in \T^N$, $\theta \in \Theta$,
\[
	\left| L^\theta_{x,y}\varphi(x) - L^\theta_{x',y}\varphi(x') \right|
	 \ \le \ C \, \|\varphi\|_{C^{\sigma+\iota}(\R^N)} \, |x-x'|^{\min(\iota,\alpha)}.
\]
for a constant $C > 0$ depending only on $N,\sigma,\alpha,\Lambda,K_0$ and $\iota$.
\end{lema}

\begin{proof}
\emph{The elementary bound.} By~\ref{H0} and Lemma~\ref{lem:near-far} (with $\iota$ in place of $\alpha$, $r=\delta$, and $\delta\in(0,1]$ fixed, e.g. $\delta=1$), if $\sigma + \iota <2$, then
\begin{eqnarray*}
	|L^\theta_{x,y}\varphi(x)|
	& \le &
	\frac\Lambda2 \left( \int_{B_\delta}|\delta^2\varphi(x,z)|\,|z|^{-(N+\sigma)}\,dz+
	\int_{B_\delta^c}|\delta^2\varphi(x,z)|\,|z|^{-(N+\sigma)}\,dz\right) \\
	& \le &
	\frac\Lambda 2 \left( 2[\varphi]_{C^{\sigma+\iota}(\R^N)}\int_{B_\delta} |z|^{\iota - N}\,dz+
	4 \| \varphi \|_\infty \int_{B_\delta^c} |z|^{-(N+\sigma)}\,dz\right) \\
	& \le & C(N, \sigma,\Lambda,\iota) \| \varphi \|_{C^{\sigma+\iota}(\R^N)}.
\end{eqnarray*}
In the case where $2\leq \sigma + \iota <3$, the first bound is replaced by
\[
\int_{B_\delta}|\delta^2\varphi(x,z)|\,|z|^{-(N+\sigma)}\,dz
\leq   \| D^2\varphi \| _{\infty} \int_{B_\delta} |z|^{ 2-N-\sigma}\,dz
\leq  C(N, \sigma) \| \varphi \|_{C^{\sigma+\iota}(\R^N)},
\]
and the same conclusion follows.

\emph{H\"older estimate of $L_{x,y}$.}
Writing $L^\theta_{z,y}\varphi(z)$ for $z \in \{x,x'\}$ in the symmetric form, we note that
\[
	L^\theta_{x,y}\varphi(x) - L^\theta_{x',y}\varphi(x') \ = \ J_1 + J_2,
\]
with
\begin{eqnarray*}
	&&J_1 := \frac12\int_{\R^N}\delta^2\varphi(x',z)\big[{ K^\theta(x,y,z)}-{ K^\theta(x',y,z)}\big]\,dz, \\
	&&J_2 := \frac12\int_{\R^N}\ \left(  \delta^2\varphi(x,z)-\delta^2\varphi(x',z)\right)\,{ K^\theta(x,y,z)}\,dz.
\end{eqnarray*}
By the same near/far-field computation as above, using~\ref{H2} and Lemma~\ref{lem:near-far},
\[
	|J_1| \ \le \ C \ \|\varphi\|_{C^{\sigma+\iota}(\R^N)}\,|x-x'|^\alpha.
\]
Set $d := |x-x'|$.
If $d>1/2$, then, by the elementary bound above,
\[
	\big|L^\theta_{x,y}\varphi(x)-L^\theta_{x',y}\varphi(x')\big| \ \le \
	2C\|\varphi\|_{C^{\sigma+\iota}(\R^N)} \ \le \
	2C\,2^{\min\{\iota,\alpha\}}\,\|\varphi\|_{C^{\sigma+\iota}(\R^N)}\,d^{\min\{\iota,\alpha\}},
\]
which proves the claim in this case, adjusting the constant.
Assume henceforth $d \le 1/2$, so that Lemma~\ref{lem:delta2-increments} applies with $\iota$ in place of $\alpha$.
By~\ref{H0},  we have
\[
	|J_2| \ \le \ \frac\Lambda2\int_{\R^N}| \delta^2\varphi(x,z)-\delta^2\varphi(x',z) |\,|z|^{-(N+\sigma)}\,dz,
\]
which we further split over $B_d$, the annulus $B_{1/2}\setminus B_d$, and $B_{1/2}^c$. Applying Lemma \ref{lem:delta2-increments} with $\varphi\in C^{\sigma+\iota}(\R^N)$, it is easy to see that
\[
	|J_2| \ \le \ C \,\|\varphi\|_{C^{\sigma+\iota}(\R^N)}\,d^\iota.
\]

\emph{Conclusion.} Combining both estimates yields
\[
	|L^\theta_{x,y}\varphi(x) - L^\theta_{x',y}\varphi(x')| \ \le \ |J_1|+|J_2| \ \le \ C\,\|\varphi\|_{C^{\sigma+\iota}(\R^N)}\,|x-x'|^{\min\{\iota,\alpha\}}.
\]
\end{proof}

We now turn to the approximate correctors themselves. Proposition~\ref{prop:approx-cell} below collects the quantitative estimates on $w_\lambda$, both in $\varphi$ and in $x$.

\begin{prop}[Properties of the approximate correctors]\label{prop:approx-cell}
Assume \ref{H0}-\ref{H2}, and \ref{H4}, \ref{H5} hold, with the same exponent $\alpha \in (0,\bar\alpha)$, where $\bar\alpha$ is given by Theorem~\ref{thm1} and $\sigma+\alpha\notin\N$, and let $\iota\in(0,1)$ with $\sigma+\iota\notin\{1,2\}$. Then, for every $x \in \R^N$, $\varphi \in C^{\sigma+\iota}_b(\R^N)$, and $\lambda > 0$,~\eqref{eq:dp} admits a unique periodic viscosity solution $w_\lambda(\cdot\,;x,\varphi) \in C^{\sigma+\alpha}(\T^N)$.
Moreover, there exists a constant $C >0$ depending on $N,\sigma,\Lambda_0,\Lambda,K_0,C_0,\alpha$ and $\iota$, such that:
\begin{enumerate}[label=\textup{(\roman*)}]
\item for every $x\in\R^N$, and $\varphi \in C^{\sigma+\iota}(\R^N)$,
	\[
	\lambda \|w_\lambda (\cdot; x,\varphi) \|_\infty
	\ \le \ C \Big(1 +  \|\varphi\|_{C^{\sigma+\iota}(\R^N)}\Big);
	\]
\item for every $x\in\R^N$ and $\varphi \in C^{\sigma+\iota}_b(\R^N)$,
	if $y_* \in \operatorname{argmin}_{\T^N} w_\lambda$,
	\[
	\|w_\lambda (\cdot; x,\varphi)  - w_\lambda(y_*; x,\varphi) \|_{C^{\sigma+\alpha}(\T^N)}
	\ \le \ C \Big(1 + \|\varphi\|_{C^{\sigma+\iota}(\R^N)}\Big);
	\]
\item for every $x\in\R^N$, and $\varphi_1,\varphi_2 \in C^{\sigma+\iota}(\R^N)$,
	\[
	\lambda \big\| w_\lambda(\cdot\,;x,\varphi_1) - w_\lambda(\cdot\,;x,\varphi_2) \big\|_\infty
	 \ \le \ C\, \|\varphi_1-\varphi_2\|_{C^{\sigma+\iota}(\R^N)};
	\]
\item for every $x,x' \in \R^N$ and $\varphi \in C^{\sigma+\iota}(\R^N)$,
	\[
	\lambda \big\| w_\lambda(\cdot\,;x,\varphi) - w_\lambda(\cdot\,;x',\varphi) \big\|_\infty
	 \ \le \ C\, |x-x'|^{\min\{\iota,\alpha\}} \Big(1+\|\varphi\|_{C^{\sigma+\iota}(\R^N)}\Big){ ;}
	\]
\item for every $x \in \R^N$, $\varphi \in C^{\sigma+\iota}(\R^N)$,
	\[
	\big\| \lambda w_\lambda( \cdot ;x,\varphi) + \overline H(x,\varphi) \big\|_\infty
	\le \ \lambda \, C \big(1 + \|\varphi\|_{C^{\sigma+\iota}(\R^N)}\big).
	\]
\end{enumerate}
\end{prop}

\begin{proof}
(i). The slow variable $x$ remains fixed throughout the statement and proof of the first part, so it is obtained exactly as in Lemma 3.1 in \cite{RPT23}. By Lemma~\ref{lem:Lxy-holder-x} and~\ref{H5},
\[
	\sup_{\theta} \|\tilde\ell^\theta(\cdot\,;x,\varphi)\|_\infty \ \le \
	C\big(1+\|\varphi\|_{C^{\sigma+\iota}(\R^N)}\big).
\]
Making use of the comparison principle given in Lemma~\ref{lem:comparison-regular}, with $v$ the constant barrier $\pm\lambda^{-1}C(1+\|\varphi\|_{C^{\sigma+\iota}(\R^N)})$, the result follows.
\smallskip

(ii).
For the second part, we aim to establish
\begin{equation}\label{dp_Linfty2}
	\|w_\lambda(\cdot; x, \varphi) - w_\lambda(y_*; x, \varphi)\|_\infty \leq C \left(1 + \|\varphi\|_{C^{\sigma+\iota}(\R^N)}\right),
\end{equation}
by contradiction, as in~\cite{RPT23}: the slow variable $x$ remains fixed throughout, and the argument takes place entirely in $y \in \T^N$. Suppose there exist $\lambda_k \to 0^+$, $x_k \in \R^N$, and $\varphi_k \in C^{\sigma+\iota}(\R^N)$ such that, writing $w_k := w_{\lambda_k}(\cdot\,;x_k,\varphi_k)$ and $y_k \in \operatorname{argmin}_{\T^N} w_k$,
\[
	\|w_k - w_k(y_k)\|_\infty \ \ge \ k\big(1+\|\varphi_k\|_{C^{\sigma+\iota}(\R^N)}\big).
\]
Set $\eta_k := \|w_k-w_k(y_k)\|_\infty^{-1}$ and $\widetilde w_k := \eta_k(w_k - w_k(y_k))$, so that $\widetilde w_k \ge 0$, $\widetilde w_k(y_k)=0$, and $\|\widetilde w_k\|_\infty=1$. Since $w_k$ solves~\eqref{eq:dp} and $L^\theta_{x_k,y}$ is linear and vanishes on constants, substituting $w_k = \eta_k^{-1}\widetilde w_k + w_k(y_k)$ into~\eqref{eq:dp} and multiplying through by $\eta_k>0$ shows that $\widetilde w_k$ solves
\[
	\lambda_k \widetilde w_k(y) + \lambda_k \eta_k w_k(y_k) + \sup_\theta \big\{ -L^\theta_{x_k,y} \widetilde w_k(y) - \eta_k \tilde\ell^\theta(y;x_k,\varphi_k) \big\} = 0 \quad \text{in } \T^N.
\]
By part (i), $\lambda_k |w_k(y_k)| \le C(1+\|\varphi_k\|_{C^{\sigma+\iota}(\R^N)})$, and $\|\tilde\ell^\theta(\cdot;x_k,\varphi_k)\|_\infty \le C(1+\|\varphi_k\|_{C^{\sigma+\iota}(\R^N)})$.
Since $\eta_k^{-1} \ge k(1+\|\varphi_k\|_{C^{\sigma+\iota}(\R^N)})$,
\[
	\lambda_k \eta_k |w_k(y_k)| \ \le \ \eta_k C\big(1+\|\varphi_k\|_{C^{\sigma+\iota}(\R^N)}\big) \ \le \ \frac{C}{k}, \qquad \eta_k \|\tilde\ell^\theta(\cdot;x_k,\varphi_k)\|_\infty \ \le \ \frac{C}{k},
\]
both $\to 0$ as $k\to\infty$; and $\lambda_k\widetilde w_k(y) \to 0$ uniformly as well, since $\lambda_k \to 0^+$ and $\|\widetilde w_k\|_\infty=1$. Hence
\[
	\mathcal M^+_{\mathcal L_0} \widetilde w_k \ \ge \ -o_k(1), \qquad \mathcal M^-_{\mathcal L_0} \widetilde w_k \ \le \ o_k(1) \qquad \text{in } \T^N,
\]
where $\mathcal M^\pm_{\mathcal L_0}$ are the extremal (Pucci-type) operators for the ellipticity class~$\mathcal L_0(\sigma,\Lambda_0,\Lambda)$. By the Harnack inequality of Caffarelli--Silvestre~\cite[Thm.~11.1]{CS09}, applied directly to the nonnegative function $\widetilde w_k$, for every $k$, there exists a constant $C>0$ depending only on $N,\sigma,\Lambda_0,\Lambda$, and in particular not on $k$, $x_k$, or $\lambda_k$, such that
\[
	1 \ = \ \sup_{\T^N} \widetilde w_k \ \le \ C\Big( \inf_{\T^N} \widetilde w_k + o_k(1) \Big) \ = \ C \, o_k(1),
\]
using $\inf_{\T^N}\widetilde w_k = 0$ (since $\widetilde w_k \ge 0$ and $\widetilde w_k(y_k)=0$). Since $o_k(1) \to 0$, this is a contradiction for $k$ large enough, proving~\eqref{dp_Linfty2}. \smallskip

Furthermore, by the H\"older estimate of~\cite{CS09}, there exists $\alpha_1\in(0,1)$, depending only on $N,\sigma,\Lambda_0,\Lambda$, such that
\[
	[w_\lambda - w_\lambda(y_*)]_{C^{\alpha_1}(\T^N)}
	\le C\Big(\|w_\lambda - w_\lambda(y_*)\|_\infty +
	\lambda\|w_\lambda\|_\infty + \sup_\theta \|\tilde\ell^\theta(\cdot\,;x,\varphi)\|_\infty\Big)
	 \le C\big(1+\|\varphi\|_{C^{\sigma+\iota}(\R^N)}\big).
\]
\smallskip

The regularity upgrade to $C^{\sigma+\alpha}(\T^N)$ now follows from the Schauder-type estimate of Serra~\cite{Serra15}. Given the base regularity $w_\lambda - w_\lambda(y_*) \in C^{\alpha_1}(\T^N)$ just established, and the H\"older continuity of the right-hand side $\tilde\ell^\theta(\cdot\,;x,\varphi)$ of~\eqref{eq:dp}, the estimate of~\cite{Serra15} gives $w_\lambda \in C^{\sigma+\alpha}(\T^N)$, with
\begin{align*}
	\|w_\lambda - w_\lambda(y_*)\|_{C^{\sigma+\alpha}(\T^N)}
	&\le C\Big( \|w_\lambda - w_\lambda(y_*)\|_{C^{\alpha_1}(\T^N)} +
	\lambda\|w_\lambda\|_\infty + \sup_\theta \|\tilde\ell^\theta(\cdot\,;x,\varphi)\|_{C^\alpha(\T^N)}\Big) \\
	&\le C\big(1+\|\varphi\|_{C^{\sigma+\iota}(\R^N)}\big).
\end{align*}
\smallskip

(iii). Write $w_i := w_\lambda(\cdot\,;x,\varphi_i)$, $i=1,2$, both solving~\eqref{eq:dp} with the \emph{same} frozen point $x$, and set
\[
	M \ := \ \lambda^{-1} C \|\varphi_1-\varphi_2\|_{C^{\sigma+\iota}(\R^N)},
\]
where $C>0$ is the same constant as in  Lemma~\ref{lem:Lxy-holder-x}, so that,
for every $y\in\T^N,\ \theta\in\Theta$,
\[
	\big|L^\theta_{x,y}(\varphi_1-\varphi_2)(x)\big| \ \le \
	C\|\varphi_1-\varphi_2\|_{C^{\sigma+\iota}(\R^N)}.
\]
It is easy to see that $v_\pm := w_1 \pm M$ are respectively a classical supersolution/subsolution of the equation solved by $w_2$. In view of Lemma~\ref{lem:comparison-regular} we deduce that $v_- \le w_2 \le v_+$ on $\T^N$, i.e.
\[
 	\|w_1-w_2\|_\infty \le M = \lambda^{-1}C\|\varphi_1-\varphi_2\|_{C^{\sigma+\iota}(\R^N)}.
\]

(iv). Write $w := w_\lambda(\cdot\,;x,\varphi)$, $w' := w_\lambda(\cdot\,;x',\varphi)$, both solving~\eqref{eq:dp} with the \emph{same} data $\varphi$ but at \emph{different} frozen points $x,x'$.
Since, by part (i), $\lambda\|w-w'\|_\infty \le 2C(1+\|\varphi\|_{C^{\sigma+\iota}(\R^N)})$, the claim is immediate if $|x-x'|>1$; we therefore assume $|x-x'|\le1$, so that $|x-x'|^\alpha\le|x-x'|^{\min(\iota,\alpha)}$. We first establish a preliminary bound, to cope with the differences in the two equations. In this case, both the operator itself and the right-hand side change between the two equations.

\emph{The operator's dependence on $w'$.} Note first that
\(
\delta^2 w'(y,z) = \delta^2\big(w'-w'(y_*)\big)(y,z),
\)
and therefore
\begin{eqnarray*}
	L^\theta_{x,y}w'(y) - L^\theta_{x',y}w'(y)
	& = & \frac12 \int_{\R^N} \delta^2 \big(w'-w'(y_*)\big)(y,z)\,\big[{ K^\theta(x,y,z)}-{ K^\theta(x',y,z)}\big]\,dz,
\end{eqnarray*}
so, by Lemma~ \ref{lem:near-far} and~\ref{H2},
\begin{eqnarray*}
	\big|L^\theta_{x,y}w'(y) - L^\theta_{x',y}w'(y)\big|
	& \le & \frac12\int_{\R^N} \big|\delta^2 \big(w'-w'(y_*)\big)(y,z)\big|\,K_0\,|x-x'|^\alpha\,|z|^{-(N+\sigma)}\,dz \\
	& \le & C\,\|w'-w'(y_*)\|_{C^{\sigma+\alpha}(\T^N)}\,|x-x'|^\alpha.
\end{eqnarray*}

\emph{The data term.} By~ \eqref{eq:cell-pb},
\[
	\tilde\ell^\theta(y;x,\varphi)-\tilde\ell^\theta(y;x',\varphi) \ = \ \big[\ell^\theta(x,y)-\ell^\theta(x',y)\big] \ + \ \big[L^\theta_{x,y}\varphi(x)-L^\theta_{x',y}\varphi(x')\big],
\]
bounded, respectively, by~\ref{H5} and by Lemma~\ref{lem:Lxy-holder-x}:
\begin{eqnarray*}
	\big|\tilde\ell^\theta(y;x,\varphi)-\tilde\ell^\theta(y;x',\varphi)\big|
	& \le & C_0\,|x-x'|^\alpha \ + \ C\|\varphi\|_{C^{\sigma+\iota}(\R^N)}\,|x-x'|^{\min(\iota,\alpha)}\\
	& \le & C \left( 1 + \|\varphi\|_{C^{\sigma+\iota}(\R^N)}\right) |x-x'|^{\min(\iota,\alpha)}.
\end{eqnarray*}

Set
\[
	M \ := \ \lambda^{-1}C\Big(1+\|\varphi\|_{C^{\sigma+\iota}(\R^N)}+
	\|w'-w'(y_*)\|_{C^{\sigma+\alpha}(\T^N)}\Big)|x-x'|^{\min(\iota,\alpha)},
\]
with $C$ as in the preliminary bound above, and $y_* \in \operatorname{argmin}_{\T^N}w'$ as in part (ii) above.
It is then easy to check that $v_\pm := w' \pm M$ are respectively a classical supersolution/subsolution of the equation solved by $w$. In view of Lemma~\ref{lem:comparison-regular} we deduce that $v_- \le w \le v_+$ on $\T^N$, i.e.
\[
 	\|w-w'\|_\infty \ \le \ M \ = \ \lambda^{-1}C\Big(1+\|\varphi\|_{C^{\sigma+\iota}(\R^N)}+
	\|w'-w'(y_*)\|_{C^{\sigma+\alpha}(\T^N)}\Big)|x-x'|^{\min(\iota,\alpha)}.
\]
Since $\|w'-w'(y_*)\|_{C^{\sigma+\alpha}(\T^N)} \le C\big(1+\|\varphi\|_{C^{\sigma+\iota}(\R^N)}\big)$ by part (ii) above, uniformly in $\lambda$, substituting into $M$ gives
\[
	\big\|w_\lambda(\cdot\,;x,\varphi) - w_\lambda(\cdot\,;x',\varphi)\big\|_\infty \ \le \ \lambda^{-1}C\big(1+\|\varphi\|_{C^{\sigma+\iota}(\R^N)}\big)\,|x-x'|^{\min(\iota,\alpha)}.
\]

(v). {By Theorem~\ref{thm:cell-phi}, applied to $\phi = \tilde\ell(\cdot\,;x,\varphi) \in E$, there exists a periodic corrector $w \in C^{\sigma+\alpha}(\T^N)$, unique up to additive constants, solving~\eqref{eq:cell-pb} with constant $\overline H(x,\varphi)$. By the same vanishing-discount construction underlying Theorem~\ref{thm:cell-phi}, $w$ is obtained as the locally uniform limit of $\widetilde w_{\lambda_n} = w_{\lambda_n}-w_{\lambda_n}(y_*)$, along a vanishing-discount sequence $\lambda_n \to 0^+$, so that}
\[
	\sup_\theta\big\{-L^\theta_{x,y}w(y)-\tilde\ell^\theta(y;x,\varphi)\big\} \ = \ \overline H(x,\varphi) \qquad \text{in } \T^N.
\]
Therefore, passing to the limit in the bound \eqref{dp_Linfty2} of part (ii), we get
\[
	\|w\|_\infty \ \le \ C\big(1+\|\varphi\|_{C^{\sigma+\iota}(\R^N)}\big) \ =: \ M.
\]

Since  $w$ solves~\eqref{eq:cell-pb} with constant $\overline H(x,\varphi)$, it follows that $v^\pm:= w - \overline H(x,\varphi)/\lambda \pm M$ are respectively super-/subsolutions of~\eqref{eq:dp}. Lemma~\ref{lem:comparison-regular} then gives $v^- \le w_\lambda \le v^+$, i.e., for every $y \in \T^N$,
\[
	\Big| w_\lambda(y;x,\varphi) - w(y) + \frac{\overline H(x,\varphi)}{\lambda} \Big| \ \le \ M.
\]
Multiplying the above by $\lambda$ further yields
\[
	\big| \lambda w_\lambda(y;x,\varphi) + \overline H(x,\varphi) \big|  \le
	\big|\lambda w_\lambda(y;x,\varphi)+\overline H(x,\varphi) - \lambda w(y)\big| + \lambda|w(y)|  \le
	 2\lambda M,
\]
which is the desired vanishing rate.
\end{proof}

\subsection{The rate of convergence}\label{subsec:rates}
We now prove Theorem~\ref{thm:rates-homog}. The overall strategy is the one of Camilli and Marchi~\cite{CM09}, as already adapted to the nonlocal, non-translation-invariant setting by Rodr\'iguez-Paredes and Topp~\cite{RPT23}. We first record the well-posedness of~{\eqref{eq:HJ-eps}} itself, for fixed $\epsilon\in(0,1)$, together with a bound independent of $\epsilon$. By~\cite[Proposition 2.5]{RPT23}, combined with the Schauder-type estimate of Serra~\cite{Serra15}, for every $\epsilon\in(0,1)$ there exists a unique solution $u^\epsilon \in C_b^{\sigma+\alpha}(\R^N)$ of~{\eqref{eq:HJ-eps}}, so that~{\eqref{eq:HJ-eps}} holds classically at every point. Moreover, there exists some ${ \alpha_2}\in(0,1)$ and a constant $M>0$, such that, uniformly in $\epsilon \in (0,1)$,
\begin{equation}\label{eq:rate-ueps-bound}
	\|u^\epsilon\|_{L^\infty(\R^N)} + [u^\epsilon]_{C^{{ \alpha_2}}(\R^N)} \ \le \ M.
\end{equation}
However, the full $C^{\sigma+\alpha}$ norm is \emph{not} uniform in $\epsilon$, since the fast oscillation of $K,\ell$ in $x/\epsilon$ prevents the Schauder constant from staying bounded as $\epsilon\to0$.

\begin{proof}
The proof is analogous to the earlier result of \cite{RPT23}, and we therefore present the argument in an abbreviated form, emphasizing the difficulties created by the extra $x$-dependence. The idea is to make the perturbed test function method of Evans~\cite{E89,E92} quantitative: rather than using the corrector $w$ itself, one uses instead approximate correctors $w_\lambda$, and compares $u^\epsilon$ directly against the perturbed test function $u + \epsilon^\sigma w_\lambda(\cdot/\epsilon\,;\cdot,u)$.

\smallskip
\emph{Average profile $u$.}
By Proposition~\ref{prop:existence-Csigma}, problem~\eqref{eq:HJ-bar} has a unique solution
$u \in C^{\sigma+\alpha}_b(\R^N)$, solving~\eqref{eq:HJ-bar} classically, with
\begin{equation}\label{eq:rate-M1}
	\|u\|_{C^{\sigma+\alpha}(\R^N)} \ \le \ \overline M
\end{equation}
for some $\overline M>0$ depending on the data.
This is an improvement with respect to~\cite{RPT23}, where the analogous interior estimate forced the choice of a strictly smaller auxiliary exponent.

\smallskip
\emph{Approximate correctors $w_\lambda$.}
Fix $\lambda >0$. Let $w_\lambda : = w_\lambda (y; x,u)$ be the approximate corrector associated to $u\in C_b^{\sigma+\alpha}(\R^N)$. In view of Proposition~\ref{prop:approx-cell} there exists a constant $C_1>0$, depending only on the data, such that
\begin{equation}\label{eq:rate-corrector-Linfty}
	\lambda \sup_{x\in\R^N} \| w_\lambda(\cdot; x,u) \|_\infty
	\le \ C_1 \left( 1 + \|u\|_{C^{\sigma+\alpha}(\R^N)} \right) \leq C_1(1+\overline M) =: C_2.
\end{equation}

\smallskip
\emph{Localization function $\psi$.}
Let $R :=  \ 2(M+\overline M) + 2C_2 +1$.
Consider a radial,  radially nondecreasing function $\psi \in C^{\infty}(\R^N)\cap L^\infty(\R^N)$ such that
	\begin{equation}\label{psi_func}
			\psi (x) = |x|^2 \  \mbox{ in } B, \quad
			\psi = R  \mbox{ in } \R^N \setminus B_2, \quad
			0 \leq \psi \leq R  \mbox{ in } B_2 \setminus B,
	\end{equation}
and set $\psi_\beta(x):=\psi(\beta x)$, $\beta\in(0,1)$. In view of assumption \ref{H0}, there exists  $M_1 =  M_1(N,\sigma,\Lambda,\psi)>0$ such that, for every $x\in\R^N,\ y\in\T^N,\ \theta\in\Theta$,
\[
	\big|L^\theta_{x,y}\psi (x)\big| \ \le \ M_1, \qquad
	\big|L^\theta_{x,y}\psi_\beta(x)\big| \ \le \ M_1\,\beta^\sigma.
\]

\smallskip
\emph{The perturbed test function.} The strategy is to build a perturbed test function $\varphi$ touching $u^\epsilon$ from above, evaluate~{\eqref{eq:HJ-eps}} classically at the touching point, and use property (v) of Proposition~\ref{prop:approx-cell} to make the effective Hamiltonian $\overline H$ appear, which then connects, via~\eqref{eq:HJ-bar}, to $u$ itself. The one subtlety is that the natural candidate, $\varphi^\epsilon(x)=u(x)+\epsilon^\sigma w_\lambda(x/\epsilon;x,u)$, is not itself regular enough to serve this role, as it carries the slow variable $x$ in the corrector's own frozen parameter. We therefore need to freeze the corrector's parameter at a point $\hat x$, and touch $u^\epsilon$ a second time, at a nearby point $\tilde x$, at the cost of tracking the small error between $\hat x$ and $\tilde x$.

As such, consider at first, for $x\in\R^N$,
\begin{equation*}
	\Phi(x) = u^\epsilon(x) - u(x) - \epsilon^\sigma w_\lambda(\frac{x}{\epsilon};x,u ) - \psi_\beta(x),
\end{equation*}
with $\lambda$ depending on $\epsilon$ so that $\epsilon^\sigma \lambda^{-1} \leq 1$.
By the choice of $R$ above, it is easy to see that the continuous function $\Phi$ attains its supremum at some $\hat x \in \overline{B_{2/\beta}}$.

\smallskip
Then, fix $\hat x$ as above. For $r\in(0,1)$ and $\kappa>0$ to be chosen, let
\[
	\varphi^\epsilon(x) \ := \ u(x) + \epsilon^\sigma w_\lambda\big(\frac{x}{\epsilon};\hat x,u\big) + \psi_\beta(x) + \kappa\,\psi(x-\hat x), \qquad x\in\R^N,
\]
and set $\tilde\Phi := u^\epsilon - \varphi^\epsilon$. The function $\varphi^\epsilon\in C_b^{\sigma+\alpha}(\R^N)$: its corrector term $w_\lambda(\cdot\,;\hat x,u)$ is now frozen at the fixed point $\hat x$, so the only remaining $x$-dependence is through the fast variable $x/\epsilon$. It is easy to see, in view of Proposition~\ref{prop:approx-cell}(iv) and~\eqref{eq:rate-corrector-Linfty}, that choosing
\begin{equation}\label{eq:kappa}
	\kappa \ := \ 2C_2\,r^{\alpha-2}\,\epsilon^\sigma\lambda^{-1},
\end{equation}
the continuous function $\tilde\Phi$ attains its supremum at some point $\tilde x\in B_r(\hat x)$.

\smallskip
From above, $u^\epsilon -\varphi^\epsilon$ has a global maximum at $\tilde x$. Since $\varphi^\epsilon\in C_b^{\sigma+\alpha}(\R^N)$, the viscosity inequality for $u^\epsilon$ at $\tilde x$ gives
\[
	u^\epsilon(\tilde x) \ \le \ -\sup_\theta\Big\{-L^\theta_{\tilde x,\frac{\tilde x}{\epsilon}}\varphi^\epsilon(\tilde x)
		- \ell^\theta(\tilde x,\frac{\tilde x}{\epsilon})\Big\}.
\]
By linearity of $L^\theta_{\tilde x,\frac{\tilde x}{\epsilon}}$,
\[
	L^\theta_{\tilde x,\frac{\tilde x}{\epsilon}}\varphi^\epsilon(\tilde x) \ = \
	L^\theta_{\tilde x,\frac{\tilde x}{\epsilon}}u(\tilde x)
	\ + \ \epsilon^\sigma L^\theta_{\tilde x,\frac{\tilde x}{\epsilon}}\big[w_\lambda(\cdot/\epsilon;\hat x,u)\big](\tilde x)
	\ + \ L^\theta_{\tilde x,\frac{\tilde x}{\epsilon}}\psi_\beta(\tilde x)
	\ + \ \kappa\,L^\theta_{\tilde x,\frac{\tilde x}{\epsilon}}\big[\psi(\cdot-\hat x)\big](\tilde x).
\]
The corrector term is related to $w_\lambda$ itself, rather than its rescaling, by the identity
\[
	\epsilon^\sigma L^\theta_{\tilde x,\frac{\tilde x}{\epsilon}}\big[w_\lambda(\cdot/\epsilon;\hat x,u)\big](\tilde x) \ = \ L^\theta_{\tilde x,\frac{\tilde x}{\epsilon}}\big[w_\lambda(\cdot\,;\hat x,u)\big]\Big(\frac{\tilde x}{\epsilon}\Big),
\]
a consequence of the scaling property~\ref{H3}, exactly as in the derivation of the ansatz~\eqref{ansatz} in Section~\ref{subsec:derivation}. The two localization terms are bounded uniformly in $\theta$:
\[
	\big|L^\theta_{\tilde x,\frac{\tilde x}{\epsilon}}\psi_\beta(\tilde x)\big| \ \le \ M_1\beta^\sigma, \qquad
	\big|L^\theta_{\tilde x,\frac{\tilde x}{\epsilon}}\big[\psi(\cdot-\hat x)\big](\tilde x)\big| \ \le \ M_1.
\]
This leads to
\begin{eqnarray*}
u^\epsilon(\tilde x)
& \le & -\sup_\theta\Big\{-L^\theta_{\tilde x,\frac{\tilde x}{\epsilon}}u(\tilde x)
	- L^\theta_{\tilde x,\frac{\tilde x}{\epsilon}}\big[w_\lambda(\cdot\,;\hat x,u)\big]\Big(\frac{\tilde x}{\epsilon}\Big)
	- \ell^\theta\left(\tilde x,\frac{\tilde x}{\epsilon}\right)\Big\} \ + \ M_1(\beta^\sigma+\kappa) \\
& \le & -\sup_\theta\Big\{-L^\theta_{\hat x,\frac{\tilde x}{\epsilon}}u(\hat x)
	- L^\theta_{\hat x,\frac{\tilde x}{\epsilon}}\big[w_\lambda(\cdot\,;\hat x,u)\big]\Big(\frac{\tilde x}{\epsilon}\Big)
	- \ell^\theta\left(\hat x,\frac{\tilde x}{\epsilon}\right)\Big\} \ + \ Cr^\alpha \ + \ M_1(\beta^\sigma+\kappa)\\
& = & \lambda\,w_\lambda(\frac{\tilde x}{\epsilon};\hat x,u) + \ Cr^\alpha \ + \ M_1(\beta^\sigma+\kappa),
\end{eqnarray*}
where we replaced $\tilde x$ by $\hat x$ at a cost $Cr^\alpha$, using $|\tilde x-\hat x|<r$ and three Hölder-in-$x$ estimates: Lemma~\ref{lem:Lxy-holder-x} for $u$, assumption~\ref{H5} for $\ell$, and, for $w_\lambda(\cdot\,;\hat x,u)$, the estimate opening the proof of Lemma~\ref{lem:A-reg} combined with Proposition~\ref{prop:approx-cell}(ii). We then recognize, in the last identity, the approximate cell problem satisfied by $\lambda\,w_\lambda(\tilde x/\epsilon;\hat x,u)$.

\smallskip
By Proposition~\ref{prop:approx-cell}(v), we further note that
\[
	\lambda\,w_\lambda\Big(\frac{\tilde x}{\epsilon};\hat x,u\Big) \ \le \ -\overline H(\hat x,u) \ + \ \lambda C_2.
\]
Since $u$ solves~\eqref{eq:HJ-bar} classically at $\hat x$, we have $\overline H(\hat x,u) = -u(\hat x)$.
Therefore we arrive at the key estimate:
\begin{equation}\label{eq:key-bound}
	u^\epsilon(\tilde x) - u(\hat x) \ \le \ C\big(\lambda + r^\alpha + \beta^\sigma + \kappa\big).
\end{equation}

\smallskip
\emph{Conclusion.} Since $\hat x$ maximizes $\Phi$ and $\tilde x$ maximizes $\tilde\Phi$ over all of $\R^N$, we have $\Phi(x)\le\Phi(\hat x)=\tilde\Phi(\hat x)\le\tilde\Phi(\tilde x)$ for every $x\in\R^N$. Unwinding the definitions of $\Phi,\tilde\Phi$, and discarding the nonpositive terms $-\psi_\beta(\tilde x)$, $-\kappa\psi(\tilde x-\hat x)$,
\[
	u^\epsilon(x) - u(x) \ \le \ \big[u^\epsilon(\tilde x)-u(\tilde x)\big] \ + \ \epsilon^\sigma\Big[w_\lambda\big(\frac{x}{\epsilon};x,u\big) - w_\lambda\big(\frac{\tilde x}{\epsilon};\hat x,u\big)\Big] \ + \ \psi_\beta(x).
\]
The middle term is bounded crudely, via~\eqref{eq:rate-corrector-Linfty} alone, by $2\epsilon^\sigma C_2/\lambda$. For the first, we employ that $u\in C^\alpha_b(\R^N)$ and $|\tilde x-\hat x|<r$ to get $u(\tilde x)-u(\hat x)\ge-\overline M r^\alpha$, so that, by~\eqref{eq:key-bound},
\[
	u^\epsilon(\tilde x)-u(\tilde x) \ = \ \big[u^\epsilon(\tilde x)-u(\hat x)\big] - \big[u(\tilde x)-u(\hat x)\big] \ \le \ C\big(\lambda+r^\alpha+\beta^\sigma+\kappa\big).
\]
Combining and substituting $\kappa$ from~\eqref{eq:kappa}, we obtain, for every $x\in\R^N$,
\[
	u^\epsilon(x) - u(x) \ \le \ C\Big(\lambda + r^\alpha + \beta^\sigma + \frac{\epsilon^\sigma}{\lambda} + \frac{\epsilon^\sigma}{\lambda\,r^{2-\alpha}}\Big) \ + \ \psi_\beta(x).
\]
Letting first $\beta\to0^+$, then choosing
\(
	\lambda :=  \epsilon^{{\sigma\alpha}/{(2+\alpha)}},
	 r  :=  \epsilon^{{\sigma}/{(2+\alpha)}},
\)
the terms $\lambda$, $r^\alpha$ and $\epsilon^\sigma\lambda^{-1}r^{\alpha-2}$ coincide and equal $\epsilon^{{\alpha_*}}$, while $\epsilon^\sigma/\lambda = \epsilon^{2\sigma/(2+\alpha)}\le\epsilon^{{\alpha_*}}$ (and in particular $\epsilon^\sigma\lambda^{-1}\le1$, as required), with
\(
	{\alpha_*} := {\sigma\alpha}/{(2+\alpha)} \in (0,\alpha).
\)
The reverse inequality is obtained symmetrically, exchanging the roles of $u^\epsilon$ and $u$ throughout, and we conclude
\[
	\|u^\epsilon-u\|_{L^\infty(\R^N)} \ \le \ C\,\epsilon^{{\alpha_*}}, \qquad {\alpha_*} = \frac{\sigma\alpha}{2+\alpha},
\]
which proves Theorem~\ref{thm:rates-homog}.

\end{proof}

\appendix \section{Appendix}

We record the following elementary bound on $\delta^2u$ at a single point, used repeatedly throughout the paper.

\begin{lema}\label{lem:near-far}
Let $\sigma+\alpha \in (0,3)\setminus\{1,2\}$, $\nu := \lfloor \sigma+\alpha \rfloor$, $x \in \R^N$ and $r>0$, and let $u \in C^{\sigma+\alpha}(B_r(x)) \cap L^\infty(\R^N)$. Then
\[
	|\delta^2u(x,z)| \ \le \
	\begin{cases}
		2\,[u]_{C^{\sigma+\alpha}(B_r(x))} \, |z|^{\sigma+\alpha}, & |z| < r, \ \nu \in \{0,1\}, \\[4pt]
		\|D^2u\|_{L^\infty(B_r(x))} \, |z|^{2}, & |z| < r, \ \nu = 2, \\[4pt]
		4\,\|u\|_{L^\infty(\R^N)}, & |z| \geq r.
	\end{cases}
\]
\end{lema}

\begin{proof}
For $|z|\ge r$, we only use the boundedness of $u$:
\[
	|\delta^2u(x,z)| \ \le \ |u(x+z)| + |u(x-z)| + 2|u(x)| \ \le \ 4\,\|u\|_{L^\infty(\R^N)}.
\]
Let now $|z|<r$, so that $x \pm tz \in B_r(x)$ for all $t\in[0,1]$.

\emph{Case $\nu = 0$.} Writing $\delta^2u(x,z) = \big[u(x+z)-u(x)\big] - \big[u(x)-u(x-z)\big]$, we get
\[
	|\delta^2u(x,z)| \ \le \ |u(x+z)-u(x)| + |u(x)-u(x-z)| \ \le \ 2\,[u]_{C^{\sigma+\alpha}(B_r(x))}\,|z|^{\sigma+\alpha}.
\]

\emph{Case $\nu = 1$.} Write $\sigma+\alpha = 1+\gamma$ with $\gamma \in (0,1)$, so that $[u]_{C^{\sigma+\alpha}(B_r(x))} = [Du]_{C^{\gamma}(B_r(x))}$. By the fundamental theorem of calculus,
\[
	\delta^2u(x,z) \ = \ \int_0^1 \big[Du(x+tz) - Du(x-tz)\big]\cdot z \; dt,
\]
and since $|Du(x+tz)-Du(x-tz)| \le [Du]_{C^{\gamma}(B_r(x))}\,|2tz|^{\gamma}$, we get
\[
	|\delta^2u(x,z)| \ \le \ 2^{\gamma}\,[Du]_{C^{\gamma}(B_r(x))}\,|z|^{1+\gamma} \int_0^1 t^{\gamma}\,dt
	\ \le \  2\, [u]_{C^{\sigma+\alpha}(B_r(x))}\,|z|^{\sigma+\alpha}.
\]

\emph{Case $\nu = 2$.} By Taylor's formula at second order,
\[
	\delta^2u(x,z) \ = \ \int_0^1 \int_{-t}^{t} \big\langle D^2u(x+sz)\, z, z \big\rangle \, ds\, dt,
\]
so that
\[
|\delta^2u(x,z)| \le \|D^2u\|_{L^\infty(B_r(x))}\,|z|^2 \int_0^1 2t\,dt = \|D^2u\|_{L^\infty(B_r(x))}\,|z|^2.
\]
\end{proof}

{While Lemma~\ref{lem:near-far} bounds $\delta^2u$ at a single point $x$, the Hölder-in-$x$ estimates  require bounding the \emph{difference} $\delta^2u(x,z)-\delta^2u(x',z)$ between two nearby points $x,x'$. This is the content of the following lemma.}

\begin{lema}\label{lem:delta2-increments}
Let $\sigma+\alpha \in (0,3)\setminus\{1,2\}$, set $\nu := \lfloor \sigma+\alpha \rfloor \in \{0,1,2\}$, and let
$u \in C^{\sigma+\alpha}(B) \cap C_b^\alpha(\R^N)$. Define
$\beta = \min\{\sigma+\alpha,2\}$ and $\beta' = \min\{\sigma+\alpha,1\}$.
Then, for all $x, x' \in B_{1/2}$ with $d := |x-x'| \leq 1/2$,
\begin{equation*}
|\delta^2u(x, z) - \delta^2u(x', z)| \ \leq \
\begin{cases}
4\, [u]_{C^{\sigma + \alpha}(B)}\, d^{\sigma+\alpha-\beta}\, |z|^{\beta} & \text{if } |z| \leq d, \\[1mm]
4\, [u]_{C^{\sigma + \alpha}(B)}\, d^{\beta'}\, |z|^{\sigma + \alpha - \beta'} & \text{if } d < |z| \leq 1/2, \\[1mm]
4\, \|u\|_{C^\alpha(\R^N)}\, d^\alpha & \text{if } |z| \geq 1/2.
\end{cases}
\end{equation*}
\end{lema}

\begin{proof}
In the first two cases $|z|\le 1/2$ (in the first one because $|z|\le d\le 1/2$), so that all the
points $x\pm tz$, $x'\pm tz$, $t\in[-1,1]$, lie in $B$.

\emph{Case $|z| \le d$.} If $\nu \in \{0,1\}$, then $\beta=\sigma+\alpha$, and the bound
$4[u]_{C^{\sigma+\alpha}(B)}|z|^{\sigma+\alpha}$ follows by applying Lemma~\ref{lem:near-far}
(with $r = 1/2$) at $x$ and at $x'$, and using the triangle inequality.
If $\nu = 2$, then $\beta = 2$; write $\gamma := \sigma+\alpha-2\in(0,1)$. By Taylor's formula at second order,
\[
	\delta^2u(x,z)-\delta^2u(x',z)
	 =  \int_0^1\!\!\int_{-t}^{t} \big\langle \big(D^2u(x+sz)-D^2u(x'+sz)\big) z, z \big\rangle \, ds\, dt,
\]
so that
\[
	|\delta^2u(x,z)-\delta^2u(x',z)| \le [D^2u]_{C^\gamma(B)}\, d^\gamma\,|z|^2
	= [u]_{C^{\sigma+\alpha}(B)}\, d^{\sigma+\alpha-2}|z|^2.
\]

\emph{Case $d < |z| \le 1/2$.} If $\nu = 0$, then $\beta'=\sigma+\alpha$ and
\[
	|\delta^2u(x,z)-\delta^2u(x',z)|
	\le  |u(x+z)-u(x'+z)| + |u(x-z)-u(x'-z)| + 2|u(x)-u(x')|
	\le 4[u]_{C^{\sigma+\alpha}(B)}\, d^{\sigma+\alpha}.
\]
If $\nu \in \{1,2\}$, then $\beta'=1$. Set $w(y) := \delta^2u(y,z)$ for $y\in B_{1/2}$. Then
$Dw(y) = \delta^2(Du)(y,z)$, and Lemma~\ref{lem:near-far}, applied to $Du \in C^{\sigma+\alpha-1}(B)$, gives
\[
	|Dw(y)| \le 2[u]_{C^{\sigma+\alpha}(B)}\,|z|^{\sigma+\alpha-1}.
\]
By the mean value theorem along the segment $[x',x]\subset B_{1/2}$,
\[
	|\delta^2u(x,z)-\delta^2u(x',z)| \ = \ |w(x)-w(x')| \ \le \ 2[u]_{C^{\sigma+\alpha}(B)}\,|z|^{\sigma+\alpha-1}\, d.
\]

\emph{Case $|z| \ge 1/2$.} Here the points $x\pm z$, $x'\pm z$ need not lie in $B$, but the same
three-term estimate as in the case $\nu=0$ above, with the $C^\alpha$ seminorm of $u$ on $\R^N$ in place
of the $C^{\sigma+\alpha}$ seminorm on $B$, gives
\[
	|\delta^2u(x,z)-\delta^2u(x',z)| \le 4[u]_{C^\alpha(\R^N)}\, d^\alpha \le 4\|u\|_{C^\alpha(\R^N)}\, d^\alpha.
\]
\end{proof}

\bigskip\subsection*{Acknowledgments}

A. Ciomaga is partially supported by the ANR (Agence Nationale de la Recherche) through the COSS project ANR-22-CE40-0010 and by the Research Grant GAR2023 - code 73 - from the Donors' Recurrent Fund, at the disposal of the Romanian Academy and managed through the ``PATRIMONIU'' Foundation and by CNPq Grant 408169/2023-0.
A. Rodriguez-Paredes was partially supported by Fondecyt Grant Iniciaci\'on No. 11231065.
E. Topp is partially supported by CNPq Grant 306022/2023-0, CNPq Grant 408169/2023-0, FAPERJ APQ1 Grant 210.573/2024 and FAPERJ Grant E-26/204.215/2025.

\bibliographystyle{plain}
\bibliography{references}
\end{document}